\documentclass[12pt]{article}
\usepackage{amsmath}
\usepackage{latexsym}
\usepackage{amssymb}
\newtheorem{thm}{Theorem}[section]
\newtheorem{la}[thm]{Lemma}
\newtheorem{Defn}[thm]{Definition}
\newtheorem{Remark}[thm]{Remark}
\newtheorem{Conj}[thm]{Conjecture}
\newtheorem{prop}[thm]{Proposition}
\newtheorem{cor}[thm]{Corollary}
\newtheorem{Example}[thm]{Example}
\newtheorem{Number}[thm]{\!\!}
\newenvironment{defn}{\begin{Defn}\rm}{\end{Defn}}
\newenvironment{example}{\begin{Example}\rm}{\end{Example}}
\newenvironment{rem}{\begin{Remark}\rm}{\end{Remark}}

\newenvironment{numba}{\begin{Number}\rm}{\end{Number}}
\newenvironment{proof}{{\noindent\bf Proof.}}%
                  {\nopagebreak\hspace*{\fill}$\Box$\medskip\par}
\newcommand{\Punkt}{\nopagebreak\hspace*{\fill}$\Box$}
\newcommand{\wb}{\overline}
\newcommand{\ve}{\varepsilon}
\newcommand{\at}{\symbol{'100}}
\newcommand{\wt}{\widetilde}
\newcommand{\impl}{\Rightarrow}
\newcommand{\mto}{\mapsto}

\newcommand{\isom}{\cong}
\DeclareMathOperator{\Ad}{Ad}
\newcommand{\N}{{\mathbb N}}
\newcommand{\B}{{\mathbb B}}
\newcommand{\K}{{\mathbb K}}
\newcommand{\R}{{\mathbb R}}
\newcommand{\C}{{\mathbb C}}
\newcommand{\bT}{{\mathbb T}}
\newcommand{\bO}{{\mathbb O}}
\newcommand{\Q}{{\mathbb Q}}
\newcommand{\Z}{{\mathbb Z}}
\newcommand{\cT}{{\mathcal T}}
\newcommand{\bL}{{\mathbb L}}
\newcommand{\F}{{\mathbb F}}
\newcommand{\cg}{{\mathfrak g}}

\newcommand{\cV}{{\mathcal V}}
\newcommand{\cA}{{\mathcal A}}
\newcommand{\cL}{{\mathcal L}}

\DeclareMathOperator{\Aut}{Aut}
\newcommand{\sub}{\subseteq}
\DeclareMathOperator{\GL}{GL}
\DeclareMathOperator{\SL}{SL}
\DeclareMathOperator{\id}{id}

\newcommand{\sbull}{{\scriptscriptstyle \bullet}}
\newcommand{\aeq}{\Leftrightarrow}
\DeclareMathOperator{\Hom}{Hom}

\DeclareMathOperator{\dv}{div}

\DeclareMathOperator{\tor}{tor}

\DeclareMathOperator{\para}{par}
\DeclareMathOperator{\cont}{con}
\DeclareMathOperator{\lev}{lev}

\newcommand{\parp}{\stackrel{\longrightarrow}{\para}\!}
\newcommand{\parm}{\stackrel{\longleftarrow}{\para}\!}
\newcommand{\conp}{\stackrel{\longrightarrow}{\cont}\!}
\newcommand{\conm}{\stackrel{\longleftarrow}{\cont}\!}
\DeclareMathOperator{\COS}{COS}
\DeclareMathOperator{\car}{char}
\DeclareMathOperator{\dt}{det}
\DeclareMathOperator{\End}{End}
\DeclareMathOperator{\Iso}{Iso}
\DeclareMathOperator{\op}{op}
\DeclareMathOperator{\bik}{bik}
\DeclareMathOperator{\nub}{nub}
\DeclareMathOperator{\tp}{top}
\DeclareMathOperator{\cts}{cts}

\begin{document}
\begin{center}
{\Large\bf Endomorphisms of Lie groups over local fields}\\[7mm]
{\bf Helge Gl\"{o}ckner}\vspace{2mm}
\end{center}
\begin{abstract}
\hspace*{-6mm}Lie groups over totally disconnected local fields
furnish prime examples
of totally disconnected, locally compact groups. We discuss the scale,
tidy subgroups and further subgroups (like contraction subgroups) for analytic
endomorphisms of such groups.\vspace{2mm}
\end{abstract}
{\bf Classification:} primary 22E20;
secondary 22E35, 22E46, 22E50, 37D10, 37P10, 37P20\\[3mm]
{\bf Key words:} Lie group; local field; $p$-adic Lie group; endomorphism; Willis theory; scale;
scale function; tidy subgroup; minimizing subgroup; Levi\linebreak
factor; Levi subgroup; contraction group; parabolic subgroup;
big cell;\linebreak
invariant subgroup;
dynamical system; stable manifold; stable foliation
\section{\!\!Introduction}\label{sec-intro}
\noindent
The scale $s(\alpha)\in \N$ of an automorphism (or endomorphism)
$\alpha$ of a totally disconnected locally compact group
was introduced in the works of George Willis (see \cite{Wil}, \cite{FUR}, \cite{END}),
and ample information on the concept can be found
in his lecture notes in this proceedings volume~\cite{SCA}.
Following \cite{FUR} and \cite{END}, the scale $s(\alpha)$ can be
defined as the minimum of the indices\footnote{If we wish to emphasize the underlying group~$G$,
we write $s_G(\alpha)$ instead of~$s(\alpha)$.}
\[
[\alpha(U):\alpha(U)\cap U],
\]
for $U$ ranging through the set $\COS(G)$ of compact
open subgroups of~$G$. Compact open subgroups
for which the minimum is attained are called \emph{minimizing};
as shown in~\cite{FUR} and \cite{END},
they can be characterized by certain `tidyness' properties,
and therefore coincide with the so-called \emph{tidy subgroups}
for~$\alpha$ (the definition of which is recalled in Section~\ref{sec-basic}).\\[2.3mm]
Besides the tidy subgroups, further subgroups of~$G$
have been associated to~$\alpha$ which proved to be useful
for the study of~$\alpha$, and for the structure theory
in general (see \cite{BaW} and \cite{END}).
We mention the \emph{contraction group}
\[
\conp(\alpha):=\Big\{x\in G\colon\lim_{n\to\infty}\alpha^n(x)=e\Big\}
\]
and the \emph{parabolic subgroup} $\parp(\alpha)$ of all
$x\in G$ whose $\alpha$-orbit $(\alpha^n(x))_{n\in\N_0}$
is \emph{bounded} in the sense that $\{\alpha^n(x)\colon n\in\N_0\}$
is relatively compact in~$G$.
It is also interesting to consider
group elements $x\in G$ admitting an
\emph{$\alpha$-regressive trajectory}
$(x_{-n})_{n\in\N_0}$ of group elements~$x_{-n}$
such that $x_0=x$ and $\alpha(x_{-n-1})=x_{-n}$ for all~$n$.
Setting $x_n:=\alpha^n(x)$ for $n\in \N$,
we then obtain a so-called \emph{two-sided $\alpha$-orbit}
$(x_n)_{n\in\Z}$ for~$x$.
The \emph{anti-contraction group} $\conm(\alpha)$
is defined as the group of all $x\in G$ admitting an $\alpha$-regressive
trajectory
$(x_{-n})_{n\in\N_0}$ such that
\[
\lim_{n\to\infty}x_{-n}=e;
\]
the \emph{anti-parabolic subgroup} $\parm(\alpha)$ is the
group of all $x\in G$ admitting a bounded $\alpha$-regressive
trajectory. The intersection
\[
\lev(\alpha):=\; \parp(\alpha)\, \cap\parm(\alpha)
\]
is called the \emph{Levi subgroup} of~$\alpha$;
it is the group of all $x\in G$ admitting a bounded two-sided $\alpha$-orbit
(see \cite{BaW} and \cite{END} for these concepts, which were inspired by
terminology in the theory of linear algebraic groups).\\[2.3mm]
In this work, we consider Lie groups over totally disconnected
local fields, like the field of $p$-adic numbers
or fields of formal Laurent series over a finite field
(see Sections~\ref{sec-basic} and \ref{sec-mfd} for
these concepts). The topological group
underlying such a Lie group~$G$ is a totally disconnected locally compact
group, and the analytic endomorphisms $\alpha\colon G\to G$ we consider
are, in particular, continuous endomorphisms of~$G$.\\[2.3mm]
Our goal is twofold:
On the one hand, we strive to complement the lecture
notes~\cite{SCA} by adding a detailed
discussion of one class of examples, the Lie groups
over totally disconnected local fields, which illustrates the general theory.
On the other hand, most of the text can be considered
as a research article, as it contains results which are new
(or new in the current generality),
and which are proved here in full.
We also recall necessary concepts concerning Lie groups over totally disconnected local fields,
as far as required for the purpose.
Compare~\cite{Lec} for a broader (but more sketchy) introduction with a similar thrust,
confined to the study of automorphisms.
For further information on Lie groups
over totally disconnected local fields, see \cite{Ser} and the references therein,
also~\cite{Sny} and~\cite{Bou}.\footnote{Contrary
to our conventions, the Lie groups in \cite{Bou}
are modelled on Banach spaces which need not be of finite dimension.}
Every $p$-adic Lie group has a compact open subgroup
which is an analytic pro-$p$-group;
see \cite{Dix}, \cite{Di2}, and \cite{Sny} for the theory
of such groups, and Lazard's seminal work~\cite{Laz}.
For related studies in positive characteristic, cf.\ \cite{KaZ}
and subsequent studies.\\[2.3mm]
Every group of $\K$-rational points of a linear algebraic group defined over a
totally disconnected local field~$\K$
can be considered as a Lie group over~$\K$ (see \cite[Chapter~I, Proposition~2.5.2]{Mar}).
We refer to \cite{Bor}, \cite{Hum}, \cite{Mar}, and \cite{Spr}
for further information on such groups, which can be studied
with tools from algebraic geometry, and via actions on buildings
(see~\cite{BaT} and later work).\\[2.3mm]
The Lie groups we consider need not be algebraic groups,
they are merely $\K$-analytic manifolds.
Yet, compared to
general totally disconnected groups,
we have additional structure at our disposal:
Every Lie group $G$ over a totally disconnected local field~$\K$ has a Lie algebra
$L(G)$ (its tangent space $T_e(G)$ at the neutral element $e\in G$),
which is a finite-dimensional $\K$-vector space.
If $\alpha\colon G\to G$ is a $\K$-analytic endomorphism,
then its tangent map $L(\alpha):=T_e(\alpha)$ at~$e$ is a linear endomorphism
\[
L(\alpha)\colon L(G)\to L(G)
\]
of the $\K$-vector space~$L(G)$.
It is now natural to ask how the scale and tidy subgroups for~$\alpha$
are related to those of~$L(\alpha)$.
Guided by this question, 
we describe tidy subgroups and calculate the scale
for linear endomorphisms of finite-dimensional $\K$-vector spaces
(which also provides a first illustration of the abstract concepts),
see Theorem~\ref{thm-lincase}.
For analytic endomorphisms $\alpha\colon G\to G$ of a Lie group~$G$
over a totally disconnected local field~$\K$, we shall see that
\begin{equation}\label{glob-infi}
s(\alpha)=s(L(\alpha))
\end{equation}
if and only if the contraction group $\conp(\alpha)$
is closed in~$G$ (see Theorem~\ref{main}, the main result),
which is always the case if $\car(\K)=0$
(by Corollary~\ref{isclosed}).
If $\conp(\alpha)$ is closed,
then
\begin{equation}\label{expli-yet}
s(L(\alpha))
=\prod_{\stackrel{{\scriptstyle j\in\{1,\ldots, m\}}}{{\rm s.t.}\;|\lambda_j|_\K\geq 1}}
|\lambda_j|_\K\vspace{-1mm}
\end{equation}
in terms of the eigenvalues $\lambda_1,\ldots,\lambda_m$
of $L(\alpha)\otimes_\K\id_{\wb{\K}}$ in an algebraic closure~$\wb{\K}$
of~$\K$, where $|.|_\K$ is the unique extension of the
`natural' absolute value on~$\K$ specified in~(\ref{dffrmnatu})
to an absolute value on~$\wb{\K}$
(see Theorem~\ref{thm-lincase}).\\[2.3mm]
The text is organized as follows.\\[2.3mm]
After a preparatory Section~\ref{sec-basic}
on background concerning totally disconnected
groups and totally disconnected local fields,
we study linear endomorphisms
of finite-dimensional $\K$-vector spaces (Section~\ref{sec-vector}).\\[2.3mm]
In Section~\ref{sec-mfd},
we recall elementary definitions and facts concerning
$\K$-analytic functions, manifolds,
and Lie groups.
We then construct well-behaved compact open subgroups in
Lie groups over totally disconnected local fields (see Section~\ref{constsmall}).\\[2.3mm]
In Section~\ref{sec-padic},
we calculate the scale (and determine tidy subgroups)
for $\alpha$ an endomorphism of a $p$-adic Lie group~$G$.
This is simplified by the fact that every $p$-adic
Lie group has an exponential function, which provides
an analytic local conjugacy between the dynamical systems
$(L(G),L(\alpha))$ and $(G,\alpha)$ around the fixed points~$0$ (resp., $e$).\\[2.3mm]
By contrast, $\K$-analytic endomorphisms $\alpha\colon G\to G$
cannot be linearized in general if~$G$ is a Lie group
over a local field of positive characteristic (see \cite[7.3]{Lec} for a counterexample).
As a replacement for a linearization,
we use (locally) invariant manifolds (viz.\ local stable, local unstable, and centre manifolds)
around the fixed point~$e$ of the time-discrete, analytic dynamical
system $(G,\alpha)$. As shown in
\cite{Exp} and \cite{Fin}, the latter can be constructed as
in the classical real case (cf.\ \cite{Irw} and \cite{Wel}).
The necessary definitions
and facts are compiled in Section~\ref{sec-inv-mfd}.\\[2.3mm]
The following section contains the main results,
notably a calculation of the scale for analytic
endomorphisms $\alpha\colon G\to G$ of
a Lie group~$G$ over a totally disconnected local field, if $\conp(\alpha)$ is closed
(see Theorem~\ref{main}).
We also show that if $\conp(\alpha)$ is closed, then $\conp(\alpha)$,
$\lev(\alpha)$, and $\conm(\alpha)$ are
Lie subgroups of~$G$ and the map
\[
\conp(\alpha)\times\lev(\alpha)\times \conm(\alpha)\to\;\conp(\alpha)\lev(\alpha)\conm(\alpha)=:\Omega
\]
taking $(a,b,c)$ to $abc$ has open image~$\Omega$ and is an analytic diffeomorphism
(Theorem~\ref{bigcell}).\\[2mm]
The last  three sections are devoted to automorphisms
with specific properties.
An automorphism $\alpha\colon G\to G$
of a totally disconnected, locally compact group~$G$
is called \emph{contractive} if $G=\;\conp(\alpha)$, i.e.,
\[
\lim_{n\to\infty}\alpha^n(x)=e\quad\mbox{for all $\,x\in G$}
\]
(see \cite{Sie} and \cite{SIM}).
If
\[
\bigcap_{n\in\Z}\alpha^n(V)=\{e\}
\]
for some identity neighbourhood $V\sub G$,
then $\alpha$ is called \emph{expansive} (see~\cite{GaR}),
or also \emph{of finite depth} in the case of compact~$G$ (see \cite{NUB}).
If
\[
e\not\in \overline{\{\alpha^n(x)\colon n\in\Z\}}
\]
for each $x\in G\setminus\{e\}$, then~$\alpha$ is called
a \emph{distal} automorphism (cf.\ \cite{Raj}, \cite{RaS}).
Every contractive automorphism is expansive (see, e.g., \cite{GaR}).\\[2.3mm]
If $G$ is a Lie group over a totally disconnected local field~$\K$ with algebraic closure~$\overline{\K}$
and $\alpha\colon G\to G$ an analytic automorphism,
then $\conp(\alpha)$ is open in~$G$ (resp., $\alpha$ is expansive, resp., $\alpha$ is distal)
if and only if
\[
|\lambda|_\K<1
\]
(resp., $|\lambda|_\K\not=1$, resp., $|\lambda|_\K=1$)
for each eigenvalue $\lambda$ of the $\overline{\K}$-linear
automorphism $L(\alpha)\otimes_\K \id_{\overline{\K}}$
of $L(G)\otimes_\K\overline{\K}$ obtained by extension
of scalars (see Proposition~\ref{charviaeigen} for details).\\[2.3mm]
Recall that every continuous homomorphism between
$p$-adic Lie groups is analytic, whence the Lie group
structure on a $p$-adic Lie group is uniquely determined
by the underlying topological group (see, e.g., \cite{Bou}).
Lazard~\cite{Laz} characterized $p$-adic Lie groups
within the class of all totally disconnected, locally compact
compact groups, and later many further characterizations
were found (see \cite{Dix}).\\[2.3mm]
Recent research showed that
$p$-adic Lie groups are among basic building blocks for general
totally disconnected groups in various situations,
e.g.\ in the study of ergodic $\Z^n$-actions on locally compact groups by automorphisms
(see \cite{DSW})
and also in the theory of contraction groups (see \cite{SIM}).
In both cases, Lazard's theory of analytic pro-$p$-groups
was invoked to show that the groups in contention are $p$-adic Lie groups.
Section~\ref{sec-contr}
surveys results concerning
contractive automorphisms.
We give an alternative, new argument for the appearance
of $p$-adic Lie groups, using the structure theory
of locally compact abelian groups (i.e., Pontryagin duality)
instead of the theory of analytic pro-$p$-groups.\\[2.3mm]
Section~\ref{sec-exp}
briefly surveys results concerning expansive
automorphisms.\\[2.3mm]
The final section is devoted to
distal automorphisms and Lie groups
of type~$R$; we
prove
a criterion for pro-discreteness
(Theorem~\ref{typeR}) which had been
announced in~\cite[Proposition~11.3]{Lec}.\\[2.3mm]
Further papers have been written on the foundation of~\cite{END}:
Analogues of results from~\cite{BaW}, \cite{Jaw}, and \cite{FUR}
for endomorphisms of totally disconnected, locally compact groups
were developed in~\cite{BGT};
the topological entropy $h_{\tp}(\alpha)$ of an endomorphism~$\alpha$
of a totally disconnected, locally compact group~$G$
was studied in~\cite{GBV}. It was shown there that
\begin{equation}\label{hvss}
h_{\tp}(\alpha)=\ln s(\alpha)
\end{equation}
if and only if the so-called nub subgroup $\nub(\alpha)$ of~$\alpha$ (as in \cite{END})
is trival (see \cite[Corollary~4.11]{GBV});
the latter holds if and only if $\conp(\alpha)$ is closed (see \cite[Theorem~D]{BGT}).
In the current paper, we can do with the results from~\cite{END}
and give Lie-theoretic proofs
for results which can be generalized further (see~\cite{BGT}),
by more involved arguments.\footnote{Notably,
we have the Inverse Function Theorem
at our disposal.} The results were obtained before those of~\cite{BGT},
and presented in the author's minicourse June 27--July 1, 2016
at\linebreak
the MATRIX workshop and in a talk at the AMSI workshop July 25, 2016.\footnote{Except
for
results concerning the scale on subgroups and quotients (Proposition~\ref{subquot})
and the endomorphism case of Lemma~\ref{crit-clo},
which were added in 2017. In the talks, I also confined myself to a proof of the equivalence
of (a) and (b) in Theorem~\ref{main} when $\alpha$ is an automorphism,
which is easier
(while the theorem was stated in full).}
For complementary studies of endomorphisms of pro-finite groups,
see~\cite{Rei}.\\[2.3mm]
{\bf Acknowledgements.}
The author is grateful for the support provided by the
University of Melbourne (Matrix Center, Creswick)
and the University of Newcastle (NSW), notably George A. Willis,
which enabled participation in the `Winter of Disconnectedness.'
A former unpublished manuscript concerning the scale of automorphisms
dating back to 2006 was supported by DFG grant 447 AUS-113/22/0-1
and ARC grant LX 0349209.\\[2.3mm]
{\bf Conventions.}
We write $\N:=\{1,2,\ldots\}$ $\N_0:=\N\cup\{0\}$, and $\Z:=\N_0\cup (-\N)$.
Endomorphisms of topological groups are assumed continuous;
automorphisms of topological groups are assumed continuous, with
continuous inverse. If we call a mapping~$f$ an analytic diffeomorphism
(or an analytic automorphism),
then also $f^{-1}$ is assumed analytic. If~$E$ is a vector space over a field~$\K$,
we write $\End_\K(E)$ for the $\K$-algebra
of all $\K$-linear endomorphisms of~$E$, and $\GL(E):=\End_\K(E)^\times$
for its group of invertible elements.
If~$\bL$ is a field
extension of~$\K$, we let $E_\bL:=E\otimes_\K\bL$ be the $\bL$-vector space
obtained by extension of scalars. We identify~$E$ with $E\otimes 1\sub E_\bL$
as usual. Given $\alpha\in\End_\K(E)$, we let $\alpha_\bL:=\alpha\otimes\id_\bL$
be the endomorphism of~$E_\bL$ obtained by extension of scalars.
If $\wb{\K}$ is an algebraic closure of~$\K$,
we shall refer to the eigenvalues $\lambda\in\wb{\K}$ of $\alpha_{\wb{\K}}$
simply as the \emph{eigenvalues of~$\alpha$ in~$\wb{\K}$}.
Given $n\in\N$, we write $M_n(\K)$ for the $\K$-algebra of $n\times n$-matrices.
If $f\colon X\to X$ is a self-map of a set~$X$,
we say that a subset $Y\sub X$ is \emph{$\alpha$-stable} if $\alpha(Y)=Y$.
If $\alpha(Y)\sub Y$, then~$Y$ is called \emph{$\alpha$-invariant.}
If $X$ is a set, $Y\sub X$ a subset, $f\colon Y\to X$ a map and $x\in Y$,
we say that a sequence $(x_{-n})_{n\in\N_0}$ of elements $x_n\in Y$
is an \emph{$f$-regressive trajectory for~$x$}
if $f(x_{-n-1})=x_{-n}$ for all $n\in\N_0$ and $x_0=x$.
In this situation, we also say that~$x$ \emph{admits} the $f$-regressive
trajectory~$(x_n)_{n\in\N_0}$. If, instead, $f$~is defined on
a larger subset of~$X$ which contains~$Y$
but all $x_n$ are elements of~$Y$, we call $(x_{-n})_{n\in\N_0}$
an $f$-regressive trajectory \emph{in~$Y$}.
\section{\!\!Some basics of totally disconnected groups}\label{sec-basic}
In this section, we recall basic definitions
and facts concerning totally disconnected groups
and local fields.\vspace{-2mm}
\subsection*{{\normalsize The module of an automorphism}}
\noindent
Let $G$ be a locally compact group
and $\B(G)$ be the $\sigma$-algebra
of Borel subsets of~$G$.
We let $\lambda_G\colon \B(G)\to [0,\infty]$
be a Haar measure on~$G$, i.e.,
a non-zero Radon measure which is left invariant
in the sense that $\lambda_G(gA)=\lambda_G(A)$ for all
$g\in G$ and $A\in \B(G)$.
It is well-known that a Haar measure always exists,
and that it is unique up to multiplication with a positive
real number\linebreak
(cf.\ \cite{HaR}).
If $\alpha\colon G\to G$ is an
automorphism, then also
\[
\B(G)\to [0,\infty],\qquad
A\mto \lambda_G(\alpha(A))
\]
is a left invariant non-zero Radon measure
on~$G$ and hence a multiple of Haar measure:
There exists $\Delta(\alpha)>0$ such that
$\lambda_G(\alpha(A))=\Delta(\alpha)\lambda_G(A)$
for all $A\in \B(G)$.
If $U\sub G$ is a relatively compact,
open, non-empty subset,
then
\begin{equation}\label{calcmodul}
\Delta(\alpha)=\frac{\lambda_G(\alpha(U))}{\lambda_G(U)}
\end{equation}
(cf.\ \cite[(15.26)]{HaR}, where however
the conventions differ). We also write $\Delta_G(\alpha)$
instead of $\Delta(\alpha)$, if we wish to emphasize
the underlying group~$G$.
\begin{rem}\label{indcoop}
Let $U$ be a compact open
subgroup of $G$. If $U\sub \alpha(U)$,
with index $[\alpha(U):U]=:n$,
we can pick representatives $g_1,\ldots, g_n\in \alpha(U)$
for the left cosets of $U$ in~$\alpha(U)$.
Exploiting the left invariance  of Haar measure,
(\ref{calcmodul})
turns into
\begin{equation}\label{howgtmod}
\Delta(\alpha)\, =\, \frac{\lambda_G(\alpha(U))}{\lambda_G(U)}
\, =\, \sum_{j=1}^n\frac{\lambda_G(g_jU)}{\lambda_G(U)}
\, =\, [\alpha(U):U]\,.
\end{equation}
If $\alpha(U)\sub U$, applying (\ref{howgtmod})
to $\alpha^{-1}$ instead of~$\alpha$
and $\alpha(U)$ instead of~$U$, we obtain
\begin{equation}\label{shrinkindex}
\Delta(\alpha^{-1})
\, =\, [U:\alpha(U)]\,.
\end{equation}
\end{rem}
\subsection*{{\normalsize Tidy subgroups and the scale}}
\noindent
If $G$ is a totally disconnected,
locally compact group, $\alpha\colon G\to G$
an endomorphism and $U$ a compact open subgroup of~$G$, following \cite{END} we write
\[
U_-:=\bigcap_{n\in\N_0}\alpha^{-n}(U)=\{x\in U\colon (\forall n\in\N_0)\;\alpha^n(x)\in U\},\vspace{-1mm}
\]
where $\alpha^{-n}(U)$ means the preimage $(\alpha^n)^{-1}(U)$.
Let $U_+$ be the set of all $x\in U$ admitting an
$\alpha$-regressive trajectory $(x_{-n})_{n\in \N_0}$ in~$U$.
Then
\[
U_+=\bigcap_{n\in\N_0}U_n \quad\mbox{with}\vspace{-2mm}
\]
\begin{equation}\label{defUn}
U_0:=U\;\;\mbox{and}\;\; U_{n+1}:=U\cap\alpha(U_n)\;\;\mbox{for $n\in\N_0$;}\\[1mm]
\end{equation}
moreover, $U_+$ and $U_-$ are compact subgroups of~$G$
such that
\[
\alpha(U_-)\sub U_-\quad\mbox{and}\quad \alpha(U_+)\supseteq U_+
\]
(see \cite{END}).
The sets
\[
U_{--}:=\bigcup_{n\in\N_0}\alpha^{-n}(U_-)\quad\mbox{and}\quad
U_{++}:=\bigcup_{n\in\N_0} \alpha^n(U_+)
\]
are unions of ascending sequences of subgoups, whence they are subgroups of~$G$.
\begin{numba}\label{with-label}
If we wish to emphasize which endomorphism~$\alpha$ is considered,
we write $U_{n,\alpha}$, $U_{+,\alpha}$, and $U_{-,\alpha}$
instead of $U_n$, $U_+$, and $U_-$, respectively.
\end{numba}
The following definition was given in~\cite{END}.
\begin{defn}
If $U=U_+U_-$, then $U$ is called \emph{tidy above} for~$\alpha$.
If $U_{++}$ is closed in~$G$
and the indices
\[
[\alpha^{n+1}(U_+):\alpha^n(U_+)]\in\N
\]
are independent of $n\in\N_0$,
then $U$ is called \emph{tidy below} for~$\alpha$.
If $U$ is both tidy above and tidy below for~$\alpha$, then $U$ is called \emph{tidy for $\alpha$}.
\end{defn}
\noindent
The following fact (see \cite[Proposition~9]{END}) is useful for our ends:
\begin{numba}\label{variant-tidy}
\emph{$U$ is tidy for~$\alpha$ if and only if~$U$ is tidy above
and $U_{--}$ is closed in~$G$.}
\end{numba}
\begin{numba}\label{mini-tidy}
As shown in~\cite{END}, a compact open subgroup
$U$ of~$G$ is minimizing for~$\alpha$ (as defined in the Introduction)
if and only if it is tidy for~$\alpha$, in which case
\[
s(\alpha)=[\alpha(U):\alpha(U)\cap U]=[\alpha(U_+) : U_+].
\]
\end{numba}
\begin{numba}\label{uplusauto}
If $\alpha$ is an automorphism of~$G$, then simply (as in \cite{FUR})
\[
U_+=\bigcap_{n\in\N_0}\alpha^n(U).\vspace{-2mm}
\]
\end{numba}
Let us consider some easy special cases (which will be useful later).
\begin{la}\label{nilp}
Let $\alpha$ be an endomorphism of a totally disconnected, locally compact group~$G$.
\begin{itemize}
\item[\rm(a)]
If $V\sub G$ is a compact open subgroup such that  $\alpha(V)\sub V$,
then $V$ is tidy, $V_-=V$ and $s(\alpha)=1$.
\item[\rm(b)]
If $V\sub G$ is a compact open subgroup with $V\subseteq \alpha(V)$,
then $V$ is tidy above for~$\alpha$. If, moreover, $V$ is tidy,
then $V_+=V$ and $s(\alpha)=\Delta(\alpha)$.
\item[\rm(c)]
If $\alpha$ is nilpotent $($say $\alpha^n=e)$
and $U\sub G$ a compact open subgroup, then
\begin{equation}\label{two-eq}
V:=U_-=\bigcap_{k=0}^\infty \alpha^{-k}(U)=\bigcap_{k=0}^{n-1}\alpha^{-k}(U)
\end{equation}
is a compact open subgroup of~$G$ with $\alpha(V)\sub V$.
\end{itemize}
\end{la}
\begin{proof}
(a) Since $V\sub\alpha^{-1}(V)$, we have $V\sub \alpha^{-k}(V)$ for all $k\in\N_0$ and thus
\[
V_-=\bigcap_{k=0}^\infty\alpha^{-k}(V)=V.
\]
Hence $V=V_+V_-$ is tidy above. As the subgroup
\[
V_{--}=\bigcup_{k=0}^\infty\alpha^{-k}(V_-)=\bigcup_{k=0}^\infty\alpha^{-k}(V)
\]
contains $V$, it is open and hence closed. Thus~$V$ is tidy for~$\alpha$, by~\ref{variant-tidy}.
Finally $\alpha(V)\sub V$ entails that
$s(\alpha)=[\alpha(V):\underbrace{\alpha(V)\cap V}_{=\alpha(V)}]=1$.

(b) Since $V\sub\alpha(V)$, every $x\in V$ has an $\alpha$-regressve trajectory
within~$V$, whence $x\in V_+$. Hence $V=V_+$, and thus $V=V_+V_-$ is tidy above.
If $V$ is tidy, then $s(\alpha)=[\alpha(V):\alpha(V)\cap V]=[\alpha(V):V]=\Delta(\alpha)$,
using~(\ref{howgtmod}).

(c) For integers $k\geq n$, we have $\alpha^k(x)=e\in U$ for all $x\in G$,
whence $x\in\alpha^{-k}(U)$ and thus $\alpha^{-k}(U)=G$.
This entails the second equality in~(\ref{two-eq}), and so $V$ is compact and
open. As $\alpha(U_-)\sub U_-$, the final inclusion holds.
\end{proof}
\begin{numba}
If $G$ is a totally disconnected, locally compact group
and $g\in G$, let
\[
I_g\colon G\to G,\quad x\mto gxg^{-1}
\]
be the corresponding inner automorphism of~$G$.
Given $g\in G$, abbreviate $s(g):=s(I_g)$.
Following \cite{Wil}, the mapping $s\colon G\to\N$
so obtained is called the \emph{scale function} on~$G$.
\end{numba}
\subsection*{{\normalsize Local fields}}\label{locfds}
\noindent
Basic information on
totally disconnected local fields can be found in many books,
e.g.\ \cite{Wei} and~\cite{Jac}.\\[2.3mm]
By a \emph{totally disconnected local field}, we mean
a totally disconnected, locally compact, non-discrete
topological field~$\K$.\\[2.3mm]
Each totally disconnected local field $\K$ admits an \emph{ultrametric
absolute value} $|.|$ defining
its topology, i.e.,
\begin{itemize}
\item[(a)]
$|t|\geq 0$ for each $t\in \K$, with equality if and only if
$t=0$;
\item[(b)]
$|st|=|s|\cdot |t|$ for all $s,t\in \K$;
\item[(c)]
The \emph{ultrametric inequality} holds, i.e.,
$|s+t|\leq \max\{|s|,|t|\}$ for all
$s,t\in \K$.
\end{itemize}
\noindent
An example of an absolute value defining the topology of~$\K$
is what we call the \emph{natural absolute
value} on~$\K$, given by $|0|_\K:=0$ and
\begin{equation}\label{dffrmnatu}
|x|_\K\; :=\; \Delta_\K(m_x)
\quad \mbox{for $\,x\in \K\setminus\{0\}$}
\end{equation}
(cf.\ \cite[Chapter~II, \S2]{Wei}),
where $m_x\colon \K\to\K$, $y\mto xy$
is scalar multiplication by~$x$ and $\Delta_\K(m_x)$
its module.\footnote{Note that if $\K$ is an extension of $\Q_p$
of degree~$d$, then $|p|_\K=p^{-d}$
depends on the extension.}\\[2.3mm]
It is known that every totally disconnected local field~$\K$
either is
a field of formal Laurent
series over some finite field
(if $\car(\K)>0$),
or a finite extension
of the field of $p$-adic numbers for some
prime~$p$ (if $\car(\K)=0$).
Let us fix our notation concerning these
basic examples.
\begin{example}
Given a prime number $p$,
the field $\Q_p$ of $p$-adic
numbers is the completion
of $\Q$ with respect to
the $p$-adic absolute value,
\[
\left| p^k \frac{n}{m}\right|_p\;:=\;p^{-k}\quad
\mbox{for $k\in \Z$ and $n,m\in \Z\setminus p\Z$.}
\]
We use the same notation,
$|.|_p$, for the extension of the $p$-adic
absolute value to~$\Q_p$.
Then the topology coming from $|.|_p$ makes
$\Q_p$ a totally disconnected local field,
and $|.|_p$ is the natural absolute value
on~$\Q_p$.
Every non-zero element $x$ in $\Q_p$ can be written
uniquely in the form
\[
x\;=\; \sum_{k=n}^\infty a_k\, p^k
\]
with $n\in \Z$, $a_k\in \{0,1,\ldots, p-1\}$
and $a_n\not=0$.
Then $|x|_p=p^{-n}$.
The elements
of the form $\sum_{k=0}^\infty a_kp^k$
form the subring
$\Z_p=\{x\in \Q_p \colon |x|_p\leq 1\}$
of~$\Q_p$, which is open
and also compact,
because it is homeomorphic to
$\{0,1,\ldots, p-1\}^{\N_0}$
via $\sum_{k=0}^\infty a_kp^k\mto (a_k)_{k\in \N_0}$.
\end{example}
\begin{example}
Given a finite field~$\F$ (with~$q$ elements),
we let $\F(\!(X)\!)\sub \F^\Z$
be the field of formal Laurent
series
$\sum_{k=n}^\infty a_kX^k$
with $a_k\in \F$ and $n\in\Z$. 
Here addition is
pointwise, and multiplication
is given by the Cauchy product.
We endow $\F(\!(X)\!)$
with the topology arising from the ultrametric absolute value
\begin{equation}\label{nather}
\left|\sum_{k=n}^\infty a_kX^k\right|
\; :=\; q^{-n}\quad\mbox{if $\,a_n\not=0$.}
\end{equation}
Then the set $\F[\hspace*{-.15mm}[X]\hspace*{-.17mm}]$
of formal power series
$\sum_{k=0}^\infty a_kX^k$
is a compact and open subring of
$\F(\!(X)\!)$, and thus
$\F(\!(X)\!)$
is a totally disconnected local field.
Its natural absolute value
is given by~(\ref{nather}).
\end{example}
\noindent
Beyond local fields,
we also
consider some \emph{ultrametric fields}
$(\K,|.|)$.
Thus~$\K$ is a field
and $|.|$ an ultrametric absolute value
on~$\K$ which defines a non-discrete
topology on~$\K$.
For example, we shall repeatedly
use an algebraic closure~$\wb{\K}$ of a totally disconnected local field~$\K$
and exploit that an ultrametric absolute value~$|.|$
on~$\K$ extends uniquely
to an ultrametric absolute value on~$\wb{\K}$
(see, e.g., \cite[Theorem~16.1]{Sch}).
The same notation, $|.|$,
will be used for the extended absolute value.
An ultrametric field $(\K,|.|)$
is called \emph{complete} if $\K$ is a complete
metric space with respect to the
metric given by $d(x,y):=|x-y|$.
\subsection*{{\normalsize Ultrametric norms and balls}}\label{bscconsultra}
Let $(\K,|.|)$ be an ultrametric field
and $(E,\|.\|)$ be a normed $\K$-vector
space whose norm is ultrametric in the sense that
$\|x+y\|\leq \max\{\|x\|,\|y\| \}$ for all
$x,y\in E$.
Since $\|x\|=\|x+y-y\| \leq \max\{\|x+y\|,\|y\|\}$,
it follows that $\|x+y\|\geq \|x\|$ if $\|y\|<\|x\|$
and hence
\begin{equation}\label{winner}
\|x+y\|\;=\;\|x\|\quad \mbox{for all $x,y\in E$ such that
$\|y\|<\|x\|$.}
\end{equation}
We shall use the notations
\begin{eqnarray*}
B_r^E(x) &:=& \{y\in E\colon \|y-x\|<r\}\;\;\mbox{and}\\
\wb{B}^E_r(x)&:= &\{y\in E\colon\|y-x\|\leq r\}
\end{eqnarray*}
for balls in~$E$ (with $x\in E$, $r\in \;]0,\infty[$).
The ultrametric inequality
entails
that $B_r^E(0)$
and $\wb{B}^E_r(0)$
are
\emph{subgroups} of $(E,+)$
with non-empty interior
(and hence both open and closed).
Specializing to $E=\K$, we see that
\begin{equation}\label{dfnvalr}
\bO\; :=\; \{z\in \K\colon |z|\leq 1\}
\end{equation}
is an open subring of~$\K$,
its so-called \emph{valuation ring}.
If~$\K$ is a totally disconnected local field,
then~$\bO$ is a compact
subring of~$\K$ (which is maximal
and hence independent of the choice of
absolute value). In this case, also the unit group
\[
\bO^\times=\{z\in \bO\colon |z|=1\}
\]
of all invertible elements is compact,
as it is closed in~$\bO$.\\[2.3mm]
An \emph{ultrametric Banach space}
over a complete ultrameric field
is a normed space $(E,\|.\|)$
over~$\K$, with ultrametric norm~$\|.\|$,
such that every Cauchy sequence in~$E$ is convergent.
We shall always endow a finite-dimensional
vector space~$E$ over a complete ultrametric field $(\K,|.|)$
with the unique Hausdorff topology making
it a topological $\K$-vector space (see
Theorem~2 in
\cite[Chapter~I, \S2, no.\,3]{BTV}).
Then $E\cong \K^m$ (carrying the product topology)
as a topological $\K$-vector space, with $m:=\dim_\K(E)$,
entailing that there exists a norm $\|.\|$ on $E$ (corresponding to
the maximum norm on $\K^m$) which defines its topology
and makes it an ultrametric Banach space.
If~$(E.\|.\|_E)$ and~$(F,\|.\|_F)$ are finite-dimensional
normed spaces over a complete ultrametric field, then every linear map $\alpha\colon E\to F$
is continuous (see Corollary~2 in \cite[Chapter~I, \S2, no.\,3]{BTV});
as usual, we write
\[
\|\alpha\|_{\op}:=\sup \left\{\frac{\|\alpha(x)\|_F}{\|x\|_E} \colon x\in E\setminus\{0\}\right\}
\in [0,\infty[
\]
for its operator norm. Then $\|\alpha(x)\|_F\leq\|\alpha\|_{\op}\|x\|_E$ and, if~$\alpha$
is invertible and $E\not=\{0\}$, then
\[
\|\alpha(x)\|_F\, \geq \, \frac{1}{\|\alpha^{-1}\|_{\op}\!\!\!\!}\,\|x\|_E\quad\mbox{for all $x\in E$.}
\]
\subsection*{{\normalsize Module of a linear automorphism}}\label{calcind}
We recall a formula for the module of a linear automorphism.
\begin{la}\label{laindx}
Let $E$ be a finite-dimensional vector space over a totally\linebreak
disconnected local field~$\K$
and~$\alpha\in \GL(E)$.
Then
\begin{equation}\label{modvia}
\Delta_E(\alpha)\;=\; |\hspace*{-.4mm}\dt \alpha\hspace*{.4mm}|_\K
\;=\;\prod_{i=1}^n |\lambda_i|_\K,
\end{equation}
where $\lambda_1,\ldots,\lambda_n$
are the eigenvalues of~$\alpha$
in an algebraic closure $\wb{\K}$
of~$\K$.
\end{la}
\begin{proof}
See \cite[Proposition 55 in Chapter III, \S3, no.\,16]{Bou}
for the first equality in~(\ref{modvia}).
The second equality in~(\ref{modvia})
is clear if all eigenvalues lie in~$\K$.
For the general case, pick
a finite extension~$\bL$ of~$\K$ containing
the eigenvalues,
and let $d:=[\bL:\K]$
be the degree of the field extension.
Then $\Delta_{E_\bL}(\alpha_\bL)=(\Delta_E(\alpha))^d$.
Since the extended
absolute value is given
by
\[
|x|_\K\; =\; \sqrt[d]{\Delta_{\bL}(m_x)}\quad\mbox{for $\,x\in\bL\setminus\{0\}$}
\]
(see \cite[Chapter 9, Theorem~9.8]{Jac}
or \cite[Exercise 15.E]{Sch}),
the desired equality follows
from the special case
already treated (applied now to~$\bL$).
\end{proof}
\section{Endomorphisms of {\boldmath$\K$-}vector spaces}\label{sec-vector}
Linear endomorphisms of vector spaces over totally disconnected local fields
provide first examples of endomorphisms
of totally disconnected groups,
and their understanding is essential also for our discussion
of endomorphisms of Lie groups.\\[2.3mm]
Throughout this section, $\K$ is a totally disconnected local field, $E$ a finite-dimensional $\K$-vector space
and $\alpha\in\End_\K(E)$.
We shall calculate the scale, determine the parabolic, Levi and
contraction subgroups for~$\alpha$,
and find tidy subgroups.\\[2.3mm]
Our starting point are ideas from \cite[Chapter II, \S1]{Mar}
concerning iteration of linear endomorphisms.
Following~\cite{Mar},
we shall decompose~$E$ into
certain characteristic subspaces,
which help us to understand the dynamics of~$\alpha$.
\begin{numba}\label{ifsplits}
If the characteristic polynomial $p_\alpha$ of $\alpha\in\End_\K(E)$ splits into
linear factors in the polynomial ring $\K[X]$, then~$E$ is the direct sum of the generalized
eigenspaces for~$\alpha$.
For $\rho\in[0,\infty[$, let
\[
E_\rho
\]
be the sum of all generalized eigenspaces for eigenvalues $\lambda\in\K$ with $|\lambda|_\K=\rho$;
we call $E_\rho$ the \emph{characteristic subspace} for~$\rho$.
By construction,
\begin{equation}\label{decomposes}
E=\bigoplus_{\rho\geq 0} E_\rho.
\end{equation}
\end{numba}
\begin{numba}\label{splitornot}
If $\alpha\in\End_\K(E)$ is arbitrary, we choose a finite extension field $\bL$ of~$\K$
such that $p_\alpha$ splits into linear factors in $\bL[X]$.
By \ref{ifsplits}, we have a decomposition
\[
E_\bL=\bigoplus_{\rho\geq 0} (E_\bL)_\rho
\]
into characteristic subspaces for~$\alpha_\bL$.
We call
\[
E_\rho:=(E_\bL)_\rho\cap E
\]
the \emph{characteristic subspace} of~$E$ for~$\rho$.
If $E_\rho\not=\{0\}$, then $\rho$ is called a \emph{characteristic value} of~$\alpha$.
Using the Galois Criterion, it can
be shown that each $(E_\bL)_\rho$ is defined over~$\K$, i.e.,
\[
(E_\bL)_\rho=(E_\rho)_\bL
\]
(see \cite[Chapter II, (1.0)]{Mar}).
As a consequence, again (\ref{decomposes}) holds.\footnote{As $E_\rho\not=\{0\}$ for only finitely many $\rho\geq 0$,
we can identify the direct sum $E=\bigoplus_{\rho\geq 0}E_\rho$
with the direct product $\prod_{\rho\geq 0}E_\rho$ whenever this is convenient.}
\end{numba}
\begin{rem}
(a) By construction, $\alpha(E_\rho)\sub E_\rho$ for each $\rho\geq 0$.\\[2mm]
(b) $E_0=\bigcup_{n\in\N_0}\ker(\alpha^n)$ is the generalized eigenspace
for the eigenvalue~$0$
(also known as the ``Fitting $0$-component''),
and thus $\alpha|_{E_0}$ is a nilpotent endomorphism.\\[2mm]
(c) For each $\rho>0$, the restriction $\alpha|_{E_\rho}\colon E_\rho\to E_\rho$
is an injective endomorphism of a finite-dimensional vector space
and hence an automorphism.\\[2mm]
(d) The restriction of $\alpha$ to the ``Fitting $1$-component"
$E_{>0}:=\bigoplus_{\rho>0}E_\rho$ is an automorphism, and
\[
E=E_0\oplus E_{>0}.
\]
Thus
\begin{equation}\label{twofactors}
s(\alpha)=s(\alpha|_{E_0})s(\alpha|_{E_{>0}}).
\end{equation}
\end{rem}
Since $s(\alpha|_{E_0})=1$ by Lemma~\ref{nilp} (c) and (a),
we deduce from (\ref{twofactors}) that $s(\alpha)=s(\alpha|_{E_{>0}})$.
\begin{prop}\label{scale-on-Fitt}
The scale $s(\alpha)$ of $\alpha\in\End_\K(E)$
coincides with the scale $s(\alpha|_{E_{>0}})$
of the automorphism of the Fitting $1$-component induced by~$\alpha$.\,\Punkt
\end{prop}
For endomorphisms of $p$-adic vector spaces,
this was already observed in~\cite{Rat}.\\[2.3mm]
We now recall from \cite[Proposition~2.4]{Fin}
the existence of norms which are well-adapted to
an endomorphism (see already \cite[Proposition~4.3]{Lec}
for the case of automorphisms; compare
\cite[Chapter~II, Lemma~1.1]{Mar}
for a similar, weaker result, which also applies to $\R$ and $\C$).
\begin{la}\label{fact-adapted}
There exists
an ultrametric norm $\|.\|$ on~$E$ which is \emph{adapted to $\alpha$}
in the following sense:
\begin{itemize}
\item[\rm(a)]
$\|.\|$ is a maximum norm
with respect to the decomposition~\emph{(\ref{decomposes})}
of~$E$ into the characteristic subspaces for~$\alpha$;
\item[\rm(b)]
$\|\alpha|_{E_0}\|_{\op}<1$;
and
\item[\rm(c)]
For all $\rho>0$ and $v\in E_\rho$, we have $\|\alpha(v)\|=\rho\|v\|$.
\end{itemize}
If $\ve\in\;]0,1]$ is given,
then $\|.\|$ can be chosen such that $\|\alpha|_{E_0}\|_{\op}<\ve$.\,\Punkt
\end{la}
As before, in the following theorem we write $E_\rho$ for
the characteristic subspace for $\rho>0$
with respect to~$\alpha$.
\begin{thm}\label{thm-lincase}
Let $\alpha\in\End_\K(E)$ be an endomorphism of a finite-dimensional vector space~$E\cong \K^m$
over a local field~$\K$. 
Let $\|.\|$ be a norm on~$E$ which is adapted to~$\alpha$.
Then the following holds:
\begin{itemize}
\item[\rm(a)]
The ball $B^E_r(0)$ is a compact open subgroup of $(E,+)$
which is tidy for~$\alpha$, for each $r\in\,]0,\infty[$.
\item[\rm(b)]
We have
\[
\conp(\alpha)=E_{<1}:=\bigoplus_{\rho<1}E_\rho,\quad \conm(\alpha)=E_{>1}:=
\bigoplus_{\rho>1}E_\rho,
\]
\[
\parp(\alpha)=E_{\leq 1}:=
\bigoplus_{\rho\leq 1}E_\rho,\quad \parm(\alpha)=E_{\geq 1}:=\bigoplus_{\rho\geq 1}E_\rho,
\]
and $\lev(\alpha)=E_1$.
\item[\rm(c)]
The scale of~$\alpha$ is given by
\begin{equation}\label{scale-via-eig}
s(\alpha)
=\prod_{\stackrel{{\scriptstyle j\in\{1,\ldots, m\}}}{{\rm s.t.}\;|\lambda_j|_\K\geq 1}}
|\lambda_j|_\K,\vspace{-1mm}
\end{equation}
where $\lambda_1,\ldots,\lambda_m$ are the eigenvalues of $\alpha_{\wb{\K}}$ in
an algebraic closure $\wb{\K}$ of~$\K$,
with repetitions according to algebraic multiplicities.
\item[\rm(d)]
$s(\alpha)=s(\alpha|_{E_{>0}})=s(\alpha|_{E_{\geq 1}})=s(\alpha|_{E_{>1}})
=\Delta(\alpha|_{E_{\geq 1}})=\Delta(\alpha|_{E_{>1}})$.
\end{itemize}
\end{thm}
\begin{proof}
We endow vector subspaces of~$E$ with the induced norm.

(a) Since~$E$ admits the Fitting decompositon $E=E_0\oplus E_{>0}$
into~$E_0$ and~$E_{>0}$ which are $\alpha$-invariant vector subspaces
and
\[
B^E_r(0)=B^{E_0}_r(0)\times B^{E_{>0}}_r(0),
\]
we need only check that $B^{E_0}_r(0)$ is tidy for $\alpha|_{E_0}$
and $B^{E_{>0}}_r(0)$ is tidy for $\alpha|_{E_{>0}}$.
The first property holds by Lemma~\ref{nilp}, since $\alpha(B^{E_0}_r(0))\sub B^{E_0}_r(0)$
by Lemma~\ref{fact-adapted}\,(b). To check the second property,
after replacing~$\alpha$ with $\alpha|_{E_{>0}}$ we
may assume that~$\alpha$ is an automorphism. Thus,
let us consider $\alpha\in \GL(E)$ and verify that
\[
U:=B^E_r(0)=\prod_{\rho>0}B^{E_\rho}_r(0)
\]
is tidy for~$\alpha$. For each $k\in\Z$, have
\[
\alpha^k(B^E_r(0))=\prod_{\rho>0}B_{\rho^k r}^{E_\rho}(0),
\]
using Lemma~\ref{fact-adapted}\,(c).
Hence
\begin{equation}\label{Upluslin}
U_+:=\bigcap_{k=0}^\infty \alpha^k(U)=\prod_{\rho\geq 1}B^{E_\rho}_r(0)\;\;
\mbox{and}\;\;
U_-=\bigcap_{k=0}^\infty \alpha^{-k}(U)=\prod_{0<\rho\leq 1}B^{E_\rho}_r(0)
\end{equation}
(where we used \ref{uplusauto}),
entailing that $U=U_+ + U_-$ is tidy above for~$\alpha$. Since
\[
U_{--}:=\bigcup_{k=0}^\infty\alpha^{-k}(U_-)=\left(\prod_{0<\rho<1}E_\rho\right)
\times B^{E_1}_r(0)=E_{<1}\times B^{E_1}_r(0)
\]
is closed in~$E$, we deduce with \ref{variant-tidy} that $U$ is tidy.

(b) is obvious from Lemma~\ref{fact-adapted} (a), (b), and (c).

(c) Since $U_+=\prod_{\rho\geq 1}B^{E_\rho}_r(0)=B^{E_{\geq 1}}_r(0)$ is a compact open subgroup
of $(E_{\geq 1},+)$ such that $U_+\sub \alpha(U_+)$,
using \ref{mini-tidy} and (\ref{howgtmod}) we obtain
\[
s(\alpha)=[\alpha(U_+):U_+]=\Delta(\alpha|_{E_{\geq 1}}).
\]
As the $\lambda_j$ with $|\lambda_j|_\K\geq 1$
are exactly the eigenvalues of $\alpha|_{E_{\geq 1}}$ in~$\wb{\K}$,
Lemma~\ref{laindx} yields the desired formula.

(d) Eigenvalues $\lambda_j$ with $|\lambda_j|_\K=1$ are irrelevant for the product in~(c). Using Lemma~\ref{laindx}, we deduce that also $s(\alpha)=\Delta(\alpha|_{E_{>1}})$.
The first equality in~(d) holds by Proposition~\ref{scale-on-Fitt}.
Note that $B^{E_{\geq 1}}_r(0)$ and $B^{E_{>1}}_r(0)$
are tidy for $\alpha|_{E_{\geq 1}}$ and $\alpha|_{E_{>1}}$,
respectively, and are inflated by the latter.
The third and fourth scales in the formula therefore coincide with the corresponding
modules, by Lemma~\ref{nilp}\,(b).
\end{proof}
\begin{cor}\label{vsubquot}
Let $\alpha$ be a linear endomorphsm of a finite-dimensional
vector space~$E$ over a totally disconnected local field~$\K$.
Let $F$ be an $\alpha$-invariant vector subspace of~$E$ and
\[
\wb{\alpha}\colon E/F\to E/F,\quad x+F\mto\alpha(x)+F
\]
be the induced linear endomorphism of the quotient space~$E/F$.
Then
\begin{equation}\label{linprodformu}
s_E(\alpha)=s_F(\alpha|_F)s_{E/F}(\wb{\alpha}).
\end{equation}
\end{cor}
\begin{proof}
The eigenvalues of~$\alpha$ in an algebraic closure~$\wb{\K}$ of~$\K$
are exactly the eigenvalues of $\alpha|_F$, together with those of~$\wb{\alpha}$.
Hence (\ref{linprodformu}) follows from Theorem~\ref{thm-lincase}\,(c).
\end{proof}
For basic concepts concerning Lie algebras (which we always assume
finite-dimensional),\footnote{Except for the Lie algebras
of analytic vector fields mentioned in Section~\ref{sec-mfd}.}
see \cite{Bou}, \cite{Hum}, and \cite{Ser}.
\begin{la}\label{lin-lia}
If $\cg$ is a Lie algebra over a totally disconnected local field~$\K$
and $\alpha\colon \cg\to\cg$ a Lie algebra endomorphism,
then $\conp(\alpha)$, $\conm(\alpha)$, $\parp(\alpha)$, $\parm(\alpha)$,
and $\lev(\alpha)=\;\conp(\alpha)\, \cap\conm(\alpha)$
are Lie subalgebras of~$\cg$.
\end{la}
\begin{proof}
If $x$ and $y$ are elements of $\conp(\alpha)$ (resp., of $\parp(\alpha)$),
then $\alpha^n(x)$ and $\alpha^n(y)$ tend to~$0$ as $n\to\infty$
(resp., the elements form bounded sequences), entailing that also
\[
\alpha^n([x,y])=[\alpha^n(x),\alpha^n(y)]
\]
tends to~$0$ (resp., is bounded). Hence $[x,y]\in\;\conp(\alpha)$
(resp., $[x,y]\in\;\parp(\alpha)$).\linebreak
If $x$ and $y$ are elements of $\conm(\alpha)$ (resp., of $\parm(\alpha)$),
then we find an $\alpha$-regressive trajectory $(x_{-n})_{n\in\N_0}$
for~$x$ and an $\alpha$-regressive trajectory $(y_{-n})_{n\in\N_0}$
for~$y$ such that $x_{-n}\to 0$ and $y_{-n}\to 0$ as $n\to\infty$
(resp., $(x_{-n})_{n\in\N_0}$ and $(y_{-n})_{n\in\N_0}$
are bounded sequences).
Then $([x_{-n},y_{-n}])_{n\in\N_0}$ is an $\alpha$-regressive trajectory for
$[x,y]$, since
\[
\alpha([x_{-n-1},y_{-n-1}])=[\alpha(x_{-n-1}),\alpha(y_{-n-1})]
=[x_{-n},y_{-n}]\quad\mbox{for all $n\in\N_0$.}
\]
Moreover, $[x_{-n},y_{-n}]\to 0$ as $n\to\infty$
(resp., the sequence $([x_{-n},y_{-n}])_{n\in\N_0}$ is bounded),
showing that $[x,y]\in\; \conm(\alpha)$ (resp., $[x,y]\in\; \parm(\alpha)$).
\end{proof}
The following lemma will be used in Section~\ref{sec-typeR}.
\begin{la}\label{la-periodic}
Let $E$ be a finite-dimensonal
vector space over a totally disconnected local field $\K$ and $\alpha\in\GL(E)$ be an automorphism
such that
$|\lambda|_\K=1$ for all eigenvalues of~$\alpha$
in an algebraic closure~$\wb{\K}$ of~$\K$.
Then the subgroup $\langle\alpha\rangle$ generated by~$\alpha$ is relatively compact in~$\GL(E)$.
\end{la}
\begin{proof}
If $\K$ has characteristic $p>0$, then
it suffices to show that $\alpha^{p^n}$ generates a relatively
compact subgroup for some $n\in\N_0$, since
$\langle \alpha\rangle$ is contained in the finite union
\[
\bigcup_{j=0}^{p^n-1} \alpha^j \circ K
\]
of cosets of the compact group $K:=\overline{\langle \alpha^{p^n}\rangle}$.
We may therefore assume that the characteristic polynomial $p_\alpha\in\K[X]$
of~$\alpha$ is separable over~$\K$.
Let $\bL\sub\wb{\K}$ be a finite field extension of~$\K$
which is Galois and such that $p_\alpha$ splits into linear factors in~$\bL[X]$.
Then $\alpha$ has a unique multiplicative Jordan decomposition
\[
\alpha=\alpha_h\circ \alpha_u=\alpha_u\circ\alpha_h
\]
such that $(\alpha_h)_\bL\in\GL(E_\bL)$ is diagonalizable
and $(\alpha_u)_\bL \in\GL(E_\bL)$ is unipotent (see \cite[Theorem~I.4.4]{Bor}).
Let~$\bO$ be the valuation ring of~$\bL$ and $\bO^\times$ be its compact
group of invertible elements.
Since $|\lambda|_\K=1$ for all eigenvalues $\lambda\in\bL\sub\wb{\K}$ of~$\alpha_h$ (which coincide
with those of~$\alpha$), we have $\lambda\in\bO^\times$ and deduce that $(\alpha_h)_\bL$
generates a relatively compact subgroup~$L$ of $\GL(E_\bL)$.
Identify $\GL(E)$ with the closed subgroup $\{\beta\in \GL(E_\bL)\colon \beta(E)\sub E\}$
of $\GL(E_\bL)$. Then $\langle \alpha_h\rangle$ is contained in the compact subgroup
$L\cap\GL(E)$ and hence relatively compact in $\GL(E)$.
Now $(\alpha_u)_\bL$ generates a relatively compact subgroup of~$\GL(E_\bL)$,
by \cite[Lemma~4.1]{FOR}.
Hence~$\alpha_u$ generates a relatively compact subgroup of~$\GL(E)$,
by the preceding argument. Since
\[
\langle\alpha\rangle \sub \wb{\langle \alpha_h\rangle}\circ \wb{\langle \alpha_u\rangle},
\]
we see that $\langle\alpha\rangle$ is relatively compact.
\end{proof}
\section{\!\!\!Analytic functions, manifolds and Lie groups}\label{sec-mfd}
The section compiles definitions and elementary facts
concerning analytic functions,
manifolds, and Lie groups over totally disconnected local fields,
which we shall use without further explanation.
The section ends with two versions of
the Inverse Function Theorem, which will be essential
in the following.
\subsection*{Analytic manifolds and Lie groups}\label{anft}
Given a totally disconnected local field $(\K,|.|)$
and $n\in \N$,
we endow $\K^n$
with an ultrametric norm~$\|.\|$
(the choice of norm does not really matter
because all norms are equivalent; see \cite[Theorem~13.3]{Sch}).
If $\alpha\in \N_0^n$ is a multi-index, 
we write $|\alpha|:=\alpha_1+\cdots+\alpha_n$.
Confusion with the absolute value $|.|$ is
unlikely; the intended
meaning of~$|.|$ will always be clear from the context.
If $\alpha\in \N_0^n$ and $y=(y_1,\ldots, y_n)\in \K^n$,
we abbreviate $y^\alpha:=y_1^{\alpha_1}\cdots \hspace*{.3mm}y_n^{\alpha_n}$,
as usual. Compare~\cite{Ser} for the following concepts.
\begin{defn}
Given an open subset
$U\sub \K^n$, a map $f\colon U\to \K^m$
is called
\emph{analytic}\footnote{In other parts of the literature
related to rigid analytic geometry,
such functions are called \emph{locally analytic}
to distinguish them from functions which are globally given
by a power series.} (or $\K$-\emph{analytic}, if we wish to emphasize the ground field)
if it is given locally by a convergent
power series around each point $x\in U$, i.e.,
\[
f(x+y)=\sum_{\alpha\in \N_0^n} a_\alpha \, y^\alpha
\quad
\mbox{for all $\,y\in B^{\K^n}_r\hspace*{-.3mm}(0)$,}\vspace{-1.3mm}
\]
with $a_\alpha\in \K^m$
and some $r>0$ such that
$B^{\K^n}_r\hspace*{-1mm}(x)\sub U$ and
\[
\sum_{\alpha\in \N_0^n} \|a_\alpha\|\,
r^{|\alpha|}\;<\;\infty\,.\vspace{-1mm}
\]
\end{defn}
\noindent
Compositions
of analytic functions are analytic
\cite[Theorem, p.\,70]{Ser}.
We can therefore define an \emph{$n$-dimensional
analytic manifold}~$M$
over a totally\linebreak
disconnected local field~$\K$
in the usual
way, as a Hausdorff
topological space~$M$,
equipped with a maximal set~$\cA$ of homeomorphisms
$\phi\colon U_\phi\to V_\phi$ from open subsets $U_\phi\sub M$ onto
open subsets $V_\phi\sub \K^n$
such that the transition map
$\psi\circ \phi^{-1}$
is analytic, for all $\phi,\psi\in \cA$.\\[2.1mm]
In the preceding situation, the homeomorphsms $\phi\in\cA$ are called  \emph{charts} for~$M$,
and~$\cA$ is called an \emph{atlas}.\\[2.3mm]
A mapping $f\colon M\to N$ between
analytic manifolds is called \emph{analytic}
if it is continuous and $\phi\circ f\circ \psi^{-1}$
(which is a map between open subsets of $\K^m$ and $\K^n$)
is analytic, for all charts $\phi\colon U_\phi\to V_\phi\sub\K^n$ of~$N$
and charts $\psi\colon U_\psi\to V_\psi\sub \K^m$ of~$M$.\\[2.1mm]
If $(M,\cA_M)$ and $(N,\cA_N)$ are analytic manifolds modelled on $\K^m$ and $\K^n$, respectively,
then $M\times N$ with the product topology is an analytic manifold modelled
on $\K^{m+n}$, with the atlas containing
$\{\phi\times \psi\colon \phi\in\cA_M,\psi\in\cA_N\}$.\\[2.3mm]
Every open subset $U$ of a fnite-dimensional $\K$-vector space~$E$
can be considered as an analytic manifold, endowed with the maximal
atlas containing the global chart $\id_U\colon U\to U$.
Notably, we can speak about analytic functions
\[
f\colon U\to V
\]
if $U$ and $V$ are open subsets of finite-dimensional normed $\K$-vector spaces
$E$ and $F$, respectively. Any such function is totally differentiable
at each $x\in U$, and we write
\begin{equation}\label{totdiffa}
f'(x)\colon E\to F
\end{equation}
for its total differential. Deviating from~(\ref{totdiffa}),
we write $f'(x)=\frac{d}{dt}\big|_{t=0}\,f(x+t)$ if $E=\K$, as usual
(which is $f'(x)(1)$ in the notation of~(\ref{totdiffa})).\\[2.3mm]
A \emph{Lie group}
over a totally disconnected local field~$\K$ is a group~$G$,
equipped with an
analytic manifold structure which turns
the group multiplication
\[
\mu_G \colon G\times G\to G\,,\quad
(x,y)\mto xy
\]
and the group inversion
$\eta_G \colon G\to G$,
$x\mto x^{-1}$
into analytic mappings.\\[2.3mm]
Lie groups over $\Q_p$ are also called \emph{$p$-adic Lie groups}.
Besides the additive groups of finite-dimensional
$\K$-vector spaces,
the most obvious examples of $\K$-analytic Lie
groups are general linear groups.
\begin{example}
$\GL_n(\K)=\dt^{-1}(\K^\times)$ is an open subset
of the space $M_n(\K)\isom \K^{n^2}$
of $n\times n$-matrices
and hence is an $n^2$-dimensional
$\K$-analytic manifold.
The group operations are
rational maps and hence analytic.
\end{example}
\noindent
More generally,
one can show (cf.\ \cite[Chapter~I, Proposition~2.5.2]{Mar}):
\begin{example}
Every (group of $\K$-rational points of a)
linear algebraic group defined over~$\K$
is a $\K$-analytic Lie group,
viz.\ every subgroup
$G\leq \GL_n(\K)$ which is the set of
joint zeros of a set of polynomial
functions $M_n(\K)\to\K$.
E.g., $\SL_n(\K)=\{ A \in \GL_n(\K)\colon
\dt(A )=1\}$
is a $\K$-analytic Lie group.
\end{example}
\begin{rem}
See Example~\ref{non-lin1} (first mentioned in \cite[Remark~9.7]{Lec})
for a Lie group~$G$
over $\K=\F_p(\!(X)\!)$
which is not a linear Lie group, i.e.,
which does not admit
a faithful, continuous
linear representation
$G\to\GL_n(\K)$ for any~$n$.
We shall also encounter
a $p$-adic Lie group
which is not isomorphic to a closed
subgroup of $\GL_n(\Q_p)$
for any $n\in\N$ (Example~\ref{non-lin2}).
\end{rem}
\begin{rem}
The analytic manifolds and Lie groups we consider need not be second countable
topological spaces. In particular, arbitrary discrete groups (countable or not)
can be considered as ($0$-dimensional) $p$-adic Lie groups,
which is natural from the point of view of topological groups.\\[2.3mm]
All the Lie groups and manifolds considered in these notes
are analytic and finite-dimensional.
For smooth Lie groups modelled on (not necessarily finite-dimensional)
topological vector spaces over a topological field, see \cite{BGN},
\cite{ZOO} and the references therein.
\end{rem}
{\bf Tangent vectors, tangent spaces, and tangent maps.}
Tangent vectors can be defined in many ways.
We choose a description which parallels
the so-called ``geometric'' tangent vectors
in the real case.\footnote{Compare \cite{Bou}, \cite{Ser}, also \cite{BGN}
for the following facts (although in different formulations).}
Given an $m$-dimensonal analytic manifold~$M$
over a totally disconnected local field~$\K$ and $p\in M$,
let us say that two analytic mappings
\[
\gamma\colon B^\K_\ve(0)\to M\quad\mbox{and}\quad \eta\colon B^\K_\delta(0)\to M
\]
with $\gamma(0)=\eta(0)=p$ are equivalent (and write $\gamma\sim_p\eta$)
if
\begin{equation}\label{oneorall}
(\phi\circ \gamma)'(0)=(\phi\circ \eta)'(0)
\end{equation}
for some chart $\phi\colon U\to V$ of~$M$ around~$p$ (i.e., with $p\in U$);
here $\ve,\delta>0$.
Then (\ref{oneorall}) holds for all charts around~$p$ (by the Chain Rule),
and we easily deduce that $\sim_p$ is an equivalence relation.
The equialence classes $[\gamma]$ with respect to $\sim_p$
are called \emph{tangent vectors} for~$M$ at~$p$.
The set $T_p(M)$ of all tangent vectors at~$p$ is called
the \emph{tangent space} of~$M$ at~$p$. We endow it with the unique
vector space structure making the bijection
\[
T_p(M)\to \K^m,\quad [\gamma]\mto (\phi\circ\gamma)'(0)
\]
a vector space isomorphism for some (and hence every) chart~$\phi$
of~$M$ around~$p$. The union
\[
T(M):=\bigcup_{p\in M}T_p(M)
\]
is disjoint and is called the \emph{tangent bundle} of~$M$.
If $f\colon M\to N$ is an analytic map between analytic manifolds,
we obtain a linear map
\[
T_p(f)\colon T_p(M)\to T_{f(p)}(N),\quad [\gamma]\mto [f\circ\gamma]
\]
called the \emph{tangent map of $f$} at~$p$.
The map $T(f)\colon T(M)\to T(N)$
taking $v\in T_p(M)$ to $T_p(f)(v)$ is called
the  \emph{tangent map} of~$f$.
If also $K$ is an analytic manifold over~$\K$ and $g\colon K\to M$
an analytic mapping, then
\begin{equation}\label{T-functor}
T(f\circ g)=T(f)\circ T(g)
\end{equation}
as both mappings take a tangent vector $[\gamma]\in T(K)$ to $[f\circ g\circ\gamma]$.
If $U$ is an open subset of a finite-dimensional vector space~$E$,
we identify $T(U)$ with $U\times E$ using the bijection
\begin{equation}\label{essentialid}
T(U)\to U\times E,\quad [\gamma]\mto (\gamma(0),\gamma'(0)).
\end{equation}
If $U$ is as before, $M$ an analytic manifold and $f\colon M\to U$
an analytic map, we write $df$ for the second component of the map
\[
T(f)\colon T(M)\to T(U)=U\times E,
\]
using the identification from~(\ref{essentialid}). Thus
\[
df([\gamma])=(f\circ \gamma)'(0).
\]
If $U$ and $V$ are open subsets of finite-dimensional $\K$-vector spaces~$E$
and $F$, respectively, and $f\colon U\to V$ is an analytic map, then
$T(f)\colon T(U)\to T(V)$ is the mapping
\[
U\times E\to V\times F,\quad (x,y)\mto (f(x),df(x,y))
\]
with $df(x,y)=f'(x)(y)$.\\[2.3mm]
{\bf Submanifolds and Lie subgroups.}
Let $M$ be an $m$-dimensional analytic manifold over a totally disconnected
local field~$\K$
and $n\in\{0,1,\ldots, m\}$.
A subset $N\sub M$ is called an $n$-dimensional \emph{submanifold}
of~$M$ if, for each $p\in N$, there exists a chart
$\phi\colon U\to V\sub \K^m$
of~$M$ around~$p$ such that
\[
\phi(U\cap N)=V\cap (\K^n\times \{0\}).
\]
Identifying $\K^n\times \{0\}\sub\K^m$ with $\K^n$ via $(x,0)\mto x$,
we get a homeomorphism
\[
\phi_N:=\phi|_{U\cap N} \colon U\cap N\to V\cap (\K^n\times\{0\})\sub \K^n.
\]
Then~$N$ is an $n$-dimensional analytic manifold
in a natural way, using the topology induced by~$M$ and the maximal atlas containing
all of the maps~$\phi_N$. Using this manifold structure,
the inclusion $j\colon N\to M$ is analytic.
For each $p\in N$, the tangent map $T_p(j)\colon T_p(N)\to T_p(M)$
is injective, and will be used to identify $T_p(N)$ with the image
of $T_p(j)$ in $T_p(M)$.
Moreover, for each analytic manifold~$K$, a mapping $f\colon K\to N$
is analytic if and only if $j\circ f\colon K\to M$ is analytic.
We say that a subgroup~$H$ of a Lie group~$G$ over~$\K$
is a \emph{Lie subgroup} if it is a submanifold.
By the preceding fact, the submanifold structure then turns the group
operations on~$H$ into analytic mappings and thus makes~$H$ a Lie group.
\begin{la}\label{spot-liesub}
Let $G$ be a Lie group over a totally disconnected local field~$\K$, of dimension~$m$.
A subgroup $H\sub G$ is a Lie subgroup of dimension~$n$
if and only if
there exists a chart
$\phi\colon U\to V\sub \K^m$
of~$G$ around~$e$ such that
$\phi(U\cap H)=V\cap (\K^n\times \{0\})$.
\end{la}
\begin{proof}
The necessity is clear. Sufficiency:
For each $h\in H$, the map $\phi_h\colon hU\to V$, $x\mto \phi(h^{-1}x)$
is a chart for~$G$ such that $\phi_h(hU\cap H)=\phi(U\cap H)=V\cap(\K^n\times\{0\})$.
\end{proof}
\noindent
{\bf The Lie algebra functor.}
An \emph{analytic vector field} on an $m$-dimensional $\K$-analytic manifold~$M$
is a mapping $X\colon M\to T(M)$ with $X(p)\in T_p(M)$
for all $p\in M$, which is analytic in the sense
that its local representative
\[
X_\phi:=d\phi\circ X\circ \phi^{-1}\colon V\to \K^m
\]
is an analytic function for each chart $\phi\colon U\to V\sub\K^m$
of~$M$. The set $\cV^\omega(M)$
of all analytic vector fields on~$M$ is a $\K$-vector space,
with pointwise addition and scalar multiplication.
Given $X,Y\in \cV^\omega(M)$, there is a unique vector field
$[X,Y]\in\cV^\omega(M)$ such that
\[
[X,Y]_\phi=dY_\phi\circ (\id_V,X_\phi)-dX_\phi\circ (\id_V,Y_\phi)
\]
for all charts $\phi\colon U\to V$ of~$M$,
and $[.,.]$ makes $\cV^\omega(M)$ a Lie algebra.\\[2.3mm]
If $G$ is a $\K$-analytic Lie group,
then its tangent space $L(G):=T_e(G)$
at the identity element can be made
a Lie algebra via the identification
of $v\in L(G)$ with the corresponding
left invariant vector field $v_\ell$ on~$G$
given by $v_\ell(g):=T\lambda_g(x)$ for $g\in G$
with left translation $\lambda_g\colon G\to G$, $x\mto gx$
(noting that the left invariant
vector fields form a Lie subalgebra
of~$\cV^\omega(G)$).\\[2.1mm]
If $\alpha \colon G\to H$ is an analytic
group homomorphism between $\K$-analytic Lie groups,
then the tangent map $L(\alpha):=T_e(\alpha)\colon L(G)\to L(H)$
is a linear map and actually a Lie algebra
homomorphism (cf.\ \cite[Chapter~III, \S3, no.\,8]{Bou} and
Lemma~5.1 on p.\,129 in
\cite[Part II, Chapter~V.1]{Ser}).
An \emph{analytic automorphism} of a Lie group $G$ is
an invertible group homomorphism $\alpha\colon G\to G$
such that both $\alpha$ and $\alpha^{-1}$ are analytic.
For example, each inner automorphism~$I_g$
of~$G$ is analytic. As usual, we abbreviate $\Ad_g:=L(I_g)$.\\[2.3mm]
Since $I_g\circ I_h=I_{gh}$ for $g,h\in G$,
we have $\Ad_{gh}=\Ad_g\circ\Ad_h$ by (\ref{T-functor}).
In Section~\ref{sec-typeR}, we shall use the continuity of the adjoint representation
of $G$ on its Lie algebra $\cg:=L(G)$. Even more is true (see Definition~8
in \cite[Chapter~III, \S3, no.\,12]{Bou}
and the lines preceding it):
\begin{numba}\label{Ad-ana}
The map $\Ad\colon G\to\Aut(\cg)\sub \GL(\cg)$, $g\mto \Ad_g$ is analytic.
\end{numba}
\subsection*{Ultrametric inverse function theorems}\label{ultinv}
Since small perturbations
do not change the size of a given non-zero vector
in the ultrametric case
(see~(\ref{winner})),
the ultrametric inverse function theorem
has a nicer form
than its classical real counterpart.
Around a point~$p$
with invertible differential,
an analytic map~$f$ behaves like an affine-linear map (its linearization).
If the differential at~$p$ is an isometry, then also~$f$ is isometric
on a neighbourhood of~$p$.\\[2mm]
In the following two lemmas,
we let $\K$ be a totally disconnected local
field and $|.|$ be an absolute value on~$\K$
defining its topology.
We fix an ultrametric norm~$\|.\|$ on a finite-dimensional $\K$-vector space~$E$ and write
\[
\Iso(E,\|.\|):=\{\alpha \in\GL(E)\colon (\forall x\in E)\;\|\alpha(x)\|=\|x\|\}
\]
for the group of linear isometries.
It is well-known that $\Iso(E,\|.\|)$ is open in $\GL(E)$
(see, e.g., \cite[Lemma~7.2]{Isr}), but we shall not use this fact.
Given $x\in E$ and $r>0$, we abbreviate $B_r(x):=B^E_r(x)$.
The total differential of $f$ at~$x$
is denoted by~$f'(x)$.
The ultrametric inverse function theorem
(for analytic functions)
subsumes the following:\footnote{A proof is obtained, e.g.,
by combining \cite[Proposition~7.1\,(a)$'$ and (b)$'$]{Isr}
with
the inverse function theorem for
analytic maps from \cite[p.\,73]{Ser},
recalling that analytic maps are strictly differentable
at each point (in the sense of \cite[1.2.2]{Fas}, by \cite[4.2.3 and 3.2.4]{Fas}.}
\begin{la}\label{invfct}
Let $f\colon U\to E$
be an analytic map on an open subset $U\sub E$
and $x\in U$
such that $f'(x)\in \GL(E)$.
Then there exists $r>0$ such that
$B_r(x)\sub U$,
\begin{equation}\label{great}
f(B_t(y))\;=\;f(y)+ f'(x)B_t(0)\quad
\mbox{for all $y\in B_r(x)$ and $t\in \;]0,r]$,}
\end{equation}
and $f|_{B_r(x)}\colon B_r(x)\to f(B_r(x))$ is an
analytic diffeomorphism. If $f'(x)\in\Iso(E,\|.\|)$,
then~$r$ can be chosen such that
$B_r(x)\sub U$,
\begin{equation}\label{greater}
f(B_t(y))\;=\;B_t(f(y))\quad
\mbox{for all $y\in B_r(x)$ and $t\in \;]0,r]$,}
\end{equation}
and $f|_{B_r(x)}\colon B_r(x)\to B_r(f(x))$ is an isometric,
analytic diffeomorphism.\,\Punkt
\end{la}
\noindent
It is useful that $r$ can be chosen
uniformly in the presence of parameters.
As a special case of \cite[Theorem~7.4\,(b)$'$]{Isr},
an `ultrametric
inverse function theorem with parameters'
is available:\footnote{To achieve that $f_q|_{B_r(x)}$
is an isometry for all $q\in Q$,
note that \cite[Lemma~6.1\,(b)]{Isr}
applies to all of these functions by \cite[p.\,239, lines 7--8]{Isr}.}
\begin{la}\label{invpar}
\!Let $F$ be a finite-dimensional $\K$-vector space,
$P\!\sub \! F$ and
$U\!\sub \! E$ be open,
$f\colon P \times U\to E$
be a $\K$-analytic map, $p\in P$
and $x\in U$
such that $f'_p(x)\in \Iso(E,\|.\|)$,
where $f_p:=f(p,\sbull)\colon U\to E$.
Then there exists an open neighbourhood $Q\sub P$ of~$p$
and $r>0$ such that
$B_r(x)\sub U$,
%\ma{great2}
\begin{equation}\label{great2}
f_q(B_t(y))\;=\;f_q(y)+ B_t(0)
\end{equation}
and $f_q|_{B_t(y)}$ is an isometry,
for all $q\in Q$, $y\in B_r(x)$ and $t\in \;]0,r]$.\,\Punkt
\end{la}
\section{Construction of small open subgroups}\label{constsmall}
It is essential for our following discussions
that Lie groups over totally disconnected local fields have a basis
of identity neighbourhoods consisting
of compact open subgroups which correspond to balls
in the Lie algebra. In this section,
we explain how these compact open subgroups
can be constructed.\\[2.3mm]
Let $G$ be a Lie group over a totally disconnected local field~$\K$
and $|.|$ be an absolute value on~$\K$ defining its topology.
Fix an ultrametric norm $\|.\|$ on $\cg:=L(G)$
and abbreviate $B_t(x):=B^{\cg}_t(x)$
for $x\in \cg$ and $t>0$.
Let
\[
\phi\colon U\to V
\]
by an analytic diffeomorphism from an open identity neighbourhood
$U\sub G$ onto an open $0$-neighbourhood $V\sub \cg$,
such that $\phi(e)=0$ and
\begin{equation}\label{onceorso}
d\phi|_\cg=\id_\cg.
\end{equation}
\begin{numba}\label{pre-laballs}
After shrinking~$U$ (and~$V$), we may assume that~$U$
is a compact open subgroup of~$G$.
Then
\[
\mu_V\colon V\times V\to V,\; (x,y)\mto
x*y :=\phi(\phi^{-1}(x)\phi^{-1}(y))
\]
is a group multiplication on~$V$ with neutral element~$0$
which turns~$V$ into an analytic Lie group and~$\phi$
into an isomorphism of Lie groups.
It is easy to see that the first order Taylor
expansions of multiplication and inversion in $(V,*)$
at $(0,0)$ and $0$, respectively, are given by
\begin{equation}\label{Tay01}
x*y\;=\; x+y+\cdots
\end{equation}
and
\begin{equation}\label{Tay02}
x^{-1}\;=\; -x+\cdots
\end{equation}
(compare \cite[p.\,113]{Ser}).
Applying the Ultrametric Inverse Function
Theorem with Parameters (Lemma~\ref{invpar}) to the maps
$(x,y)\mto x*y$
and $(x,y)\mto y*x$ around $(0,0)$,
we find $R>0$ with $B_R(0)\sub V$ such that
\begin{equation}\label{premeas}
x*B_t(0)\;=\; x+B_t(0)\;=\; B_t(0)*x
\end{equation}
for all $x\in B_R(0)$ and $t\in \;]0,R]$
(exploiting that both relevant partial differentials
are $\id_\cg$ and hence an isometry,
by~(\ref{Tay01})).
Notably, (\ref{premeas}) entails that
\[
B_t(0)*B_t(0)=B_t(0)\;\;\mbox{for each $t\in \;]0,R]$,}
\]
whence
$y^{-1}\in B_t(0)$ for each $t\in\;]0,R]$ and $y\in B_t(0)$.
\end{numba}
Summing up (with $B_t:=B_t(0))$:
\begin{la}\label{laballs}
$(B_t,*)$ is a group
for each $t\in \;]0,R]$ and hence
$B^\phi_t:=\phi^{-1}(B_t)$ is a compact
open subgroup of~$G$, for each $t\in \;]0,R]$.
Moreover, $B_t$ is a normal subgroup of $(B_R,*)$,
whence $B^\phi_t$ is normal in $B^\phi_R$.\,\Punkt
\end{la}
Thus small balls in $\cg$ correspond to
compact open subgroups in~$G$.
\begin{rem}\label{smidx}
(\ref{premeas}) entails that
the indices of $B_t$ in $(B_R,+)$
and $(B_R,*)$ coincide (as the cosets coincide),
for all $t\in\;]0,R]$.
\end{rem}
\begin{numba}\label{the-r}
Now consider an analytic endomorphsm
$\alpha\colon G\to G$, or, more generally,
an analytic homomorphism
\[
\alpha\colon G_0\to G
\]
defined on an open subgroup $G_0\sub G$. For the domain~$U$ of~$\phi$, assume that $U\sub G_0$.
After shrinking~$R$ (if necessary) we may assume that
\begin{equation}\label{better-r}
\alpha(B^\phi_R)\sub U,
\end{equation}
whence an analytic homomorphism
\begin{equation}\label{my-beta}
\beta:=\phi\circ \alpha|_{B^\phi_R}\circ\phi^{-1}|_{B_R}\colon B_R\to V
\end{equation}
can be defined such that
\begin{equation}\label{equibeta}
\beta\circ \phi|_{B^\phi_R}=\phi\circ\alpha|_{B^\phi_R}.
\end{equation}
As a consequence of~(\ref{onceorso}), we have
\begin{equation}\label{spot-der}
\beta'(0)=L(\alpha).
\end{equation}
\end{numba}
For $\alpha\colon G\to G$ an analytic automorphism
and $\|.\|$ adapted to $L(\alpha)$,
we shall see in Section~\ref{sec-endo} that the
groups $B^\phi_t:=\phi^{-1}(B_t)$
are tidy for~$\alpha$ and $t\in \;]0,R]$ close to~$0$,
as long as $\conp(\alpha)$ is closed
(and also the case of $\alpha\colon G_0\to G$ will be used).
This motivates us to calculate the displacement
indices for the compact open subgroups $B^\phi_t\sub G$.
\begin{la}\label{displ-balls}
Let $G$ be a Lie group over a totally disconnected local field, $G_0$ be an open subgroup of~$G$
and $\alpha\colon G_0\to G$ an analytic homomorphism which is
an analytic diffeomorphism onto an open subgroup $\alpha(G_0)$ of~$G$.
Let~$\phi$ be as before, $\|.\|$ be adapted to $L(\alpha)$,
and~$R$ be as in {\rm\ref{the-r}}.
Then there exists $t_0\in\;]0,R]$ such that
\[
[\alpha(B^\phi_t):\alpha(B^\phi_t)\cap B^\phi_t]=s(L(\alpha))\;\;
\mbox{for all $\,t\in\;]0,t_0]$.}
\]
\end{la}
\begin{proof}
Let $\beta$ be as in (\ref{my-beta}).
By (\ref{spot-der}) and the Ultrametric Inverse Function Theorem (Lemma~\ref{invfct}),
there is $t_0\in\;]0,R]$ with $L(\alpha)(B_{t_0})\sub B_R$ such that
\[
(\phi \circ \alpha\circ \phi^{-1})(B_t)= L(\alpha)(B_t)
\]
for all $t\in]0,t_0]$ and hence
\begin{equation}\label{umstell}
\alpha (B^\phi_t)=\phi^{-1}(L(\alpha)(B_t)).
\end{equation}
Given $t\in\;]0,t_0]$, there exists $\theta \in\;]0,t]$ such that $B_\theta\sub L(\alpha)(B_t)$. Then
\begin{eqnarray*}
s(L(\alpha))&=&[L(\alpha)(B_t):L(\alpha)(B_t)\cap B_t]\\[1mm]
&=&\frac{[L(\alpha)(B_t):B_\theta]}{[L(\alpha)(B_t)\cap B_t:B_\theta]}\;\;\;\mbox{\,in $(B_R,+)$}\\[1mm]
&=&\frac{[L(\alpha)(B_t):B_\theta]}{[L(\alpha)(B_t)\cap B_t:B_\theta]}\;\;\;\hspace*{.4mm}\mbox{\,in $(B_R,*)$}\\[1mm]
&=&\frac{[\alpha(B^\phi_t):B^\phi_\theta]}{[\alpha(B^\phi_t)\cap B^\phi_t:B^\phi_\theta]}
=[\alpha(B^\phi_t):\alpha(B^\phi_t)\cap B^\phi_t]
\end{eqnarray*}
using Remark~\ref{smidx} for the third equality;
to obtain the final equality, (\ref{umstell}) was used and the fact that
$\phi\colon B^\phi_R\to (B_R,*)$ is an isomorphism.
\end{proof}
The following lemma shows that different choices of~$\phi$
do not affect the $B^\phi_t$ for small~$t$
(as long as the norm is unchanged).
\begin{la}\label{ball-unq}
Let $M$ be an analytic manifold over a totally disconnected local field~$\K$,
$E$ be a finite-dimensional $\K$-vector space, and~$\|.\|$ be an ultrametric norm on~$E$.
Let $p\in M$
and $\phi_j\colon U_j\to V_j$, for $j\in\{1,2\}$,
be an analytic diffeomorphism from an open neighbourhood $U_j$ of~$p$ in~$M$
onto an open $0$-neighbourhood $V_j\sub E$ such that $\phi_j(p)=0$.
If $d\phi_1|_{T_p(M)}=d\phi_2|_{T_p(M)}$, then there exists $\ve>0$
with $B^E_\ve(0)\sub V_1\cap V_2$ such that
\[
\phi_1^{-1}(B^E_t(0))=\phi_2^{-1}(B^E_t(0))\quad\mbox{for all $\,t\in\;]0,\ve]$.}
\]
\end{la}
\begin{proof}
The map $h:=\phi_2 \circ\phi_1^{-1}\colon \phi_1(U_1\cap U_2)\to\phi_2(U_1\cap U_2)$
is an analytic diffeomorphism between open $0$-neighbourhoods in~$E$.
Since $T_0(h)=\id_{T_0(E)}$,
we have $h'(0)=\id_E$, which is an isometry.
Thus, the Ultrametric Inverse Function Theorem provide $\ve>0$
with $B^E_\ve(0)\sub \phi_1(U_1\cap U_2)$
such that $h(B^E_t(0))=B^E_t(0)$ for all $t\in\;]0,\ve]$.
Notably, $B^E_t(0)\sub \phi_1(U_1\cap U_2)$
and $\phi_1^{-1}(B^E_t(0))=\phi_2^{-1}(\phi_2(\phi_1^{-1}(B^E_t(0))))=\phi_2^{-1}(B^E_t(0))$.
\end{proof}
\section{Endomorphisms of {\boldmath$p$}-adic Lie groups}\label{sec-padic}
In this section, we first recall general facts
concerning $p$-adic Lie groups
which go beyond the properties of Lie groups over general local fields
already described. In particular, we recall that every $p$-adic Lie
group has an exponential function, and show that
contraction groups of endomorphisms of $p$-adic Lie groups
are always closed.
We then calculate the scale and describe tidy subgroups
for endomorphisms of $p$-adic Lie groups.
\subsection*{{\normalsize Basic facts concerning
{\boldmath $p$}-adic Lie groups}}
For each $n\in\N$,
the exponential series $\sum_{k=0}^\infty\frac{1}{k!}A^k$
converges for matrices $A$ in some $0$-neighbourhood
$V$ in the algebra $M_n(\Q_p)$ of $n\times n$-matrices and defines an
analytic mapping $\exp\colon V\to \GL_n(\Q_p)$.
More generally, every analytic Lie group~$G$ over~$\Q_p$
has an exponential function (see Definition~1 and the following lines
in \cite[Chapter~III, \S4, no.\,3]{Bou}):
\begin{numba}
An analytic map $\exp_G\colon V\to G$
on an open $\Z_p$-submodule $V\sub \cg:=L(G)$
is called an \emph{exponential
function} if $\exp_G(0)=e$, $T_0(\exp_G)=\id_\cg$
(identifying $T_0(\cg)=\{0\}\times \cg$ with~$\cg$ via $(0,v)\mto v$)
and
\[
\exp_G((s+t)x)\; =\; \exp_G(sx)\exp_G(tx)
\]
for all $x\in U$ and $s,t\in \Z_p$.
\end{numba}
\begin{numba}
Since $T_0(\exp_G)=\id_\cg$,
after shrinking~$V$ one can assume that
$\exp_G(V)$ is open in~$G$ and $\exp_G$ is a
diffeomorphism onto its image
(by the Inverse Function Theorem).
After shrinking~$V$ further if necessary,
we may assume that $\exp_G(V)$
is a subgroup of~$G$ (cf.\ Lemma~\ref{laballs}).
Hence also~$V$
can be considered as a Lie group.
The Taylor expansion
of multiplication with respect to the
logarithmic chart $\exp_G^{-1}$ is given by the Baker-Campbell-Hausdorff
(BCH-) series
\begin{equation}\label{BCH}
x*y\, = \, x+y+\frac{1}{2}[x,y]+\cdots
\end{equation}
(all terms of which are nested
Lie brackets with rational coefficients),
and hence $x*y$ is given by this series
for small~$V$ (see Proposition~5 in \cite[Chapter~III, \S4, no.\,3]{Bou}
and proof of Proposition~3 in
\cite[Chapter~III, \S7, no.\,2]{Bou}, also~\cite{Ser}).
If $*$ is given on all of $V\times V$
by the BCH-series,
we call $\exp_G(V)$ a \emph{BCH-subgroup} of~$G$.
\end{numba}
Next, let us consider homomorphisms between $p$-adic Lie groups.
\begin{numba}\label{recall-nat}
If $\alpha\colon G\to H$ is an analytic homomorphism
between $p$-adic Lie groups,
we can choose exponential functions
$\exp_G\colon V_G\to G$
and $\exp_H\colon V_H\to H$
such that $L(\alpha).V_G \sub V_H$
and
\begin{equation}\label{locnat}
\exp_H\circ L(\alpha)|_{V_G}\; =\; \alpha \circ \exp_G
\end{equation}
(see Proposition~8 in \cite[Chapter~III, \S4, no.\,4]{Bou}, also~\cite{Ser}).
\end{numba}
The following classical fact (see Theorem~1 in \cite[Chapter~III, \S8, no.\,1]{Bou},
also~\cite{Ser}) is important:
\begin{numba}\label{Cartan}
Every continuous homomorphism
between $p$-adic Lie groups
is analytic.
\end{numba}
\noindent
As a consequence, there is at most one
$p$-adic Lie group structure on a
topological group. As usual,
we say that a topological group is a $p$-adic Lie group
if it admits a $p$-adic Lie
group structure.
Closed subgroups of $p$-adic Lie groups are Lie subgroups (see Theorem~2 in
\cite[Chapter~III, \S8, no.\,2]{Bou} or \cite{Ser}), finite direct products
and Hausdorff
quotient groups of $p$-adic
Lie groups are $p$-adic
Lie groups (see Proposition~11 in \cite[Chapter~III, \S1, no.\,6]{Bou}, also~\cite{Ser}).
\subsection*{{\normalsize Closedness of ascending unions and contraction groups}}
Another fact is essential:
\begin{la}\label{unoclos}
Every $p$-adic Lie group~$G$
has an open subgroup which
satisfies the ascending chain condition
on closed subgroups.
As a consequence, $\bigcup_{n\in\N}H_n$ is closed
for each ascending sequence $H_1\sub H_2\sub\cdots$ of closed
subgroups of~$G$.
\end{la}
\begin{proof}
See, e.g. \cite[Propositions~2.3 and~2.5]{Lec}; cf.\ also 
step~1 of the proof of \cite[Theorem 3.5]{Wan}.
\end{proof}
Two important applications are now described.
\begin{cor}\label{autotb}
Let $\alpha$ be an endomorphism of a $p$-adic Lie group~$G$
and~$V$ be a compact open subgroup of~$G$.
If $V$ is tidy above for~$\alpha$, then~$V$ is tidy.
\end{cor}
\begin{proof}
The subgroup $V_{--}=\bigcup_{n\in \N_0}\alpha^{-n}(V_-)$
is an ascending union of closed
subgroups of~$G$ and hence closed, by Lemma~\ref{unoclos}.
Thus~$V$ is tidy, by~\ref{variant-tidy}.
\end{proof}
\noindent
The second application
of Lemma~\ref{unoclos}
concerns contraction groups.
For automorphisms, see already \cite[Theorem~3.5\,(ii)]{Wan}.
\begin{cor}\label{corclo}\label{isclosed}
Let $G$ be a $p$-adic Lie group.
Then the contraction group $\conp(\alpha)$
is closed in~$G$,
for each endomorphism
$\alpha\colon G\to G$.
\end{cor}
\begin{proof}
Let $V_1\supseteq V_2\supseteq\cdots$ be a sequence
of compact open subgroups of~$G$ which
form a basis of identity neighbourhoods
(cf.\ Lemma~\ref{laballs}).
Then an element $x\in G$ belongs to $\conp(\alpha)$
if and only if
\[
(\forall n\in \N)\, (\exists m\in \N)\,(\forall k\geq m)\;\;
\alpha^k(x)\in V_n\,.
\]
Since
$\alpha^k(x)\in V_n$ if and only if $x\in \alpha^{-k}(V_n)$,
we deduce that
\[
\conp(\alpha)\; =\; \bigcap_{n\in\N}\bigcup_{m\in \N}\bigcap_{k\geq m}
\alpha^{-k}(V_n)\,.
\]
Note that $W_n:=\bigcup_{m\in \N}\bigcap_{k\geq m}
\alpha^{-k}(V_n)$ is an ascending union of closed subgroups
of~$G$ and hence closed, by Proposition~\ref{unoclos}.
Consequently, $\conp(\alpha)=\bigcap_{n\in \N}\, W_n$
is closed.\vspace{-1mm}
\end{proof}
\begin{rem}
We shall see later that also $\conm(\alpha)$
is always closed in the situation of Corollary~\ref{corclo}
(see Theorem~\ref{bigcell}).
Alternatively, this follows from the general structure theory
(see \cite[Proposition~9.4]{BGT}).\vspace{-1mm}
\end{rem}
\subsection*{{\normalsize Scale and tidy subgroups}}
The following lemma prepares the construction of tidy subgroups
in $p$-adic Lie groups, and can also be re-used later
when we turn to Lie groups over general local fields.
As two endomorphisms are discussed simultaneously
in the lemma,
we use notation as in \ref{with-label}.
\begin{la}\label{local-conju}
Let $G$ and $H$ be totally disconnected, locally compact groups,
$\alpha\colon G\to G$ and $\beta\colon H\to H$ be endomorphisms,
$U\sub G$ and $V\sub H$ be open identity neighbourhoods and
$\psi\colon V\to U$ be a bijection.
Assume that there exists a compact open subgroup $B\sub H$ such that
$B\sub V$, $\beta(B)\sub V$, the image
$W:=\psi(B)$ is a compact open subgroup of~$G$, and
\begin{equation}\label{local-equi}
\alpha\circ\psi|_B=\psi\circ \beta|_B.
\end{equation}
Write $B_+:=B_{+,\beta}$, $B_-:=B_{-,\beta}$, $W_+:=W_{+,\alpha}$
and $W_-:=W_{-,\alpha}$. Then
\[
\psi(B_+)=W_+,\quad
\psi(B_-)=W_-\quad\mbox{and}\quad \psi(\beta(B_+))=\alpha(W_+).
\]
\end{la}
\begin{proof}
We define $B_n:=B_{n,\beta}$ and $W_n:=W_{n,\alpha}$ for $n\in\N_0$ as
in \ref{with-label}.
Then $B_n\sub B$ for each $n\in\N_0$, by construction. We show that
\begin{equation}\label{B-induct}
\psi(B_n)=W_n
\end{equation}
for all $n\in\N_0$, by induction.
The case $n=0$ is clear: we have $\psi(B_0)=\psi(B)=W=W_0$.
Now assume that (\ref{B-induct}) holds for some~$n$.
Since $B_n\sub B$, we have $\beta(B_n)\sub V$. Using that
$\psi$ is injective, (\ref{local-equi}), and the inductive hypothesis,
we see that
\begin{eqnarray*}
\psi(B_{n+1})&=& \psi(\beta(B_n)\cap B)=\psi(\beta(B_n))\cap \psi(B)
=\alpha(\psi(B_n))\cap W\\
&=&\alpha(W_n)\cap W=W_{n+1}.
\end{eqnarray*}
Thus (\ref{B-induct}) holds for all $n\in\N_0$. Since $\psi$ is injective, we deduce that
\[
\psi(B_+)=\psi\left(\bigcap_{n\in\N_0}B_n\right)=\bigcap_{n\in\N_0}\psi(B_n)=\bigcap_{n\in\N_0}
W_n=W_+.
\]
As $B_+\sub B$, using (\ref{local-equi}) also $\psi(\beta(B_+))=\alpha(\psi(B_+))=\alpha(W_+)$
follows. Finally, for $n\in \N_0$ let $B_{-n}$ be the set of all $x\in B$ such that
$\beta^k(x)\in B$ for all $k\in\{0,1,\ldots, n\}$,
and $W_{-n}$ be the set of all $w\in W$ such that $\alpha^k(w)\in W$
for all $k\in\{0,1,\ldots,n\}$. We claim that
\begin{equation}\label{B-indu2}
\psi(B_{-n})=W_{-n}\quad\mbox{for all $\,n\in\N_0$.}
\end{equation}
Since $B_-=\bigcap_{n\in\N_0}B_{-n}$ with $B_{-n}\sub B\sub V$ for all~$n\in\N_0$,
using the injectivity of~$\psi$ we then get
\[
\psi(B_-)=\psi\left(\bigcap_{n\in\N_0}B_{-n}\right)
=\bigcap_{n\in\N_0}\psi(B_{-n})=\bigcap_{n\in\N_0}W_{-n}=W_-.
\]
It only remains to prove the claim.
It suffices to show that
\begin{equation}\label{B-indu3}
\psi(B_{-n})\sub W_{-n}
\end{equation}
for all $n\in\N_0$,
as the arguments can also be applied to $G$, $\alpha$, $H$, $\beta$, $\psi^{-1}$, $W$, and~$B$
in place of $H$, $\beta$, $G$, $\alpha$, $\psi$, $B$, and~$W$, respectively.
In fact, (\ref{local-equi}) implies that
$\alpha(W)\sub U$, enabling us to compose the functions in (\ref{local-equi})
with $\psi^{-1}$ on the left. Composing also with $(\psi|_B^W)^{-1}$ on the right,
we find that
\begin{equation}\label{other-way}
\psi^{-1}\circ\alpha|_W=\beta\circ \psi^{-1}|_W.
\end{equation}
We now prove (\ref{B-indu3})
by induction,
starting with the observation that $\psi(B_0)=\psi(B)=W=W_0$.
If (\ref{B-indu2}) holds for some $n\in\N_0$,
let $x\in B_{-(n+1)}$. Then $\psi(x)\in \psi(B)=W$
and $\beta^j(\beta(x))=\beta^{j+1}(x)\in B$ for $j\in\{0,1,\ldots, n\}$
shows that $\beta(x)\in B_{-n}$, whence $\alpha(\psi(x))=\psi(\beta(x))\in W_{-n}$ by induction.
Hence $\psi(x)\in \{w\in W\colon \alpha(w)\in W_{-n}\}=W_{-(n+1)}$.\vspace{2mm}
\end{proof}
We are now ready to calculate the scale and find tidy subgroups
for endomorphisms of $p$-adic Lie groups.
It is illuminating to look at this easier case first,
before we turn to endomorphisms of Lie groups
over general local fields.
Of course, the $p$-adic case is subsumed by the later
discussion, but the latter is more technical
as techniques from
dynamical systems (local invariant manifolds)
will be used as a replacement for
the exponential function,
which provides a local conjugacy between
the linear dynamical system $(L(G),L(\alpha))$ and $(G,\alpha)$
in the case of an endomorphism $\alpha$ of a $p$-adic Lie group~$G$,
and thus enables a more elementary reasoning.\\[3mm]
\emph{Preparations.} If $G$ is a $p$-adic Lie group and $\alpha\colon G\to G$
an endomorphism, then there exists an open subgroup
$V$ of $(L(G),+)$ which is a BCH-Lie group with BCH-multiplication~$*$,
and an exponential function
$\exp_G\colon V\to U$ which is an isomorphism from
the Lie group $(V,*)$ onto a compact open subgroup~$U$ of~$G$, as recalled above.
Fix a norm $\|.\|$ on $\cg:=L(G)$ which is adapted to~$L(\alpha)$;
after shrinking~$V$, we may assume that
\begin{equation}\label{gives-detail}
V=B^\cg_R(0)
\end{equation}
for some $R>0$. Abbreviate $B_t:=B^\cg_t:=B^\cg_t(0)$ for $t>0$.
Applying \ref{recall-nat} and the Ultrametric Inverse Function Theorem (Lemma~\ref{invfct})
to $\phi:=(\exp_G)^{-1}\colon U\to V$,
we find $r\in\;]0,R]$ such that
$B^\phi_t:=\phi^{-1}(B_t)=\exp_G(B_t)$ is a compact open
subgroup of~$G$ for all $t\in\;]0,r]$ and, moreover,
\begin{equation}\label{for-p-1}
L(\alpha)(B_r)\sub V\quad \mbox{and}\quad
\exp_G\circ L(\alpha)|_{B_r}=\alpha\circ \exp_G|_{B_r},
\end{equation}
whence $\alpha(B^\phi_r)\sub U$ in particular.
Let $\cg_{<1}:=\bigoplus_{\rho\in[0,1[}\cg_\rho$
be the indicated sum of characteristic subspaces with respect to $L(\alpha)$,
and $\cg_{\geq 1}:=\bigoplus_{\rho\geq 1}\cg_\rho$.
Since
\[
\cg_{<1}=\; \conp(L(\alpha))\quad\mbox{and}\quad \cg_{\geq 1}=\; \parm(L(\alpha))
\]
are Lie subalgebras of~$\cg$ (see Theorem~\ref{thm-lincase}\,(b)
and Lemma~\ref{lin-lia}) and~$*$ is given by the BCH-series,
we see that
\[
B^{\cg_{<1}}_r:=B_r\cap\cg_{<1}\quad\mbox{and}\quad
B^{\cg_{\geq 1}}_r:=B_r\cap\cg_{\geq 1}
\]
are Lie subgroups
of $(B_r,*)$ with Lie algebras $\cg_{<1}$ and $\cg_{\geq 1}$,
respectively.
After shrinking~$r$ if necessary, we may assume that
\begin{equation}\label{geqsamecos}
x*B^{\cg_{\geq 1}}_t=x+B^{\cg_{\geq 1}}_t\quad
\mbox{for all $x\in B^{\cg_{\geq 1}}_r$ and $t\in\;]0,r]$,}
\end{equation}
see Remark~\ref{smidx}
(which applies with $\cg_{\geq 1}$ in place of~$\cg$
and $\id\colon B^{\cg_{\geq 1}}_r\to B^{\cg_{\geq 1}}_r$ in place of~$\phi$).
Now the mapping
\[
B^{\cg_{<1}}_r \times B^{\cg_{\geq 1}}_r \to B_r,\quad (x,y)\mto x*y
\]
has the derivative
\begin{equation}\label{makes-iso}
\cg_{<1}\times\cg_{\geq 1}\to \cg,\quad (x,y)\mto x+y
\end{equation}
at $(0,0)$, which is an isometry if we endow $\cg_{<1}$ and $\cg_{\geq 1}$
with the norm induced by~$\|.\|$ and use the maximum norm thereof
on the left-hand side of~(\ref{makes-iso}).
Hence, by the Ultrametric Inverse Function Theorem (Lemma~\ref{invfct}),
after shrinking~$r$ (if necessary) we may assume that
\begin{equation}\label{gives-TA}
B^{\cg_{<1}}_t*B^{\cg_{\geq 1}}_t=B_t\quad\mbox{for all $t\in\;]0,r]$.}
\end{equation}
With notation as before, we have:
\begin{thm}
If $\alpha$ is an endomorphism of a $p$-adic Lie group~$G$,
then
\[
s_G(\alpha)=s_{L(G)}(L(\alpha))
\]
holds
and $B^\phi_t$ is tidy for~$\alpha$, for all $t\in \;]0,r]$.
\end{thm}
\begin{proof}
Let $t\in\;]0,r]$.
Applying the isomorphism $\exp_G\colon (V,*)\to U$ to both sides of~(\ref{gives-TA}),
we see that
\begin{equation}\label{prod-ball}
\exp_G(B^{\cg_{<1}}_t)\exp_G(B^{\cg_{\geq 1}}_t)=B^\phi_t.
\end{equation}
In view of (\ref{for-p-1}), we can apply Lemma~\ref{local-conju}
to $G$, $\alpha$, $H:=(\cg,+)$, $\beta:=L(\alpha)$, $\phi$, $B:=B_t^\cg$, and $W:=B_t^\phi$.
Hence
\[
(B^\phi_t)_+:=(B^\phi_t)_{+,\alpha}=\exp_G((B^\cg_t)_{+,\beta})\;\;
\mbox{and}\;\;
(B^\phi_t)_-:=(B^\phi_t)_{-,\alpha}=\exp_G((B^\cg_t)_{-,\beta}).
\]
Now
\[
(B^\cg_t)_{+,\beta}=B^{\cg_{\geq 1}}_t\quad\mbox{and}\quad
(B^\cg_t)_{-,\beta}=B^\cg_t\cap \bigoplus_{\rho\in[0,1]}\cg_\rho\supseteq B^{\cg_{<1}}_t,
\]
using~(\ref{Upluslin}), whence
\[
(B^\phi_t)_+=\exp_G(B^{\cg_{\geq 1}}_t)\quad\mbox{and}\quad
(B^\phi_t)_-\supseteq \exp_G(B^{\cg_{<1}}_t).
\]
Combining this with (\ref{prod-ball}), we find that
\[
B^\phi_t\supseteq (B^\phi_t)_+(B^\phi_t)_-\supseteq \exp_G(B^{\cg_{\leq 1}}_t)
\exp_G(B^{\cg_{<1}}_t)=B^\phi_t
\]
and thus $B^\phi_t = (B^\phi_t)_+(B^\phi_t)_-$, i.e.,
$B^\phi_t$ is tidy above for~$\alpha$ and thus tidy, by Corollary~\ref{autotb}.
Note that $L(\alpha)|_{B^\cg_t}\colon (B^\cg_t,*)\to (V,*)$
is a group homomorphism, as~$*$ is given by the BCH-series.
Hence $\beta((B^\cg_t)_+)=L(\alpha)(B^{\cg_{\geq 1}}_t)$ is a subgroup of the group $(V\cap\cg_{\geq 1},*)$,
which contains $(B^\cg_t)_+=B^{\cg_{\geq 1}}_t$ as a subgroup.
Since $\exp_G\colon (V,*)\to U$
is an isomorphism of groups
and cosets of balls coincide in the groups $(V\cap\cg_{\geq 1},+)$
and $(V\cap\cg_{\geq 1},*)$ (see (\ref{geqsamecos})), we obtain
\begin{eqnarray*}
s(L(\alpha)) &=& [L(\alpha)B^{\cg_{\geq 1}}_t \colon  B^{\cg_{\geq 1}}_t]\quad
\mbox{w.r.t.\ $+$}\\
 &=& [L(\alpha)B^{\cg_{\geq 1}}_t \colon  B^{\cg_{\geq 1}}_t]\quad
\mbox{w.r.t.\ $\,*$}\\
&=&
[\exp_G(L(\alpha)(B^{\cg_{\leq 1}}_t)) \colon  \exp_G(B^{\cg_{\leq 1}}_t)]\\
&=& [\alpha(\exp_G(B^{\cg_{\leq 1}}_t)) \colon  \exp_G(B^{\cg_{\leq 1}}_t)]\\
&=&[\alpha((B^\phi_t)_+) \colon  (B^\phi_t)_+]=s(\alpha),
\end{eqnarray*}
which completes the proof.
\end{proof}
\begin{rem}
For \emph{automorphisms} of $p$-adic Lie groups,
the calculation of the scale was performed in~\cite{SCA}.
\end{rem}
\section{Invariant manifolds around fixed points}\label{sec-inv-mfd}
As in the classical real case,
(locally) invariant manifolds can be constructed around
fixed points of time-discrete analytic dynamical
systems over a totally disconnected local field (see \cite{Exp} and \cite{Fin}).
We shall use these as a tool in our discussion of analytic
endomorphisms of Lie groups over such fields.
In the current section,
we compile the required background.
\begin{defn}\label{def-hyperb}
Let $E$ be a finite-dimensional
vector space over a totally disconnected local field~$\K$,
which we endow with
its natural absolute value~$|.|_\K$.
Given $a\in \;]0,\infty]$, we call
\[
E_{<a}:=\bigoplus_{\rho\in[0,a[}E_\rho\quad\mbox{and}\quad
E_{>a}:=\bigoplus_{\rho\in\;]a,\infty[}E_\rho
\]
the \emph{$a$-stable} and \emph{$a$-unstable}
vector subspaces of~$E$ with respect to~$\alpha$,
using the characteristic subspaces~$E_\rho$ with respect to~$\alpha$
(as in \ref{splitornot}). We call $E_1$ (i.e., $E_\rho$ with $\rho=1$)
the \emph{centre subspace} of~$E$ with respect to~$\alpha$.
A linear endomorphism $\alpha$ of~$E$
is called \emph{$a$-hyperbolic}
if $a\not=|\lambda|_\K$
for all eigenvalues~$\lambda$ of~$\alpha$ in an algebraic closure~$\wb{\K}$,
i.e., if $E_a=\{0\}$ and thus
\[
E=E_{<a}\oplus E_{>a}.
\]
\end{defn}
Now consider an analytic manifold~$M$ over a local field~$\K$,
an analytic mapping $f\colon M\to M$,
a fixed point $p\in M$ of~$f$ and a submanifold $N\sub M$
such that $p\in N$. Given $a>0$, decompose
\[
T_p(M)=T_p(M)_{<a}\oplus T_p(M)_a\oplus T_p(M)_{>a}
\]
with respect to the endomorphism $T_p(f)$ of $T_p(M)$,
as in Definition~\ref{def-hyperb}.
For our purposes,
special cases of concepts in~\cite{Exp} and~\cite{Fin}
are sufficient:
\begin{defn}\label{def-loc-inv}
\begin{itemize}
\item[(a)]
If $a\in\;]0,1]$ and $T_p(f)$ is $a$-hyperbolic, we say that the submanifold~$N$ is a \emph{local $a$-stable manifold} for~$\alpha$ around~$p$
if $T_p(N)=T_p(M)_{<a}$ and $f(N)\sub N$.
\item[(b)]
$N$ is called a \emph{centre manifold} for~$\alpha$ around~$p$
if $T_p(N)=T_p(M)_1$ and $f(N)=N$.
\item[(c)]
If $b\geq 1$ and $T_p(f)$ is $b$-hyperbolic,
we say that $N$ is a \emph{local $b$-unstable manifold} for~$\alpha$ around~$p$
if $T_p(N)=T_p(M)_{>b}$ and there exists an open neighbourhood
$P$ of~$p$ in~$N$ such that $\alpha(P)\sub N$.
\end{itemize}
\end{defn}
We need a fact concerning the existence of local invariant manifolds.
\begin{prop}\label{inv-mfd-thm}
Let $M$ be an analytic manifold over a totally disconnected local field~$\K$.
Let $f\colon M\to M$ be an analytic mapping and $p\in M$ be a fixed point of~$f$.
Moreover, let $a\in \;]0,1]$ and $b\in [1,\infty[$ be
such that $a\not=|\lambda|_\K$ and $b\not=|\lambda|_\K$
for all eigenvalues~$\lambda$ of~$\alpha$ in an algebraic closure~$\wb{\K}$ of~$\K$.
Finally, let $\|.\|$ be a norm on $E:=T_p(M)$ which is adapted to the
endomorphism $T_p(f)$. Endow vector subspaces $F\sub E$ with the norm
induced by~$\|.\|$ and abreviate $B^F_t:=B^F_t(0)$ for $t>0$.
Then the following holds:
\begin{itemize}
\item[\rm(a)]
There exists a local $a$-stable manifold~$W^s_a$ for~$\alpha$ around~$p$
and an analytic diffeomorphism
\[
\phi_s\colon W^s_a\to B^{E_{<a}}_R
\]
for some $R>0$
such that $\phi_s(p)=0$ holds, $W^s_a(t):=\phi_s^{-1}(B^{E_{<a}}_t)$ is a local
$a$-stable manifold for $\alpha$ around~$p$ for
all
$t\in \;]0,R]$,
and $d\phi_s|_{T_p(N)}=\id_{E_{<a}}$.
\item[\rm(b)]
There exists a centre manifold~$W^c$ for~$\alpha$ around~$p$
and an analytic diffeomorphism
\[
\phi_c\colon W^c\to B^{E_1}_R
\]
for some $R>0$
such that $\phi_c(p)=0$ holds, $W^c(t):=\phi_c^{-1}(B^{E_1}_t)$ is
a centre manifold for $\alpha$ around~$p$ for all $t\in \;]0,R]$,
and $d\phi_c|_{T_p(N)}=\id_{E_1}$.
\item[\rm(c)]
There exists a local $b$-unstable manifold~$W^u_b$ for~$\alpha$ around~$p$
and an analytic diffeomorphism
\[
\phi_u \colon W^u_b \to B^{E_{>b}}_R
\]
for some $R>0$
such that $\phi_u(p)=0$, $\;W^u_b(t):=\phi_u^{-1}(B^{E_{>b}}_t)$ is a local
$b$-unstable manifold for $\alpha$ around~$p$ for all $t\in \;]0,R]$,
and $d\phi_u|_{T_p(N)}=\id_{E_{>b}}$.   
\end{itemize}
\end{prop}
\begin{proof}
(a) and (c) are covered by the Local Invariant Manifold Theorem
(see \cite[p.\,76]{Fin}) and its proof. To get~(b),
let $\phi\colon U\to V$ be an analytic diffeomorphism from on open neighbourhood~$U$
of~$p$ in~$M$ onto an open $0$-neighbourhood $V\sub E$
such that $\phi(p)=0$ and $d\phi|_E=\id_E$.
We can then construct centre manifolds for the analytic map
\[
\phi\circ f\circ \phi^{-1}\colon \phi(U\cap f^{-1}(U))\to V
\]
around its fixed point~$0$ with \cite[Proposition~4.2]{Exp}
and apply $\phi^{-1}$ to create the desired centre manifolds for~$f$.
We mention that the cited proposition only considers mappings
whose derivative at the fixed point is an automorphism,
but its proof never uses this hypothesis,
which therefore can be omitted.
\end{proof}
\begin{rem}\label{better-1}
Of course, we can use the same $R>0$ in parts (a), (b), and~(c)
of Proposition~\ref{inv-mfd-thm}
(simply take the minimum of the three numbers).\vspace{1mm}
\end{rem}
\begin{rem}
Note that, since $f(W^s_a)\sub W^s_a$, we have a descending sequence
\[
W^s_a\supseteq f(W^s_a)\supseteq f^2(W^s_a)\supseteq\cdots
\]
in Proposition~\ref{inv-mfd-thm}\,(a).
\end{rem}
\begin{la}\label{better-2}
After shrinking~$R$ in Proposition~{\rm\ref{inv-mfd-thm}\,(a)} if necessary,
we can assume that
\begin{equation}\label{inter-desc}
\bigcap_{n\in\N_0} f^n(W^s_a)=\{p\}\quad\mbox{and}\quad
\lim_{n\to\infty} f^n(x)=p\;\;  \mbox{for all $\,x\in W^s_a$.}
\vspace{-1mm}
\end{equation}
\end{la}
\begin{proof}
Abbreviate $F:=E_{<a}$. The map
\[
h:=\phi_s\circ f|_{W^s_a}\circ\phi_s^{-1}\colon B^F_R\to B^F_R
\]
is analytic, $h(0)=0$, and $h'(0)=T_p(f)|_F$
has operator norm $\|h'(0)\|_{\op}<a$.
Choose $\ve>0$ so small that
\[
\theta:=\|h'(0)\|_{\op}+\ve<1.
\]
Since $h$ is totally dfferentiable at~$0$, we find
$r\in\;]0,R]$ such that
\[
\|h(x)-h'(0)(x)\|\leq \ve\|x\|\quad\mbox{for all $x\in B^F_r$.}
\]
Then $\|h(x)\|=\|h'(0)(x)+(h(x)-h'(0)(x))\|\leq (\|h'(0)\|_{\op}+\ve)\|x\|=\theta\|x\|$
for all $x\in B^F_r$, whence $h(B^F_r)\sub B^F_r$ and
\[
h^n(B^F_r )\sub B^F_{\theta^n r}.
\]
As a consequence, $\bigcap_{n\in\N_0} h^n(B^F_r)=\{0\}$.
Then $Q:=W^s_a(r)=\phi_s^{-1}(B^F_r)$ is an open neighbourhood
of~$p$ in $W^s_a$ such that $\bigcap_{n\in\N_0}\alpha^n(Q)=\{p\}$.
After replacing~$R$ with~$r$, we have~(\ref{inter-desc}).
\end{proof}
\begin{la}\label{better-2-3}
After shrinking~$R$ in Proposition~{\rm\ref{inv-mfd-thm}\,(b)},
we may assume that $f|_{W^c(t)}\colon W^c(t)\to W^c(t)$
is an analytic diffeomorphism for each $t\in\;]0,R]$.
\end{la}
\begin{proof}
Abbreviate $F:=T_p(M)_1=E_1$. The mapping
\[
h:=\phi_c \circ f|_{W^c}\circ\phi_c^{-1}\colon B^F_R\to B^F_R
\]
is analytic with $h(0)=0$, and $h'(0)=T_p(f)|_F$
is an isometry. By the Ultrametric Inverse Function Theorem,
after shrinking~$R$ if necessary, we can achieve that $h$ is
an analytic diffeomorphism from $B^F_R$ onto~$B^F_R$
and an isometry. Since $W^c(t)$ is a centre manifold for all $t\in\;]0,R]$,
we have $f(W^c(t))=W^c(t)$, which completes the proof.
\end{proof}
\begin{la}\label{better-3}
We can always choose the open neighbourhood $P$ around~$p$
in a local $b$-unstable manifold $N\sub M$ $($as in Definition~{\rm\ref{def-loc-inv}\,(c))}
in such a way that, for each $x\in P\setminus\{p\}$,
there exists $n\in\N$ such that $x,f(x),\ldots,f^{n-1}(x)\in P$ but
$f^n(x)\in N\setminus P$.
\end{la}
\begin{proof}
To see this, excluding a trivial case,\footnote{Otherwise~$N$ is discrete
and we can choose $P=\{p\}$.}
we may assume that the $b$-unstable
subspace $F:=E_{>b}:=T_p(M)_{>b}$ with respect to $T_p(f)$ is
non-trivial. Let $\phi\colon U\to V$ be an analytic diffeomorphism
from an open neighbourhood~$U$ of $p$ in~$N$ onto
an open $0$-neighbourhood $V\sub T_p(N)=F$, such that
$\phi(p)=0$ and $d\phi|_{T_p(N)}=\id_F$.
Then
\[
h:=\phi\circ f\circ \phi^{-1}\colon \phi(f^{-1}(U)\cap U)\to V
\]
is an analytic mapping defined on an open $0$-neighbourhood,
such that $h'(0)=T_p(f)|_F$ is invertible and
\[
\frac{1}{\|h'(0)^{-1}\|_{\op}\!\!\!\!}\, >b.
\]
Since $h$ is totally differentiable at~$0$,
there exists $r>0$ with $B^F_r(0)$ in the domain~$D$ of~$h$ such that
\[
h(B^F_r(0))\sub D,
\]
\begin{equation}\label{equ-star}
\|h(x)-h'(0)(x)\|\leq b\|x\|
\end{equation}
for all $x\in B^F_r(0)\setminus\{0\}$,
and $f(P)\sub U$ with $P:=\phi^{-1}(B^F_r(0))$. Using~(\ref{equ-star}) and~(\ref{winner}),
we deduce that
\[
\|h(x)\|=\|h'(0)(x)+(h(x)-h'(0)(x))\|=\|h'(0)(x)\|> b\|x\|,
\]
as $\|h'(0)(x)\|\geq \|h'(0)^{-1}\|_{\op}^{-1}\|x\|>b\|x\|$.
So, for all $x\in B^F_r(0)\setminus \{0\}$,
there is $n\in\N$ such that $x,h(x),\ldots, h^n(x)$ are defined
and in $B^F_r(0)$, but $h^{n+1}(x)\in D\setminus  B^F_r(0)$.
Now~$P$ is a neighbourhood of~$p$
% in~$N$
with the desired property.
\end{proof}
\begin{la}\label{better-4}
Let $U$ be an open neighbourhood of~$p$ in~$M$
and $\phi \colon U\to V$ be an analytic diffeomorphism onto an open $0$-neighbourhood
$V\sub T_p(M)=:E$ such that $\phi(0)=0$ and $d\phi|_E=\id_E$.
After decreasing $R$ in Proposition~{\rm\ref{inv-mfd-thm}\,(c)} if necessary,
we can always assume that $B^E_R(0)\sub V$ and
the following additional property holds  for all $t\in\;]0,R]$:\\[2.3mm]
$W^u_b(t)$ is the set of all $x\in \phi^{-1}(B^E_t(0))=:B^\phi_t$
for which there exists an $f$-regressive trajectory $(x_{-n})_{n\in\N_0}$
in~$B^\phi_t$ with $x_0=x$, such that
\begin{equation}\label{atleasttwice}
\lim_{n\to\infty}\phi(x_{-n})=0.\vspace{-.3mm}
\end{equation}
Then ${\displaystyle\lim_{n\to\infty} x_{-n}=p}$ in particular, and $x_{-n}\in W^u_b(t)$
for all $n\in\N_0$.\vspace{-1mm}
\end{la}
\begin{proof}
As before, abbreviate $B^E_t:=B^E_t(0)$ and $B^{E_{>b}}_t:=B^{E_{>b}}_t(0)$
for $t>0$.
There is $r>0$ such that $B^E_r\sub V$ and $f(\phi^{-1}(B^E_r))\sub U$,
whence an analytic map
\[
h:=\phi\circ f\circ\phi^{-1}|_{B^E_r}\colon B^E_r\to E
\]
can be defined with $h(0)=0$ and $h'(0)=T_p(f)$.
For $t\in\;]0,r]$, let $\Gamma_t$ be the set of all
$z\in B^E_t$ for which there exists an $h$-regressive trajectory $(z_{-n})_{n\in\N_0}$
in $B^E_t$ with $z_0=z$ such that
\[
\lim_{n\to\infty} b^n \|z_n\|=0.
\]
By \cite[Theorem~B.2]{Exp}
and the proof of Theorem~8.3 in~\cite{Exp}, after shrinking $r$ we may assume that
$\Gamma_t$ is a submanifold of $B^E_t$ and
\[
\phi^{-1}(\Gamma_t)
\]
a local $b$-unstable submanifold of~$G$ for each $t\in \;]0,r]$;
and, moreover, there is an analytic map
\[
\mu \colon B^{E_{>b}}_r\to B^{E_{<b}}_r
\]
(called $\phi$ there) with $\mu(0)=0$ and
\begin{equation}\label{goodfor}
\mu'(0)=0
\end{equation}
such that
\[
\Gamma_t=\{(\mu(y),y)\colon y\in B^{E_{>b}}_t\}\quad\mbox{for all $t\in\;]0,r]$,}
\]
identifying $E=E_{<b}\oplus E_{>b}$ with $E_{<b}\times E_{>b}$.
Hence
\[
\nu \colon B^{E_{>b}}_r \to\phi^{-1}(\Gamma_r),\quad y\mto \phi^{-1}(\mu(y),y)
\]
is an analytic diffeomorphism,
and thus also $\nu^{-1}\colon \nu(B^{E_{>b}}_r)\to B^{E_{>b}}_r$
is an analytic diffeomorphism. As a consequence of~(\ref{goodfor}),
we have
\[
d(\nu^{-1})|_{E_{>b}}=\, \id_{E_{>b}}.\vspace{-.3mm}
\]
Since, like $W^u_b=\phi^{-1}_u(B^{E_{>b}}_R)$,
also $\phi^{-1}(\Gamma_r)=\nu(B^{E_{>b}}_r)$ is a local $b$-unstable
manifold for~$f$, \cite[Theorem~8.3]{Exp}
shows that there exists a subset $Q\sub W^u_b\cap   \nu(B^{E_{>b}}_r)$
which is an open neighbourhood of~$p$
in both~$W^u_b$ and $\nu(B^{E_{>b}}_r)$.
Hence, there exists $\tau>0$ with $\tau\leq\min\{R,r\}$ such that
$W^u_b(\tau)\sub Q$ and $\nu(B^{E_{>b}}_\tau)\sub Q$.
Since~$Q$ is a submanifold of~$M$, the manifold
structures induced on~$Q$ as an open subset of $W^u_b$ and $\nu(B^{E_{>b}}_r)$
coincide. By Lemma~\ref{ball-unq}, after shrinking~$\tau$ if necessary
we may assume that
\[
W^u_b(t)=\phi^{-1}_u(B^{E_{>b}}_t)=(\nu^{-1})^{-1}(B^{E_{>b}}_t)=
\phi^{-1}(\Gamma_t)
\]
for all $t\in\;]0,\tau]$. Let $t\in\;]0,\tau]$ and $x\in B^\phi_t$.\\[2.3mm]
If $x\in \phi^{-1}(\Gamma_t)$,
then there exists an $h$-regressive trajectory $(z_{-n})_{n\in\N_0}$
in~$B^E_t$ with $z_0=\phi(x)$ and $b^n\|z_{-n}\|\to 0$.
Now, for $m\in\N_0$, the sequence
$(z_{-n-m})_{n\in\N_0}$ is an $h$-regressive trajectory for~$z_{-m}$
in $B^E_t$ such that
\[
b^n\|z_{-n-m}\|=b^{-m}b^{n+m}\|z_{-n-m}\|\to 0
\]
as $n\to\infty$ and thus $z_m\in\Gamma_t$.
As a consequence,
$(\phi^{-1}(z_{-n}))_{n\in\N_0}$
is an $f$-regressive trajectory in $W^u_b(t)=\phi^{-1}(\Gamma_t)$
such that $\phi^{-1}(z_0)=x$ and $b^n\|\phi(\phi^{-1}(z_{-n}))\|=b^n\|z_{-n}\|\to 0$
as $n\to\infty$.\\[2.3mm]
Conversely, assume there exists an $f$-regressive trajectory $(x_{-n})_{n\in\N_0}$
in $B^\phi_t$ with $x_0=x$ and~(\ref{atleasttwice}).
Then $(\phi(x_{-n}))_{n\in\N_0}$ is an $h$-regressive trajectory in~$B^E_t$
such that~(\ref{atleasttwice}) holds, whence $\phi(x)=\phi(x_0)\in\Gamma_t$
and thus $x\in\phi^{-1}(\Gamma_t)$.\\[2.3mm]
Summing up, the conclusion of the lemma holds if we replace~$R$ with~$\tau$.
\end{proof}
Let us consider a first application of invariant manifolds.
\begin{prop}\label{charviaeigen}
Let $\alpha$ be an analytic automorphism of a Lie group~$G$
over a totally disconnected local field~$\K$. Let $\wb{\K}$ be an algebraic closure of~$\K$.
Then the following holds:
\begin{itemize}
\item[\rm(a)]
$\conp(\alpha)$ is open in~$G$ if and only if $|\lambda|_\K<1$
for each eigenvalue~$\lambda$ of~$L(\alpha)$ in~$\wb{\K}$.
\item[\rm(b)]
$\alpha$ is expansive if and only if $|\lambda|_\K\not=1$
for each eigenvalue~$\lambda$ of $L(\alpha)$ in~$\wb{\K}$.
\item[\rm(c)]
$\alpha$ is a distal automorphism if and only if $|\lambda|_\K=1$
for each eigenvalue~$\lambda$ of $L(\alpha)$ in~$\wb{\K}$.
\end{itemize}
\end{prop}
\begin{proof}
(c) If $|\lambda|_\K<1$ for some $\lambda$,
choose $a\in\;]0,1]$ such that $a>|\lambda|_\K$
and $L(\alpha)$ is $a$-hyperbolic.
Then $G$ has a local $a$-stable manifold $W\not=\{e\}$,
which can be chosen such that $W\sub\; \conp(\alpha)$
(see Lemma~\ref{better-2}).
Since $\alpha^n(x)\to e$ for all $x\in W\setminus\{e\}$,
we see that~$\alpha$ is not distal.\\[2.3mm]
If $|\lambda|_\K>1$ for some~$\lambda$,
then again we see that~$\alpha$ is not distal,
replacing~$\alpha$ with $\alpha^{-1}$ and its iterates
in the preceding argument.\\[2.3mm]
If $|\lambda|_\K=1$ for each $\lambda$, then $\cg:=L(G)$ coincides
with its centre subspace, whence every centre manfold for~$\alpha$ around~$e$
is open in~$G$. If $x\in G\setminus\{e\}$,
then Proposition~\ref{inv-mfd-thm}\,(b)
provides a centre manifold~$W$ for~$\alpha$ around~$e$ such that $x\not\in W$.
Since $\alpha^n(W)=W$ for all $n\in\Z$
and $\alpha^n$ is a bijection, we must have $\alpha^n(x)\not\in W$
for all $n\in\Z$. As a consequence, the set $\{\alpha^n(x)\colon n\in\Z\}$
(and hence also its closure) is contained in the closed set $G\setminus W$.
Thus $e\not\in\overline{\{\alpha^n(x)\colon n\in \Z\}}$ and thus~$\alpha$
is distal.

The proofs for (b) and the implication ``$\impl$'' in~(a) are similar and again involve
local invariant manifolds, see \cite[Proposition~7.1]{GaR} and \cite[Corollary~6.1 and Proposition~3.5]{Fin},
respectively.

(a) To complete the proof of~(a), assume that $|\lambda|_\K<1$ for all~$\lambda$.
Choose $a\in\;]0,1[$ such that $a>|\lambda|_\K$ for all~$\lambda$.
Then $\cg=\cg_{<a}$ with respect to~$L(\alpha)$.
Let $W^s_a$ and the analytic diffeomorphism $\phi_s\colon W^s_a\to B^\cg_R(0)$
be as in Proposition~\ref{inv-mfd-thm}\,(a). Then $W^s_a$ is open in~$G$.
By Lemma~\ref{better-2}, after shrinking~$R$ (if necessary) we can achieve
that $W^s_a\sub\;\conp(\alpha)$.
Thus $\conp(\alpha)$ is an open identity neighbourhood in~$G$
and hence $\conp(\alpha)$ is open, being a subgroup.
\end{proof}
\section{Endomorphisms of Lie groups over {\boldmath$\K$}}\label{sec-endo}
In this section, we formulate and prove our main results concerning
analytic endomorphisms of Lie groups over totally disconnected local fields.
\subsection*{{\normalsize Some preparations}}
\begin{defn}
Let $\alpha$ be an endomorphism of a totally disconnected, locally compact
group~$G$. We say that $G$ has \emph{small tidy subgroups}
for~$\alpha$
if each identity neighbourhood of~$G$ contains
a compact open subgroup of~$G$ which is tidy for~$\alpha$.
\end{defn}
For $\alpha$ an automorphism, the existence of small tidy
subgroups is equivalent to closedness of $\conp\!(\alpha)$
(see \cite[Theorem~3.32]{BaW} for the case of metrizable groups;
the general case can be deduced with arguments from~\cite{Jaw}).
The following result concerning endomorphisms
is sufficient for our Lie theoretic applications.
\begin{la}\label{tidyclo}
Let $\alpha$ be an endomorphsm of a totally disconnected, locally compact group~$G$.
\begin{itemize}
\item[\rm(a)]
If $G$ has small subgroups tidy for $\alpha$, then
$\conp(\alpha)$ is closed.
\item[\rm(b)]
If $\conp(\alpha)$ is closed and a compact open subgroup $U$ of~$G$ satisfies
\begin{equation}\label{unnec}
U_-=(\conp(\alpha)\cap U_-)(U_+\cap U_-),
\end{equation}
then $U$ is tidy below for~$\alpha$. Hence, if $\conp(\alpha)$ is closed and
each identity neighbourhood of~$G$
contains a compact open subgroup $U$ which satisfies~{\rm(\ref{unnec})}
an is tidy above for~$\alpha$, then $G$ has small tidy subgroups.
\end{itemize}
\end{la}
\begin{proof}
(a) Let $\cT(\alpha)$ be the set of all tidy subgroups for~$\alpha$.
If  $\cT(\alpha)$ is a basis of identity neighbourhoods, then
\begin{eqnarray*}
\conp(\alpha)&=& \{x\in G\colon \lim_{n\to\infty}\alpha^n(x)=e\}\\[1mm]
&=&\{x\in G\colon (\forall U\in\cT(\alpha))\underbrace{(\exists m)(\forall n\geq m)\;\,\alpha^n(x)\in U}_{\Leftrightarrow x\in U_{--}}\}\\[-4mm]
&=& \bigcap_{U\in\cT(\alpha)}U_{--},
\end{eqnarray*}
which is closed.

(b) Assuming that $U_-=(\conp(\alpha)\cap U_-)(U_+\cap U_-)$,
let us show that
\begin{equation}\label{thusclose}
U_{--}=\; \conp(\alpha) (U_+\cap U_-).
\end{equation}
If (\ref{thusclose}) holds, then $U_{--}$ is closed (as $\conp(\alpha)$ is assumed
closed and $U_+\cap U_-$ is compact). Hence~$U$ will be tidy for~$\alpha$
(by \ref{variant-tidy}), and also the final assertion is then immediate.\\[2.3mm]
The inclusion ``$\supseteq$"  in (\ref{thusclose}) is clear.
To see that the converse inclusion holds, let
$x\in U_{--}$. Then $\alpha^n(x)\in U_-$ for some~$n$.
As $U_-=(\conp(\alpha)\cap U_-)(U_+\cap U_-)$ by hypothesis, we have
\[
\alpha^n(x)=yz \quad\mbox{for some $y\in \;\conp(\alpha)$ and some $z\in U_+\cap U_-$.}
\]
Since $\alpha(U_+\cap U_-)=U_+\cap U_-$, find $w\in U_+\cap U_-$ such that $z=\alpha^n(w)$.
Then $\alpha^n(xw^{-1})=y\in \,\conp(\alpha)$, whence $xw^{-1}\in \,\conp(\alpha)$
and thus $x=(xw^{-1})w\in \,\conp(\alpha) (U_+\cap U_-)$. Thus $U_{--}\sub \;\conp(\alpha) (U_+\cap U_-)$; the proof is complete.
\end{proof}
\begin{rem}
With much more effort, it can be shown that
closedness of $\conp(\alpha)$ is always equivalent
to the existence of small tidy subgroups,
for every endomorphism~$\alpha$ of a totally disconnected,
locally compact group~$G$ (see \cite[Theorem~D]{BGT}).
Lemma~\ref{tidyclo}, which is sufficient for our ends,
was presented at the AMSI workshop July 25, 2016 (before the cited
theorem was known).
\end{rem}
We need a result from the structure theory of totally disconnected groups.
\begin{la}\label{divisi}
Let $\alpha$ be an endomorphism of a totally disconnected, locally compact group~$G$
and
$V\sub G$ be a compact open subgroup which
is tidy above for~$\alpha$.
Then
$s(\alpha)$ divides
$[\alpha(V):\alpha(V)\cap V]$.\,\Punkt
\end{la}
\begin{proof}
As in \cite[Definition 5]{END}, let $\cL_V$ be the subgroup of all
$x\in G$ for which there exist $y\in V_+$ and $n,m\in\N_0$
such that $\alpha^m(y)=x$ and $\alpha^n(y)\in V_-$.
Let $L_V$ be the closure of $\cL_V$ in~$G$,
\[
\wt{V}:=\{x\in V\colon xL_V\sub L_V V\}
\]
(as in \cite[(7)]{END}) and $W:=\wt{V}L_V$.
Then $\wt{V}$ is a compact open subgroup
of~$G$ which is tidy above for~$\alpha$ and
\begin{equation}\label{equay}
[\alpha(V):\alpha(V)\cap V]=[\alpha(\wt{V}):\alpha(\wt{V})\cap\wt{V}]
\end{equation}
(see \cite[Lemma~16]{END}).
Moreover, $W$ is
a compact open subgroup of~$G$
which is tidy for~$\alpha$,
by the third step of the `tidying procedure' (see \cite[Step 3 following Definition 10]{END}).
Let $W_{-1}:=W\cap\alpha^{-1}(W)$ and $\wt{V}_{-1}:=\wt{V}\cap\alpha^{-1}(\wt{V})$.
Then the left action
\[
\wt{V}_+\times Y\to Y,\quad (v,w(W_+\cap W_{-1}))\mto vw(W_+\cap W_{-1})
\]
of $\wt{V}_+$ on $Y:=W_+/(W_+\cap W_{-1})$ is transitive
(as the map $\phi$ defined in the proof of \cite[Proposition~6\,(4)]{END}
is surjective). The point $W_+\cap W_{-1}\in Y$ has stabilizer
$\wt{V}_+\cap W_+\cap W_{-1}=\wt{V}_+\cap W_{-1}$. Hence
\begin{equation}\label{equay2}
\hspace*{-.2mm}s(\alpha)=[\alpha(W):\alpha(W)\cap W]=[W_+:W_+\cap W_{-1}]=|Y|
=[\wt{V}_+:\wt{V}_+\cap W_{-1}],\!
\end{equation}
using tidiness of~$W$ for the first equality, \cite[Lemma~3 (1) and (4)]{END}
for the second and the orbit formula for the $\wt{V}_+$-action
for the last.
Since $\wt{V}_+\cap W_{-1}$ contains $\wt{V}_+\cap \wt{V}_{-1}$ as a subgroup,
using \cite[Lemma 3]{END} again we deduce that
\begin{eqnarray}
[\alpha(\wt{V}):\alpha(\wt{V})\cap\wt{V}]&=&
[\wt{V}_+:\wt{V}_+\cap \wt{V}_{-1}]\notag \\
&=&[\wt{V}_+:\wt{V}_+\cap W_{-1}][\wt{V}_+\cap W_{-1}:\wt{V}_+\cap\wt{V}_{-1}].\label{equay3}
\end{eqnarray}
Substituting (\ref{equay}) and (\ref{equay2}) into (\ref{equay3}), we obtain
\begin{equation}
[\alpha(V):\alpha(V)\cap V]
=s(\alpha)[\wt{V}_+\cap W_{-1}: \wt{V}_+\cap\wt{V}_{-1}],
\end{equation}
which completes the proof.
\end{proof}
\subsection*{{\normalsize Scale and tidy subgroups}}
If $\alpha\colon G\to G$ is an analytic endomorphism
of a Lie group~$G$ over a local field~$\K$, we fix a norm
$\|.\|$ on its Lie algebra $\cg:=L(G):=T_e(G)$ which is adapted
to the associated
linear endomorphism $L(\alpha):=T_e (\alpha)$ of~$\cg$.
Let $\wb{\K}$ be an algebraic closure of~$\K$.
In the proof of our main result, Theorem~\ref{main},
we want to use Lemma~\ref{displ-balls}
to create compact open subgroups $B^\phi_t$ of~$G$.
To get more control over these subgroups,
we now make a particular choice of~$\phi$.
\begin{numba}\label{phimain}
Pick $a\in \;]0,1]$ such that
$L(\alpha)$ is $a$-hyperbolic and $a>|\lambda|_\K$
for each eigenvalue $\lambda$ of~$L(\alpha)$ in~$\wb{\K}$ such that $|\lambda|_\K<1$.
Pick $b\in [1,\infty[$ such that
$L(\alpha)$ is $b$-hyperbolic and $b<|\lambda|_\K$
for each eigenvalue $\lambda$ of~$L(\alpha)$ in~$\wb{\K}$
such that $|\lambda|_\K>1$.
With respect to the endomorphism~$L(\alpha)$, we then have
\[
\cg_{<1}=\cg_{<a}\quad\mbox{and}\quad \cg_{>1}=\cg_{>b},
\]
entailing that
\begin{equation}\label{decab}
\cg=\cg_{<a}\oplus \cg_1\oplus \cg_{>b}.
\end{equation}
We find it useful to identify $\cg$ with the direct product
$\cg_{<a}\times \cg_1\times \cg_{>b}$; an element $(x,y,z)$
of the latter is identified with $x+y+z\in\cg$.\\[2.3mm]
Let $W^s_a$, $W^c$, and $W^u_b$ be a local $a$-stable manifold, centre-manfold,
and local $b$-unstable manifold for~$\alpha$ around~$P:=e$ in $M:=G$, respectively,
$R>0$ and
\[
\phi_s\colon W^s_a\to B^{\cg_{<a}}_R(0),\quad
\phi_c\colon W^c\to B^{\cg_1}_R(0),
\]
as well as $\phi_u\colon W^u_b\to B^{\cg_{>b}}_R(0)$ be analytic diffeomorphisms
as described in Proposition~\ref{inv-mfd-thm}.
We abbreviate $B^F_t:=B^F_t(0)$ whenever~$F$ is a vector subspace of~$\cg$.
Using the inverse maps
\[
\psi_s:=\phi_s^{-1},\quad \psi_c:=\phi_c^{-1},\quad\mbox{and}\quad
\psi_u:=\phi_u^{-1},
\]
we define the analytic map
\[
\psi\colon B^\cg_R=B^{\cg_{<a}}_R\times B^{\cg_1}_R\times B^{\cg_{>b}}_R\to G,\quad (x,y,z)\mto
\psi_s(x)\psi_c(y)\psi_u(z).
\]
Then $T_0\psi=\id_\cg$ by (\ref{Tay01}) and the properties of $d\phi_s$,
$d\phi_c$, and $d\phi_u$ described in Proposition~\ref{inv-mfd-thm} (a), (b), and (c),
respectively,
if we identify $T_0(\cg)=\{0\}\times\cg$ with $\cg$ as usual,
forgetting the first component.
By the Inverse Functon Theorem,
after shrinking~$R$ if necessary, we may assume
that the image $W^s_aW^cW^u_b$ of $\psi$ is an open identity neighbourhood in~$G$,
and that
\begin{equation}\label{psithe}
\psi\colon B^\cg_R\to W^s_aW^cW^u_b
\end{equation}
is an analytic diffeomorphism. We define
\begin{equation}\label{goodphi}
\phi:=\psi^{-1},
\end{equation}
with domain $U:=W^s_aW^cW^u_b$ and image $V:=B^\cg_R$.
After shrinking~$R$ further if necessary,
we may assume that~$\phi$ and~$R$
have all the properties described in \ref{pre-laballs}
and Lemma~\ref{laballs}.
\end{numba}
\begin{numba}\label{sett-2}
In the following result and its proof,
$\phi$ and $R$ are as in~\ref{phimain}.
We let $B_t:=B^\cg_t \sub \cg$ and the compact open subgroups
\[
B^\phi_t:=\phi^{-1}(B^\cg_t)=\psi(B^\cg_t)
=\psi_s(B^{\cg_{<a}}_r)\psi_c(B^{\cg_1}_t)\psi_u(B^{\cg_{>b}}_t)
=W^s_a(t)W^c(t)W^u_b(t)
\]
of~$G$ for $t\in\;]0,R]$ be as in Lemma~\ref{laballs}
(using notation as in Proposition~\ref{inv-mfd-thm}).
The multiplication $*\colon B_R\times B_R\to B_R$
is as in \ref{pre-laballs}.
\end{numba}
We shrink~$R$ further (if necessary) to achieve the following:
\begin{la}\label{details-WcWu}
After shrinking~$R$, we can achieve that
$W^u_b(t)=\phi_u^{-1}(B^{\cg_{>b}}_t)$
is a subgroup of~$G$ for all $t\in\;]0,R]$ and
$W^c(t)=\phi_c^{-1}(B^{\cg_1}_t)$
normalizes $W^u_b(t)$.
\end{la}
\begin{proof}
Let $\phi$, $R$ and further notation be as in~\ref{phimain} and~\ref{sett-2};
notably, $V=B^\cg_R$.
Using Lemma~\ref{better-4} with $M:=G$, $f:=\alpha$ and $p:=e$, we see that,
after shrinking~$R$ if necessary,
we may assume the following condition ($*$) for all $t\in \;]0,R]$:\\[2.3mm]
\emph{$W^u_b(t)$ is the set of all $x\in B^\phi_t =W^s_a(t)W^c(t)W^u_b(t)$
for which there exists an $\alpha$-regressive trajectory $(x_{-n})_{n\in\N_0}$
in $B^\phi_t$ with $x_0=x$ such that}
\[
\lim_{n\to\infty}b^n\|\phi(x_{-n})\|=0
\]
(and then $x_{-n}\in\phi^{-1}(\Gamma_t)=W^u_b(t)$
for all $n\in\N_0$).
As the analytic map
\[
g\colon \colon V\times V\to V\sub\cg,\quad (x,y)\mto x*y^{-1}
\]
is totally differentiable at~$(0,0)$ with $g(0,0)=0$ and $g'(0,0)(x,y)=x-y$,
after shrinking~$R$ if necessary we may assume that
\[
\|x*y^{-1}-x+y\|\leq \max\{\|x\|,\|y\|\}
\]
for all $x,y\in B_R=V$ and thus
\begin{eqnarray}
\|x*y^{-1}\|&=&\|x-y+(x*y^{-1}-x+y)\|\notag\\
& \leq & \max\{\|x-y\|,\|x*y^{-1}-x+y\|\}
\leq  \max\{\|x\|,\|y\|\},\label{xyinvest}
\end{eqnarray}
using the ultrametric inequality.
It is clear that $e\in W^u_b(t)$. Hence
$W^u_b(t)$ will be a subgroup of~$G$ for all $t\in \;]0,R]$
if we can show that $xy^{-1}\in W^u_b(t)$
for all $x,y\in W^u_b(t)$. Let $(x_{-n})_{n\in\N_0}$ and $(y_{-n})_{n\in\N_0}$
be $\alpha$-regressive trajectories in $B^\phi_t$
such that $x_0=x$, $y_0=y$ and
\[
b^n\|\phi(x_{-n})\|, b^n\|\phi(y_{-n})\|\to 0\quad\mbox{as $n\to\infty$.}
\]
Then $(x_{-n}y_{-n}^{-1})_{n\in\N_0}$ is an $\alpha$-regressive trajectory in the group~$B^\phi_t$
with $x_0y_0^{-1}=xy^{-1}$ and
\[
b^n\|\phi(x_n y_n^{-1})\|=b^n\|\phi(x_n)*\phi(y_n)^{-1}\|\leq \max\{b^n\|\phi(x_n)\|,b^n\|\phi(y_n)\|\}
\to 0,
\]
using that $\phi\colon U\to (B_R,*)$ is a homomorphism
of groups, and using the estimate~(\ref{xyinvest}). Thus $xy^{-1}\in W^u_b$,
by~($*$).\\[2.3mm]
We now show that, after shrinking $R$ if necessary,
$W^u_b(t)$ is normalized by $W^c(R)$
for all $t\in\;]0,R]$.
To this end, consider the analytic map
\[
h\colon V\times V\to V,\quad
(x,y)\mto x*y*x^{-1}.
\]
For $x\in V$, abbreviate $h_x:=h(x,.)$. Since $h_0=\id_V$,
we see that $h_0'(0)=\id_\cg$ which is an isometry.
By the Ultrametric Inverse Function Theorem with Parameters,
after shrinking~$R$ we can achieve that $h_x\colon B^\cg_R\to B^\cg_R$
is an isometry for all $x\in B^\cg_R$. Hence, using that $h_x(0)=0$,
\begin{equation}\label{yetestim}
\|x*y*x^{-1}\|=\|h_x(y)\|=\|y\|\quad\mbox{for all $x,y\in B^\cg_R$.}
\end{equation}
If $x\in W^c(R)$, $t\in\;]0,R]$ and $y\in W^u_b(t)$,
let $(y_{-n})_{n\in\N_0}$ be an $\alpha$-regressive trajectory in~$B^\phi_t$
such that $y_0=y$ and $b^n\|\phi(y_{-n})\|\to 0$ as $n\to\infty$.
Since $\alpha(W^c(R))=W^c(R)$, we can find an $\alpha$-regressive
trajectory $(x_{-n})_{n\in\N_0}$ in $W^c(R)$ such that $x_0=x$.
Recall from Lemma~\ref{laballs}
that~$B^\phi_t$ is a normal subgroup of~$B^\phi_R$.
Hence $(x_{-n}y_{-n}x_{-n}^{-1})_{n\in\N_0}$ is an $\alpha$-regressive
trajectory in~$B^\phi_t$ such that $x_0y_0x_0^{-1}=xyx^{-1}$ and
\[
b^n\|\phi(x_{-n}y_{-n}x_{-n}^{-1})\|=b^n
\|\phi(x_{-n})*\phi(y_{-n})* \phi(x_{-n})^{-1}\|=b^n\|\phi(y_{-n})\|\to 0
\]
as $n\to\infty$, using that $\phi$ is a homomorphism of groups
and~(\ref{yetestim}). Thus $xyx^{-1}\in W^u_b(t)$, by~($*$).
\end{proof}
\begin{numba}\label{has-descen}
By Lemma~\ref{better-2},
after shrinking~$R$ if necessary, we may assume that
\begin{equation}\label{descent-2}
\bigcap_{n\in\N_0} \alpha^n(W^s_a)=\{e\}\quad\mbox{and}\quad
\lim_{n\to\infty} \alpha^n(x)=e\;\;  \mbox{for all $\,x\in W^s_a$.}
\vspace{-1mm}
\end{equation}
\end{numba}
\begin{numba}\label{c-stable}
By Lemma~\ref{better-2-3},
after shrinking~$R$ if necessary, we may assume that
$\alpha|_{W^c(t)}\colon W^c(t)\to W^c(t)$ is an analytic
diffeomorphism for each $t\in\;]0,R]$.
\end{numba}
\begin{numba}\label{goodd-4}
By Lemma~\ref{better-4},
after shrinking~$R$ if necessary, we may assume that,
for each $t\in\;]0,R]$, for each $x\in W^u_b(t)$ there exists an $\alpha$-regressive
trajectory $(x_{-n})_{n\in\N_0}$ in~$W^u_b(t)$ such that $x_0=x$ and
\[
\lim_{n\to\infty}x_{-n}=e.
\]
In particular, $W^u_b(t)\sub \alpha(W^u_b(t))$ for all $t\in\;]0,R]$. 
\end{numba}
\begin{numba}\label{goodd-x}
By Lemma~\ref{better-3},
there exists an open neighbourhood $P$ of~$p$ in $W^u_b$
with $\alpha(P)\sub W^u_b$
such that, for each $x\in W^u_b\setminus P$, there exists $n\in\N_0$
such that $\alpha^n(x)\not\in P$.
After shrinking~$P$,
we may assume that $P=W^u_b(r)$ for some $r\in\;]0,R]$.
\end{numba}
The next lemma will be applied later to $A:=W^c(t)$,
$B:=W^u_b(t)$ and $C:=\alpha(B)$.
\begin{la}\label{cancel-A}
Let $G$ be a group, $B\sub C\sub G$ be  subgroups and $A\sub G$
be a subset
such that $AB$ and $AC$ are subgroups of~$G$ and $C\cap AB=B$.
Then
\[
[AC:AB]=[C:B].
\]
\end{la}
\begin{proof}
The group~$C$ acts on $X:=AC/AB$ on the left via
$c'.acAB:=c'acAB$ for $c,c'\in C$, $a\in A$. To see that the action
is transitive, let $a\in A$ and $c\in C$. Since $AC$ is a group,
we have $(ac)^{-1}=a'c'$ for certain $a'\in A$ and $c'\in C$,
entailing that $ac=(c')^{-1}(a')^{-1}$ and thus $acAB=(c')^{-1}AB=(c')^{-1}.AB$.
The stabilizer of the point $AB\in X$ is $C\cap AB=B$.
Now the Orbit Formula
shows that the map
\[
C/B\to X=AC/AB,\quad cB \mto cAB
\]
is a well-defined bijection. The assertion follows.
\end{proof}
Using notation as before (notably $\phi$ as in (\ref{goodphi})),
we have:
\begin{thm}\label{main}
Let $\alpha$ be an analytic endomorphism of a Lie group~$G$
over a totally disconnected local field~$\K$.
Then the scale $s(\alpha)$ divides the scale $s(L(\alpha))$ of
the associated Lie algebra endomorphism~$L(\alpha)$.
The following conditions are equivalent:
\begin{itemize}
\item[\rm(a)]
$s_G(\alpha)=s_{L(G)}(L(\alpha))$;
\item[\rm(b)]
There is $t_0 \in \;]0,r]$ such that the compact open subgroups
$B_t^\phi\cong (B_t,*)$ of~$G$ are tidy for~$\alpha$,
for all $t\in \;]0,t_0]$;
\item[\rm(c)]
$G$ has small tidy subgroups for $\alpha$;
\item[\rm(d)]
The contraction group $\conp(\alpha)$ is closed.
\end{itemize}
\end{thm}
\begin{proof}
The implication (b)$\impl$(c) holds as the compact open
subgroups $B^\phi_t$ for $t\in\;]0,t_0]$ form a basis
of identity neighbourhoods in~$G$.
The implication (c)$\impl$(d)
is a general fact, see Lemma~\ref{tidyclo}.

(a)$\aeq$(b):
We claim that
there exists $t_0 \in \;]0,r]$
such that
the compact open subgroups $B^\phi_t$
have displacement index
\begin{equation}\label{index-ball}
[\alpha(B^\phi_t)\colon \alpha(B^\phi_t)\cap B^\phi_t]=s(L(\alpha))\quad\mbox{for all $\,t\in\;]0,t_0]$.}
\end{equation}
If this is true, then the equivalence of~(a) and~(b) is clear.
If~$\alpha$ is an automorphism, then the claim holds by Lemma~\ref{displ-balls}.
For~$\alpha$ an endomorphism, the argument is more involved.
We first note that the product map
\[
\pi\colon W^s_a\times W^c\times W^u_b\to W^s_aW^cW^u_b=B^\phi_R
\]
is an analytic diffeomorphism as so is~$\psi$ (from~(\ref{psithe})).
Let $t\in\;]0,r]$.
Since
\[
\alpha(W^s_a(t)W^c(t))\sub W^s_a(t)W^c(t),
\]
we have $\alpha^n(W^s_a(t)W^c(t))\sub W^s_a(t)W^c(t)\sub B^\phi_t$
for all $n\in\N_0$ and thus
\begin{equation}\label{grob-minus}
W^s_a(t)W^c(t)\sub (B^\phi_t)_-.
\end{equation}
Since $\alpha(W^c(t))=W^c(t)$ and each $x\in W^u_b(t)$
has an $\alpha$-regressive trajectory within $W^u_b(t)$ (see \ref{goodd-4}),
we have
\begin{equation}\label{before-modlar}
W^c(t)W^u_b(t)\sub (B^\phi_t)_+.
\end{equation}
Thus $B^\phi_t=W^s_a(t)W^c(t)W^u_b(t)\sub (B^\phi_t)_-(B^\phi_t)_+\sub B^\phi_t$,
whence $B^\phi_t=(B^\phi_t)_-(B^\phi_t)_+$ and so
$B^\phi_t=(B^\phi_t)^{-1}=(B^\phi_t)_+(B^\phi_t)_-$
is tidy above for~$\alpha$.\\[2.3mm]
Since $B^\phi_t=W^s_a(t)W^c(t)W^u_b(t)$ is a group and $(B^\phi_t)_+$
a subgroup, (\ref{before-modlar}) implies that
\[
(B^\phi_t)_+= J_t W^c(t)W^u_b(t)
\]
with $J_t:=(B^\phi_t)_+\cap W^s_a(t)$. Since $\pi$ is a bijection, $\alpha(J_t)\sub\alpha(W^s_a(t))\sub
W^s_a(t)$ and $\alpha(W^u_b(t))\sub W^u_b$ (see \ref{goodd-x})), the inclusion
\[
J_t W^c(t)W^u_b(t)=B^\phi_t\sub \alpha((B^\phi_t)_+)=\alpha(J_t)W^c(t)\alpha(W^u_b(t))
\]
entails that $J_t\sub \alpha(J_t)$, whence
\[
J_t\sub\bigcap_{n\in\N_0}\alpha^n(J_t)\sub\bigcap_{n\in\N_0}\alpha^n(W^s_a)=\{e\},
\]
using~(\ref{descent-2}).
Thus $J_t=\{e\}$ and hence
\begin{equation}\label{quart}
W^c(t)W^u_b(t)=(B^\phi_t)_+,
\end{equation}
which is a subgroup. Also
\[
\alpha((B^\phi_t)_+)=W^c(t)\alpha(W^u_b(t))
\]
is a subgroup, and
\[
\alpha(W^u_b(t))\cap W^c(t)W^u_b(t)=W^u_b(t)
\]
since~$\pi$ is a bijection.
Hence, by \cite[Lemma~5]{END} and Lemma~\ref{cancel-A},
\begin{eqnarray}
[\alpha(B^\phi_t): \alpha(B^\phi_t)\cap B^\phi_t] &=&
[\alpha((B^\phi_t)_+):(B^\phi_t)_+]\notag \\
&=&[W^c(t)\alpha(W^u_b(t)): W^c(t)W^u_b(t)]\notag \\
&=&
[\alpha(W^u_b(t)):W^u_b(t)].\label{semi-half}
\end{eqnarray}
Applying now Lemma~\ref{displ-balls}
to $\alpha|_{W^u_b(r)}\colon W^u_b(r)\to W^u_b$
instead of $\alpha\colon G_0\to G$
and $\phi_u$ instead of~$\phi$,
we see that there is $t_0\in\;]0,r]$ such that
\begin{equation}\label{semi-full}
[\alpha(W^u_b(t)):W^u_b(t)]=s(L(\alpha)|_{\cg_{>b}})=s(L(\alpha)|_{\cg_{>1}})
=s(L(\alpha))
\end{equation}
for all $t\in\;]0,t_0]$,
using Theorem~\ref{thm-lincase}\,(d) for the penultimate equality.
Combining (\ref{semi-half}) and (\ref{semi-full}), we get~(\ref{index-ball}).

(d) $\Rightarrow$ (b):
Recall that $B^\phi_t$ is tidy above for all $t\in\;]0,r]$;
from (\ref{grob-minus}) and (\ref{before-modlar}), we deduce that
\begin{equation}\label{uno}
(B^\phi_t)_+\cap (B^\phi_t)_-\supseteq W^c(t).
\end{equation}
By (\ref{descent-2}) and (\ref{grob-minus}), we have
\begin{equation}\label{due}
W^s_a(t)\sub \; \conp(\alpha)\cap (B^\phi_t)_-.
\end{equation}
Since $B^\phi_t=W^s_a(t)W^c(t)W^u_b(t)$ and $(B^\phi_t)_-$
is a subgroup of $B^\phi_t$ which contains $W^s_a(t)W^c(t)$,
we have
\[
(B^\phi_t)_-=W^c(t)W^u_b(t)I_t
\]
with $I_t:=(B^\phi_t)_-\cap W^u_b(t)$.
Then $I_t=\{e\}$ as the existence of an element $x\in I_t\setminus\{e\}$
gives rise to a contradiction as follows:
Since $I_t\sub (B^\phi_t)_-$, we must have $\alpha^n(x)\in B^\phi_t$
for all $n\in \N_0$.
However, by~\ref{goodd-x},
there exists $n\in\N$ such that $\alpha^n(x)\in W^u_b\setminus W^u_b(r)$
and thus $\alpha^n(x)\not\in B^\phi_t$, which is absurd.
Hence
\begin{equation}\label{tres}
(B^\phi_t)_- = W^s_a(t)W^c(t),
\end{equation}
and thus $(B^\phi_t)_-=\;(\conp(\alpha)\cap (B^\phi_t)_-)((B^\phi_t)_+\cap (B^\phi_t)_-)$.
Using Lemma~\ref{tidyclo}, we deduce that $B^\phi_t$ is tidy for~$\alpha$,
for all $t\in\;]0,r]$.

Finally, as $B^\phi_r$ is tidy above for~$\alpha$,
we deduce from Lemma~\ref{divisi} and~(\ref{index-ball})
that $s(\alpha)$ divides $[\alpha(B^\phi_r):\alpha(B^\phi_r)\cap B^\phi_r]=s(L(\alpha))$.
\end{proof}
\begin{rem}\label{the-balls-tdy}
The proof of the implication (d)$\impl(b)$ showed that
the compact open subgroups~$B^\phi_t$ are tidy for~$\alpha$,
for all $t\in \;]0, r]$. This information enables
a better understanding of $W^s_a(t)$ and $W^c(t)$.
By (\ref{quart}) and (\ref{tres}),
\[
W^c(t)=W^s_a(t)W^c(t)\cap W^c(t)W^u_b(t)=(B^\phi_t)_-\cap (B^\phi_t)_+
\]
is a compact \emph{subgroup} of~$B^\phi_t$, for each $t\in\;]0,r]$.
If we can show that
\begin{equation}\label{also-Ws}
\conp(\alpha)\cap B^\phi_t=\; \conp(\alpha)\cap (B^\phi_t)_-=W^s_a(t),
\end{equation}
then $W^s_a(t)$ is a compact \emph{subgroup} of~$G$, for all $t\in\;]0,r]$.
Now, the first equality in (\ref{also-Ws}) holds by \cite[Proposition~11\,(b)]{END}.
As for the second equality, the inclusion ``$\supseteq$''
holds by~(\ref{due}). Since $\conp\!(\alpha)\cap (B^\phi_t)_-$ is a subgroup
of $(B^\phi_t)_-=W^s_a(t)W^c(t)$ which contains~$W^s_a(t)$,
it is of the form
\[
\conp(\alpha)\cap (B^\phi_t)_-=W^s_a(t)K_t
\]
with $K_t:=\; \conp(\alpha)\cap (B^\phi_t)_-\cap W^c(t)$.
But $K_t=\{e\}$ since $\conp(\alpha)\cap W^c(t)=\{e\}$;
to see the latter, let $e\not=x\in W^c(t)$. There is $\tau\in\;]0,t]$
such that $x\not\in W^c(\tau)$.
Since $\alpha|_{W^c(t)}\colon W^c(t)\to W^c(t)$
is a bijection and $\alpha(W^c(\tau))=W^c(\tau)$ (see~\ref{c-stable}),
we have $\alpha^n(x)\in W^c(t)\setminus W^c(\tau)$
for all $n\in\N_0$ and thus $\alpha^n(x)\not\in W^c(\tau)$,
showing that $\alpha^n(x)$ does not converge to~$e$ in $W^c(t)$ as $n\to\infty$
(and hence neither in~$G$).
\end{rem}
\subsection*{{\normalsize Foliations of the `big cell'}}
\begin{thm}\label{bigcell}
Let $\alpha$ be an analytic endomorphism of a Lie group~$G$
over a totally disconnected local field~$\K$. If $\conp(\alpha)$
is closed in~$G$, then also $\conm(\alpha)$ is closed
and the following holds:
\begin{itemize}
\item[\rm(a)]
$\conp(\alpha)$, $\lev(\alpha)$, and $\conm(\alpha)$
are Lie subgroups of~$G$ with Lie algebras
$\conp(L(\alpha))$, $\lev(L(\alpha))$, and $\conm(L(\alpha))$,
respectively;
\item[\rm(b)]
$\Omega:=\;\conp(\alpha)\lev(\alpha)\conm(\alpha)$
is an $\alpha$-invariant open identity neighbourhood in~$G$.
The product map
\[
\pi\colon \conp(\alpha)\times \lev(\alpha)\times \conm(\alpha)
\to \; \conp(\alpha)\lev(\alpha)\conm(\alpha),\;\;(x,y,z)\mto xyz
\]
is an analytic diffeomorphism.
\item[\rm(c)]
$\alpha|_{\lev(\alpha)}$ and $\alpha|_{\conm(\alpha)}$
are analytic automorphisms.
\end{itemize}
\end{thm}
\begin{proof}
\emph{Openness of~$\Omega$:}
Note first that $\Omega:=\;\conp(\alpha)\lev(\alpha)\conm(\alpha)$
contains the open identity neighbourhood $B^\phi_r=W^s_a(r)W^c(r)W^u_b(r)$
encountered in the proof of Theorem~\ref{main},
since $W^s_a(r)\sub \;\conp(\alpha)$ by~(\ref{descent-2}), $W^u_b(r)\sub \;\conm(\alpha)$
by~\ref{goodd-4}
and $W^c(r)\sub \; \lev(\alpha)$ since $W^c(r)$ is compact and $\alpha$-stable.
Since $\lev(\alpha)$ normalizes $\conp(\alpha)$ and $\conm(\alpha)$,
both $P_\alpha:=\conp(\alpha)\lev(\alpha)$ and $P_\alpha^-:=\lev(\alpha)\conm(\alpha)$
are subgroups of~$G$.
For each $g\in P_\alpha$,
the left translation $\lambda_g\colon G\to G$, $x\mto gx$
is a homeomorphism which takes the identity-neighbourhood~$\Omega$ onto the $g$-neighbourhood
\[
g\Omega=gP_\alpha\conm(\alpha)=P_\alpha\conm(\alpha)=\Omega.
\]
If $h\in\Omega$, then $h=gk$ with $g\in P_\alpha$ and $k\in \;\conm(\alpha)$.
Now the right translation $\rho_k\colon G\to G$, $x\mto xk$ is a homeomorphism
which takes the $g$-neighbourhood~$\Omega$ onto the neighbourhood
\[
\Omega k = P_\alpha \conm(\alpha) k=\; P_\alpha\conm(\alpha) =\Omega
\]
of $x=gk$. Hence~$\Omega$ is a neighbourhood of each $x\in\Omega$
and thus~$\Omega$ is open.\\[2.3mm]
\emph{Lie subgroups.}
The open subset $B^\phi_r=W^s_a(r)\times W^c(r)\times W^u_b(r)$
of~$G$ has $W^s_a(r)$ as a submanifold.
Since $\conp(\alpha)\cap B^\phi_r=W^s_a(r)$ (see (\ref{also-Ws})) is a submanifold,
we deduce that $\conp(\alpha)$ is a Lie subgroup of~$G$
(cf.\ Lemma~\ref{spot-liesub})
which has $W^s_a(r)$ as an open submanifold;
thus
\[
L(\conp(\alpha))=T_e(W^s_a(r))=\cg_{<a}=\; \conp(L(\alpha)).
\]
Recall from Remark~\ref{the-balls-tdy} that~$B^\phi_r$ is tidy for~$\alpha$.
Using the proof of \cite[Proposition~19]{END} for the first equality, we have
\[
\lev(\alpha)\cap B^\phi_r=(B^\phi_r)_+\cap (B^\phi_r)_-=W^c(r),
\]
which is a submanifold of~$B^\phi_r$. Hence $\lev(\alpha)$ is a Lie subgroup of~$G$
which has $W^c(r)$ as an open submanifold, and thus
\[
L(\lev(\alpha))=\cg_1=\lev(L(\alpha)).
\]
Next, recall from \cite[Proposition~11\,(a)]{END} that
\[
\conm(\alpha)\cap B^\phi_r\sub (B^\phi_r)_+=W^c(r)W^u_b(r).
\]
Since $W^u_b(r)\sub \; \conm(\alpha)$, this entails that
\[
\conm(\alpha)\cap B^\phi_r = L_r W^u_b(r)
\]
with $N_r:=\; \conm(\alpha)\cap W^c(r)$.
Let $\bik(\alpha)$ be the bounded iterated kernel of~$\alpha$
and $\nub(\alpha)$ be the nub subgroup (see \cite{END} and \cite{BGT}).
Then $\bik(\alpha)\sub\nub(\alpha)$ and
since $\alpha$ has small tidy subgroups by Theorem~\ref{main},
we have $\nub(\alpha)=\{e\}$ (see \cite{END}).
If $x\in N_r$, then there exists
an $\alpha$-regressive trajectory $(x_{-n})_{n\in\N_0}$ such that $x=x_0$ and
\[
\lim_{n\to\infty} x_{-n}=e.
\]
On the other hand, since $x\in W^c(r)$ which is $\alpha$-stable,
there exists an $\alpha$-regressive trajectory $(y_{-n})_{n\in\N_0}$
in $W^c(r)$ with $y_0=x$.
Then $(x_{-n}y_{-n}^{-1})_{n\in\N_0}$ is an $\alpha$-regressive trajectory
such that $\{x_{-n}y_{-n}^{-1}\colon n\in\N_0\}$ is relatively compact.
Thus $x_{-n}y_{-n}^{-1}\in\; \parm(\alpha)$ for each $n\in\N_0$
and $\alpha^n(x_{-n}y_{-n}^{-1})=e$, whence $x_{-n}y_{-n}^{-1}
\in\bik(\alpha)\sub\nub(\alpha)=\{e\}$. Hence $(x_{-n})_{n\in\N_0}=(y_{-n})_{n\in\N_0}$
is an $\alpha|_{W^c(r)}$-regressive trajectory which tends to~$e$ as $n\to\infty$.
As $\alpha|_{W^c(r)}$ is a distal automorphism of~$W^x(r)$
(cf.\ Proposition~\ref{charviaeigen}), the latter is only possible if $x=e$.
Thus $N_r=\{e\}$ and hence
\[
\conm(\alpha)\cap B^\phi_r =W^u_b(r),
\]
which is a submanfold of~$\Omega$. Hence $\conm(\alpha)$
is a Lie subgroup with Lie algebra $\cg_{>b}=\; \conm(L(\alpha))$.\\[2.3mm]
\emph{$\pi$ is injective.}
Let $a,a'\in\; \conp(\alpha)$, $b,b'\in \lev(\alpha)$, and $c,c'\in\;\conm(\alpha)$
such that $abc=a'b'c'$.
Then
\[
x:=(a')^{-1}a=b'c'c^{-1}b^{-1}\in\;\conp(\alpha)\, \cap\parm(\alpha)\sub\lev(\alpha).
\]
There exists $n\in\N_0$ such that $\alpha^n(x)\in B^\phi_r$.
Since $x\in\lev(\alpha)$, also $\alpha^n(x)\in\lev(\alpha)$ and thus
\[
\alpha^n(x)\in (B^\phi_r)_+\cap B^\phi_r)_-=W^c(r).
\]
Let $y\in W^c(r)$ such that $\alpha^n(y)=\alpha^n(x)$.
Since $W^c(r)\sub\lev(\alpha)$,
we then have $(a')^{-1}ay^{-1}\in\bik(\alpha)\sub\nub(\alpha)=\{e\}$,
whence $(a')^{-1}a=y\in\lev(\alpha)\,\cap \conm(\alpha)$.
Let $(z_{-n})_{n\in\N_0}$ be an $\alpha$-regressive trajectory
with $z_0=(a')^{-1}a$, such that $z_{-n}\to e$ as $n\to\infty$.
For each $t\in\;]0,r]$,
we have $z_{-n}\in B^\phi_t$ for some $n\in\N_0$,
entailing that
\[
z_{-n}\in B^\phi_t\cap\lev(\alpha)=(B^\phi_t)_+\cap (B^\phi_t)_-=W^c(t)
\]
and thus $(a')^{-1}a=\alpha^n(z_{-n})\in W^c(t)$.
Therefore
\[
(a')^{-1}ay^{-1}\in\bigcap_{t\in]0,r]}B^\phi_t=\{e\},
\]
whence $(a')^{-1}a=e$ and hence $a=a'$.
Now $bc=b'c'$ entails that
\[
(b')^{-1}b=c'c^{-1}\in\lev(\alpha)\, \cap\conm(\alpha).
\]
Also this group element has an $\alpha$-regressive trajectory which
converges to~$e$ and hence enters $B^\phi_t$ for each $t\in \;]0,r]$,
entailing that $(b')^{-1}b\in (B^\phi_t)_+\cap (B^\phi_t)_-=W^c(t)$
for all $t\in\;]0,r]$ and thus $b=b'$. Hence also $c=c'$.\\[2.3mm]
\emph{$\pi$ is a diffeomorphism.}
Since $\psi$ is an analytic diffeomorphsm,
also
\[
\mu:=\pi|_{W^s_a\times W^c\times W^u_b}\colon W^s_a\times W^c\times W^u_b\to B^\phi_R
\]
is an analytic diffeomorphism.
For $(a,b,c)\in \;\conp(\alpha)\times\lev(\alpha)\times\conm(\alpha)=:Y$,
let us show that~$\pi$ is a local diffeomorphism at~$(a,b,c)$.
It suffices to prove that the map
\[
h\colon Y\to \Omega,\quad (a',b',c')\mto\pi(aa',bb',c'c)
\]
is a local diffeomorphism at $(e,e,e)$.
Since $\lev(\alpha)$ normalizes $\conp(\alpha)$
and $\conp(\alpha)$ is a submanifold of~$G$, the map
\[
\beta\colon \conp(\alpha)\to\; \conp(\alpha),\quad a'\mto b^{-1}a'b
\]
is an analytic diffeomorphism. Let $Q\sub \;\conp(\alpha)$ be an open identity neighbourhood
such that $\beta(Q)\sub W^s_a$.
Then the formula
\[
h(a',b',c')=ab \,\mu(\beta(a'),b',c')c
\]
shows that $h$ is a local diffeomorphism at $(e,e,e)$.

To prove~(c), note that $\gamma:=\alpha|_{\lev(\alpha)}$ and $\delta:=\alpha|_{\conm(\alpha)}$
are local diffeomorphsms
at~$e$ (by the Inverse Function Theorem), since
$L(\gamma)=L(\alpha)|_{\cg_1}$ and $L(\delta)=L(\alpha)|_{\cg_{>1}}$
are automorphisms of the tangent spaces $L(\lev(\alpha))=\cg_1$
and $L(\conm(\alpha))=\cg_{>1}$, respectively, at~$e$.
Since~$\gamma$ and~$\delta$ are, moreover,
bijective analytic endomorphisms,
they are analytic automorphisms.
\end{proof}
\begin{rem}
(a) Note that also the groups $P_\alpha$ and $P_\alpha^-$
encountered in the preceding proof are Lie subgroups
since~$\pi$ is an analytic diffeomorphism.\\[2mm]
(b) We mention that $P_\alpha=\;\parp(\alpha)$ and $P_\alpha^-=\; \parm(\alpha)$
(see \cite[Lemma~12.1 (d) and (f)]{BGT}).\\[2mm]
(c) Since $\pi$ is an analytic diffeomorphism, we see that the ``big cell'' $\Omega$ can be
foliated into right translates of $\conp(\alpha)$ parametrized by $\parm(\alpha)$, or alternatively into right translates of~$\parp(\alpha)$, parametrized by $\conm(\alpha)$.
Likewise, we can foliate~$\Omega$ into left translates of $\conm(\alpha)$ parametrized by $\parp(\alpha)$,
or into left translates of $\parm(\alpha)$ parametrized by $\conp(\alpha)$.
\end{rem}
\begin{numba}\label{def-imme}
Consider an analytic map $f\colon M\to N$
between analytic manifolds over a totally disconnected local field~$\K$.
Recall from~\cite[Part I, Chapter~III]{Ser}
that~$f$ is called an \emph{immersion}
if $f$ locally looks like a linear injection
around each point, in suitable charts
(or equivalently, if $T_p(f)$ is injective for all $p\in M$).
If~$G$ is a Lie group over a totally disconnected local field
and $H$ a subgroup of~$G$, endowed with an analytic manifold structure
turning it into a Lie group and making the inclusion map $H\to G$ an immerson,
then~$H$ is called an \emph{immersed Lie subgroup}
of~$G$.
\end{numba}
\begin{rem}\label{rem-nonclo}
(a) If $\alpha$ is an analytic \emph{automorphism}
of a Lie group~$G$ over a totally disconnected local field~$\K$
and $\conp(\alpha)$ is not closed,
then it is still possble to turn $\conp(\alpha)$
and $\conm(\alpha)$ into \emph{immersed}
Lie subgroups of~$G$ modelled on $\conp(L(\alpha))$
and $\conm(L(\alpha))$, respectively,
such that $\alpha|_{\conp(\alpha)}$
and $\alpha^{-1}|_{\conm(\alpha)}$
are contractive analytic automorphisms
of these Lie groups (see \cite[Proposition~6.3\,(b)]{Fin}).\\[2.3mm]
(b) After this research was completed, it was shown in \cite{CLO}
that the `big cell'
\[
\Omega:=\,\conp(\alpha)\lev(\alpha)\conm(\alpha)
\]
is
open in~$G$ for each endomorphism~$\alpha$ of a totally disconnected locally compact group~$G$.
If~$G$ is a Lie group over a totally disconnected local field
and $\alpha\colon G\to G$ an analytic endomorphism,
then $\conp(\alpha)$, $\lev(\alpha)$ and $\conm(\alpha)$
can be turned into immersed Lie subgroups $\conp\hspace*{-.2mm}{\!}^*(\alpha)$,
$\lev{\!}^*(\alpha)$ and $\conm\hspace*{-.2mm}{\!}^*(\alpha)$ of $G$
modelled on $\conp(L(\alpha))$, $\lev(L(\alpha))$ and $\conm(\alpha)$,
respectively, such that~$\alpha$ induces analytic endomorphisms of the
immersed Lie subgroups and the product map
\[
\conp^{\!*}(\alpha)\times \lev^{\!*}(\alpha)\, \times \conm^{\!*}(\alpha)\to \Omega,\;\; (a,b,c)\mto abc
\]
is surjective and \'{e}tale (i.e., a local diffeomorphism at each point), see~\cite{CLO}.
\end{rem}
\subsection*{{\normalsize Closedness of contraction groups}}
We now mention a characterization and describe a criterion
for closedness of contraction groups of endomorphisms.
The next lemma is covered by \cite[Theorem D and F]{BGT};
for the case of automorphisms, see already \cite[Theorem 3.32]{BaW}
(if $G$ is metrizable).
\begin{la}\label{tidy-via-levi}
Let $\alpha$ be an endomorphism of a totally disconnected
locally compact group~$G$. Then the contraction group
$\conp(\alpha)$ is closed in~$G$ if and only if $\conp(\alpha)\cap \lev(\alpha)=\{e\}$.\,\Punkt
\end{la}
\begin{la}\label{crit-clo}
Let $\phi\colon G\to H$ be an injective,
continuous homomorphism between totally disconnected,
locally compact groups
and $\alpha\colon G\to G$ as well as
$\beta\colon H\to H$ be endomorphisms
such that $\beta\circ \phi=\phi\circ \alpha$.
If $\conp(\beta)$ is closed, then also $\conp(\alpha)$ is closed.
\end{la}
\begin{proof}
Using Lemma~\ref{tidy-via-levi},
we get
\[
\phi(\conp(\alpha)\cap \lev(\alpha))\sub\;  \conp(\beta)\cap\lev(\beta)=\{e\}.
\]
Thus $\conp(\alpha)\cap\lev(\alpha)=\{e\}$, whence $\conp(\alpha)$
is closed (by Lemma~\ref{tidy-via-levi}).
\end{proof}
\begin{prop}\label{clolina}
For every totally disconnected local field~$\K$,
every inner automorphism of a closed subgroup $G\sub\GL_n(\K)$
has a closed contraction group.
\end{prop}
\begin{proof}
It suffices to show that each inner automorphism of $\GL_n(\K)$
has a closed contraction group. Let $e\in\GL_n(\K)$ be the identity matrix.
Given $g\in\GL_n(\K)$,
consider the inner automorphism
\[
I_g \colon \GL_n(\K)\to\GL_n(\K),\quad h\mto ghg^{-1}
\]
and the linear endomorphism
\[
\alpha\colon M_n(\K)\to M_n(\K),\quad A\mto gAg^{-1}.
\]
We know from Section~\ref{sec-vector} that $\conp(\alpha)$ is closed.
Then $V:=\GL_n(\K)-e$ is an $\alpha$-stable
$0$-neighbourhood in~$M_n(\K)$
and
\[
\phi\colon V \to \GL_n(\K),\quad A\mto A+e
\]
is a homeomorphism. Now $I_g \circ \phi=\phi\circ\alpha|_V$
as $I_g(\phi(A))=g(A+e)g^{-1}=gAg^{-1}+e=\phi(\alpha(A))$, i.e.,
$\phi$ is a topological conjugacy between the dynamical systems $(V,\alpha|_V)$
and $(\GL_n(\K),I_g)$.
Hence
\[
\conp(I_g)=\phi(\conp(\alpha)\cap V),
\]
which is closed in $\GL_n(\K)$.]
\end{proof}
Combining Lemma~\ref{crit-clo} and Proposition~\ref{clolina},
we get:
\begin{cor}\label{lin-then-clo}
If a totally disconnected, locally compact group~$G$
admits a faithful continuous representation
$\pi\colon G\to\GL_n(\K)$
over some totally disconnected local field~$\K$, then every inner automorphism
of~$G$ has a closed contraction group.\,\Punkt
\end{cor}
\begin{rem}
In particular, every group~$G$ of $\K$-rational points
of a linear algebraic group over a totally disconnected local field~$\K$
is a closed subgroup of some $\GL_n(\K)$,
whence $\conp(\alpha)$ is closed in~$G$ by Example~\ref{clolina}
and so
\begin{equation}\label{seepadic}
s_G(g)=s_{L(G)}(\Ad_g)=\prod_{|\lambda_j|_\K\geq 1}|\lambda_j|_\K
\end{equation}
in terms of the eigenvalues $\lambda_1,\ldots,\lambda_m$
of $\Ad_g$ in an algebraic closure~$\wb{\K}$,
repeated according to their algebraic multiplicities
(by Theorem~\ref{main} and Theorem~\ref{thm-lincase}).
For Zariski-connected reductive $\K$-groups,
this was already shown in \cite[Proposition~3.23]{BaW}.
See also \cite[Remark~9.7]{Lec}.
\end{rem}
The following result was announced in~\cite{DIR}
(for automorphisms).
\begin{prop}
Let $\alpha$ be an analytic endomorphism of a $1$-dimensional Lie group~$G$
over a local field~$\K$, with Lie algebra $\cg:=L(G)$.
If
\[\conm(L(\alpha))=\cg,
\]
assume that $\bigcup_{n\in\N_0}\ker(\alpha^n)$
is discrete;\footnote{Which is, of course, automatic if~$\alpha$ is an automorphism.}
if $\conm(L(\alpha))\not=\cg$,
we do not impose further hypotheses.
Then $\conp\!(\alpha)$ is closed
in~$G$
and thus $s(\alpha)=s(L(\alpha))$.
\end{prop}
\begin{proof}
Since $\cg$ is a $1$-dimensional
$\K$-vector space and
\[
\cg=\; \conp(L(\alpha))\oplus \lev(L(\alpha))\oplus \conm(L(\alpha)),
\]
we see that~$\cg$ coincides with one of the three summands.
Let $a$, $b$, $R$, $W^s_a$, $W^c$, and $W^u_b$ be as in~\ref{phimain}.

If $\conp(L(\alpha))=\cg$, then $W^s_a$ is a submanifold of~$G$
of full dimension~$1$ and hence open in~$G$.
By~\ref{has-descen}, we may assume that $W^s_a\sub \; \conp(\alpha)$,
after shrinking~$R$ if necessary, whence
the subgroup $\conp(\alpha)$ is open and hence closed in~$G$.

If $\lev(L(\alpha))=\cg$, then $W^c$ is open in~$G$
and we may assume that $W^c(t)$ is a compact open subgroup of~$G$
for all $t\in\;]0,R]$, after shrinking~$R$ if necessary.
The bijective analytic endomorphism
$\alpha|_{W^c}\colon W^c\to W^c$ is a local analytic diffeomorphism at~$e$
(by the Inverse Function Theorem) and hence an analytic automorphism of
the Lie group~$W^c$. By Proposition~\ref{charviaeigen}\,(c),
the automorphism $\alpha|_{W^c}$ is distal and hence
\[
\conp(\alpha)\cap W^c=\; \conp(\alpha|_{W^c})=\{e\}.
\]
Thus $\conp(\alpha)$ is discrete.

If $\conm(L(\alpha))=\cg$, then $W^u_b$ open
and we choose an open neighbourhood~$P$ of~$e$ in $W^u_b$
as in Lemma~\ref{better-3} (with $M:=G$, $f:=\alpha$, and $p:=e$).
Then $\conp(\alpha)=\bigcup_{n\in\N_0}\ker(\alpha^n)$, which is discrete
(and thus closed) by hypothesis.
In fact, if $x\in\; \conp(\alpha)$, then there exists $n_0\in\N_0$ such that
$\alpha^n(x)\in P$ for all $n\geq n_0$. Then $\alpha^{n_0}(x)=e$,
as we chose~$P$ in such a way that the $\alpha$-orbit of each $y\in P\setminus\{e\}$
leaves~$P$.

In each case, the final assertion follows from Theorem~\ref{main}.
\end{proof}
\begin{rem}
If $\F$ is a field of prime order, then
$G:=(\F[\hspace*{-.2mm}[X]\hspace*{-.2mm}],+)$ is a $1$-dimensional
Lie group over $\K=\F(\hspace*{-.3mm}(X)\hspace*{-.3mm})$ and
the left shift
\[
\sum_{n=0}^\infty a_nX^n\mto \sum_{n=0}^\infty a_{n+1} X^n
\]
is an analytic endomorphism of~$G$, as it coincides with the linear
(and hence analytic) map
\[
\beta\colon \K\to\K,\quad z\mto X^{-1}z
\]
on the open subgroup $X\F[\hspace*{-.2mm}[X]\hspace*{-.2mm}]$ of~$G$.
It is easy to see that $\conm(\alpha)=G$ and
\[
\conp(\alpha)=\bigcup_{n\in\N_0}\ker(\alpha^n)\\[-1mm]
\]
is the proper dense
subgroup of all finitely supported sequences (see \cite[Remark~9.5]{BGT}).
Since~$G$ is compact, $s(\alpha)=1$. As $L(\alpha)=\beta$ with scale~$p$ (by Theorem~\ref{thm-lincase}\,(c)), we have $s(L(\alpha))=p\not=s(\alpha)$.
\end{rem}
\subsection*{{\normalsize A non-closed contraction group}}
We now describe an analytic automorphism~$\alpha$
of a Lie group over a local field
of positive characteristic
such that $\conp(\alpha)$ is not closed.
The example is taken from~\cite{Lec}.
\begin{example}\label{non-lin1}
Let $\F$ be a finite field, with $p$ elements. Consider the set
$G:=\F^\Z$ of all functions $f\colon\Z\to \F$.
Then~$G$ is a compact topological group under addition, with the product topology.
The right shift
\[
\alpha\colon G\to G\,, \quad \alpha(f)(n):=f(n-1)
\]
is an automorphism of~$G$.
It is easy to check that $\conp(\alpha)$
is the set of all functions $f\in \F^\Z$
with support bounded below (i.e., there exists $n_0\in\Z$ such that
$f(n)=0$ for all $n<n_0$).
Thus $\conp(\alpha)$
is a dense, proper subgroup
of~$G$.\\[2.3mm]
Now $G$ can be considered as a $2$-dimensional Lie group
over $\K:=\F(\!(X)\!)$, using the bijection
$G\to \F[\hspace*{-.17mm}[X]\hspace*{-0.16mm}]\times
\F[\hspace*{-.17mm}[X]\hspace*{-.16mm}]$,
\[
f\mto \left(\sum_{n=1}^\infty f(-n)X^{n-1},\,
\sum_{n=0}^\infty f(n)X^n\right)
\]
as a global chart.
The automorphism of $\F[\hspace*{-.17mm}[X]\hspace*{-.16mm}]^2$
corresponding to~$\alpha$
coincides on the open $0$-neighbourhood
$X \F[\hspace*{-.17mm}[X]\hspace*{-.16mm}]\times
\F[\hspace*{-.17mm}[X]\hspace*{-.16mm}]$
with the linear map
\[
\beta\colon \K^2\to \K^2\,,\quad
\beta(v,w)=(X^{-1}v,Xw)\,.
\]
Hence $\alpha$ is an analytic
automorphism. Since $\conp(\alpha)$ is not closed,
$G$ cannot admit a faithful continuous
representation $G\to\GL_n(\K)$ for any $n\in\N$,
see Corollary~\ref{lin-then-clo}.
\end{example}
\subsection*{{\normalsize The scale on closed subgroups and quotients}}
For $\alpha$ an automorphism of a totally disconnected locally compact group~$G$, the scale of the restriction
$\alpha|_H$ to a closed $\alpha$-stable subgroup $H\sub G$
and the scale pf the induced automorphism
on the quotient group $G/H$ (for normal~$H$) were studied in~\cite{FUR};
some generalizations for endomorphisms were obtained in~\cite{BGT}
(compare also \cite{GBV}, if~$\alpha$ has small tidy subgroups).
The following proposition generalizes a corresponding
result for inner automorphisms of $p$-adic Lie groups
established in \cite[Corollary~3.8]{FOR}.
\begin{prop}\label{subquot}
Let $G$ be a Lie group over a local field,
$\alpha$ be an analytic endomorphism of~$G$ and $H\sub G$
be an $\alpha$-invariant Lie subgroup of~$G$.
Then the following holds:
\begin{itemize}
\item[\rm(a)]
If $\conp(\alpha)$ is closed, then $s_H(\alpha|_H)$ divides $s_G(\alpha)$.
\item[\rm(b)]
If $H$ is a normal subgroup, $\conp(\alpha)$ is closed and also the induced analytic endomorphism~$\wb{\alpha}$
of $G/H$ has a closed contraction group $\conp(\wb{\alpha})$,
then $s_G(\alpha)=s_H(\alpha|_H)s_{G/H}(\wb{\alpha})$.
\end{itemize}
\end{prop}
\begin{proof}
This is immediate from Theorem~\ref{main}
and Corollary~\ref{vsubquot}.
\end{proof}
\subsection*{{\normalsize When homomorphisms are subimmersions}}
Consider an analytic mapping $f\colon M\to N$
between analytic manifolds over a totally disconnected local field~$\K$.
Recall from~\cite[Part I, Chapter~III]{Ser}
that~$f$ is called a \emph{submersion}
if~$f$ locally looks like a linear projection around each point, in suitable charts.
If $f$ locally looks like $j\circ q$ where $q$ is a submersion and $j$ an immersion
(as in \ref{def-imme}),
then $f$ is called a \emph{subimmersion}.
If $\car(\K)=0$, then an analytic map is a subimmersion
if and only if
$T_x(f)$ has constant rank for $x$ in some neighbourhood of each point $p\in M$
(see \cite[Part II, Chapter III, \S10, Theorem in 4)]{Ser}).
As a consequence, every analytic homomorphism
between Lie groups over a totally disconnected local field of characteristic~$0$ is
a subimmersion.
Analytic homomorphisms between Lie groups
over local fields of positive characteristic need not be
subimmersions, as the following example shows.
\begin{example}
Let $\F$ be a finite field with~$p$ elements
and $\K:=\F(\hspace*{-.3mm}(X)\hspace*{-.3mm})$.
Since $\car(\K)=p$,
the Frobenius homomorphism
\[
\alpha \colon \K\to\K, z\mto z^p
\]
is an injective endomorphism of the field~$\K$
and an injective endomorphism of the additive topological
group $(\K,+)$.
If $\alpha$ was a subimmersion then $\alpha$, being injective,
would be an immersion which it is not as $\alpha'(z)=0$ for all $z\in\K$.
Thus~$\alpha$ is not a subimmersion.
Note that $\conp(\alpha)$ coincides with the subgroup
$X\F[\hspace*{-.3mm}[X]\hspace*{-.3mm}]$, which is open;
hence also $\alpha|_{\conp(\alpha)}$ is not a subimmersion.
\end{example}
It is not a coincidence that the endomorphism $\alpha$ in the preceding
example is pathological also on $\conp(\alpha)$: If an endomorphism
fails to be a subimmersion, then the trouble
must be caused by its restriction
to the contraction group:
\begin{cor}
Let $\alpha$ be an analytic endomorphism
of a Lie group~$G$ over a local field~$\K$ of positive
characteristic, with closed contraction group $\conp(\alpha)$.
If $\beta:=\alpha|_{\conp(\alpha)}$
is a subimmersion, then~$\alpha$ is a subimmersion.
\end{cor}
\begin{proof}
If $\beta$ is a subimmersion,
then the restricton of $\alpha$ to the open set $\Omega$ from
Theorem~\ref{bigcell} corresponds to the self-map
\[
\pi^{-1}\circ \alpha|_\Omega\circ\pi =\beta\times \alpha|_{\lev(\alpha)\times\conm(\alpha)}
\]
of $\,\conp(\alpha)\times\lev(\alpha)\times\conm(\alpha)$
whose second factor is an analytic diffeomorphism, and thus~$\alpha$
is a subimmerson.
\end{proof}
\section{Contractive automorphisms}\label{sec-contr}
As shown in \cite{SIM},
$p$-adic Lie groups appear naturally in the classification
of the simple totally disconnected contraction groups,
and are among the building blocks for general
contraction groups. We recall some of the results
and give a new proof for the occurrence of $p$-adic Lie
groups in the classification.
\begin{defn}
An automorphism $\alpha$ of a Hausdorff topological group~$G$
is called \emph{contractive} if $\conp(\alpha)=G$.
We then call $(G,\alpha)$ a \emph{contraction group}.
If, moreover, $G$ is totally disconnected and locally compact,
we say that $(G,\alpha)$ is a \emph{totally disconnected
contraction group}.
An \emph{isomorphism} between totally disconnected contraction
groups $(G,\alpha)$ and $(H,\beta)$ is a continuous
group homomorphism $\phi\colon G\to H$ such that
$\beta\circ \phi=\phi\circ\alpha$.
A totally disconnected contraction group $(G,\alpha)$
is called \emph{simple}
if $G\not=\{e\}$ and $G$ does not have
closed $\alpha$-stable normal subgroups other than~$\{e\}$ and~$G$.
\end{defn}
\begin{rem}
We mention that
contraction groups $\conp(\alpha)$ of automorphisms
arise in many contexts:
In representation theory
in connection with the Mautner phenomenon
(see \cite[Chapter~II, Lemma~3.2]{Mar}
and (for the $p$-adic case) \cite{Wan});
in probability theory
on groups (see \cite{HaS}, \cite{Sie}, \cite{Si2}
and (for the $p$-adic case)~\cite{DaS});
and in the structure theory
of totally disconnected, locally compact groups
(see \cite{BaW}, \cite{Jaw}, and \cite{BGT}).\\[2.3mm]
If a locally compact group~$G$ admits a contractive automorphism~$\alpha$,
then there exists an $\alpha$-stable,
totally disconnected, closed normal subgroup $N\sub G$ such that
\[
G=N\times G_e
\]
internally as a topological group,
where $G_e$ is the identity component of~$G$ (see \cite{Sie}).
Thus $(G,\alpha)$ is the direct product of the totally disconnected contraction group
$(N,\alpha|_N)$ and the connected contraction group $(G_e,\alpha|_{G_e})$
(which is a simply connected, nilpotent
real Lie group, as shown by Siebert).
\end{rem}
\begin{numba}\label{weak-DP}
If $F$ is a finite group and~$X$ a set, we write $F^{(X)}$
for the group of all functions $f\colon X\to F$
whose support $\{x\in X\colon f(x)\not=e\}$ is finite.
We endow $F^{(X)}$ with the discrete topology.
\end{numba}
See \cite[Theorem~A]{SIM} for the following result.
\begin{thm}\label{classi}
If $(G,\alpha)$ is a simple totally disconneced contraction group,
then $G$ is either a torsion group or torsion free.
We have the following classification:
\begin{itemize}
\item[\rm(a)]
If $G$ is a torsion group,
then $(G,\alpha)$
is isomorphic to $F^{(-\N)}\times F^{\N_0}$
with the right shift, for some finite simple group~$F$.
\item[\rm(b)]
If $G$ is torsion free, then
$(G,\alpha)$ is isomorphic to
$(\Q_p)^d$ with a
$\Q_p$-linear contractive automorphism
for which there are no invariant vector subspaces,
for some prime number~$p$ and some $d\in\N$.
\end{itemize}
Conversely, all of these are simple contraction groups.
\end{thm}
To explain part (b) of the theorem,
let us recall some concepts and facts.
\begin{numba}\label{Zpmodule}
If $x$ is an element of a pro-$p$-group~$G$
(i.e., a projective limit of finite $p$-groups),
then $\Z\to G$, $n\mto x^n$ is a continuous homomorphism
with respect to the topology induced by~$\Z_p$ on~$\Z$,
and hence extends to a continuous homomorphism
\[
\phi\colon \Z_p\to G.
\]
As usual, we write $x^z:=\phi(z)$ for $z\in\Z_p$.
If $G$ is abelian and the group operation is written
additively, we write $zx:=\phi(z)$.
\end{numba}
\begin{numba}
Let $\bT:=\R/\Z$ with the quotient topology.
If $G$ is a locally compact abelian group,
we let $G^*:=\Hom_{\cts}(G,\bT)$
be its \emph{dual group}, endowed with the compact-open
topology; thus, the elements of $G^*$
are continuous homomorphisms $\xi\colon G\to\bT$
(see \cite{HaR}, \cite{HaM}, \cite{Str}).
We shall use the well-known fact that the dual group
$\Z(p^\infty)^*$ of the Pr\"{u}fer $p$-group
is isomorphic to~$\Z_p$ (compare, e.g., \cite[Exercise~23.2 and Theorem 22.6]{Str}).
\end{numba}
{\bf Some ideas of the proof of Theorem~\ref{classi}.}
As the closure $C$ of the commutator group $G'$ is $\alpha$-stable, closed and normal in~$G$,
we must have $C=\{e\}$ (in which case $G$ is abelian)
or $C=G$, in which case $G$ is topologically perfect.
If~$G$ is abelian, then either~$G$ is torsion free,
or $G$ is a torsion group of prime exponent~$p$:
In fact, if $G$ has a torsion element $g\not=e$,
then a suitable power $g^n$ is an element of order~$p$
for some prime number~$p$, entailing that the $p$-socle
$N:=\{x\in G\colon x^p=e\}$ is a non-trivial, $\alpha$-stable closed (normal) subgroup
of~$G$ and thus $N=G$.
As shown in~\cite{SIM}, $p$-adic Lie groups occur in
the case that $G$ is abelian and torsion free, which we assume now.
Like every totally disconnected contraction group, $G$~has a compact open
subgroup~$W$
such that $\alpha(W)\sub W$ (see \cite[3.1]{Sie}).
Then $W$ can be chosen as a pro-$p$-group for some~$p$.
In fact, there exists a $p$-Sylow subgroup
$P\not=\{e\}$ of~$W$ for some prime number~$p$,
which is unique as~$W$ is abelian (see \cite[Proposition~2.2.2 (a) and (d)]{Wls}).
Since every pro-$p$ subgroup of a pro-finite group is contained in a $p$-Sylow subgroup
(see \cite[Proposition~2.2.2\,(c)]{Wls}),
we deduce that $\alpha(P)\sub P$.
However, non-trivial $\alpha$-invariant closed normal subgroups
of the simple contraction group~$G$ must be open
(see \cite[Lemma~5.1]{SIM}). 
Thus~$P$ is open. Now replace $W$ with $P$ if necessary.]\\[2.3mm]
For $0\not=x\in W$, can define $z x$ for $z\in \Z_p$ by continuity (see \ref{Zpmodule}).
Let $W(x)$ be the image of the continuous homomorphism\vspace{-.5mm}
\[
\phi \colon \Z_p^{\N_0}\to W,\quad (z_n)_{n\in\N_0}\mto
\sum_{n=0}^\infty\alpha^n(z_nx).\vspace{-.5mm}
\]
Then $W(x)$ is a compact, non-trivial $\alpha$-invariant subgroup of~$G$
and hence open by \cite[Lemma~5.1]{SIM} just mentioned.
Being a torsion free abelian pro-$p$-group,
$W(x)$ is isomorphic to $\Z_p^J$ for some set~$J$.
This can be shown using Pontryagin duality: Since $W(x)$ is torsion free and a projective
limit of finite $p$-groups $F$, its dual group $W(x)^*$
is a divisible discrete group and a direct limit of the dual groups~$F^*\cong F$,
hence a $p$-group (see \cite[Corollary 23.10]{Str} as well as \cite[Proposition~7.5\,(i) and
(1)$\aeq$(2) in Corollary 8.5]{HaM}).
By the classification of the divisible abelian groups, $W(x)^*$
is isomorphic to a direct sum $\bigoplus_{j\in J}\Z(p^\infty)$
of Pr\"{u}fer $p$-groups (see \cite[Theorem (A.15)]{HaR}, cf.\ also \cite[Theorem A1.42]{HaM}).
As a consequence, $W(x)\cong W(x)^{**}$
is isomorphic to the direct product
\[
\prod_{j\in J} \Z(p^\infty)^* \cong (\Z_p)^J,
\]
as asserted (by \cite[Theorem 7.63]{HaM}
and \cite[Lemma 21.2 and Theorem 23.9]{Str}).\\[2.3mm]
Since $pW(x)=W(px)$ is a non-trival $\alpha$-invariant closed (normal)
subgroup of~$G$ and hence open,
$(p\Z_p)^J$ must be open in $\Z_p^J$,
whence $J$ is finite and $W(x)\cong \Z_p^J$
a $p$-adic Lie group.\\[2.3mm]
Now a linearization argument\footnote{Let $\beta:=L(\alpha)$.
Using the underlying additive topological group of~$L(G)$, the pair
$(L(G),\beta)$ is a $p$-adic contraction group such that $L(\beta)=L(\alpha)$.
Hence $(G,\alpha)\cong (L(G),\beta)$ by the last statement
of \cite[Proposition~5.1]{CON}.}
shows that
$(G,\alpha)\cong
(L(G),L(\alpha))$.\,\Punkt\vspace{2mm}

\noindent
The classification implies a structure theorem
for general totally disconnected contraction groups
$(G,\alpha)$ (see \cite[Theorem B]{SIM}):
\begin{thm}\label{contraB}
The set $\tor(G)$ of torsion elements
and the set $\dv(G)$ of divisible elements
are fully invariant closed subgroups of~$G$ and
\[
G \;=\; \tor(G) \,\times\,\dv(G)\,.
\]
Moreover, $\tor(G)$ has finite exponent and
\[
\dv(G)\; = \; G_{p_1}\, \times \, \cdots\,
\times \, G_{p_n}
\]
is a direct product of $\alpha$-stable $p$-adic
Lie groups~$G_p$ for certain primes~$p$.\,\Punkt
\end{thm}
\begin{rem}
By \cite[Theorem~3.5\,(iii)]{Wan},
each $G_p$ is nilpotent, and
it
is in fact the
group of $\Q_p$-rational points of a unipotent
linear algebraic group
defined over~$\Q_p$.
See \cite{CLO} for algebraic properties
of $\conp(\alpha)$ if $\alpha$ is an analytic \emph{endomorphism}
of a Lie group~$G$ over a totally disconnected local field;
if~$\alpha$ is an analytic automorphism, then
$\conp(\alpha)$ is nilpotent (cf.\ Remark~\ref{rem-nonclo} and~\cite{CON}).
\end{rem}
\begin{numba}\label{finitnss}
If $(G,\alpha)$ is a totally disconnected
contraction group with $G\not=\{e\}$,
then~$G$ has a compact open subgroup $U$ such that $\alpha(U)$
is a proper subgroup of~$U$ (cf.\ \cite[3.1]{Sie}),
whence $U$ is a proper subgroup of $\alpha^{-1}(U)$ and thus
\[
\Delta(\alpha^{-1})=[\alpha^{-1}(U):U]\in\{n\in\N\colon n\geq 2\}
\]
(see \cite[Lemma 3.2\,(i)]{Sie} and \cite[Proposition~1.1\,(e)]{SIM}).
If
\begin{equation}\label{aseries}
\{e\}=G_0\lhd  G_1\lhd \cdots\lhd G_n=G
\end{equation}
is a properly
ascending series of $\alpha$-stable closed subgroups of~$G$
and $\alpha_j$ the contractive automorphism of $G_j/G_{j-1}$
induced by~$\alpha$ for $j\in\{1,\ldots, n\}$, then
\[
\Delta(\alpha^{-1})=\Delta(\alpha_1^{-1})\cdots\Delta(\alpha_n^{-1}),
\]
showing that $n$ is bounded by the number of prime factors of $\Delta(\alpha^{-1})$,
counted with multiplicities (see \cite[Lemma~3.5]{SIM}).
As a consequence, we can choose a properly ascending series~(\ref{aseries})
of maximum length. Then all of the subquotients $(G_j/G_{j-1},\alpha_j)$
are simple contraction groups.
To deduce Theorem~\ref{contraB}
from Theorem~\ref{classi},
one shows that the
series can always be chosen in such a way that the torsion factors
appear at the bottom, whence $\tor(\alpha)=G_k$ for some~$k$.
A major step then is to see that $G_k$ is complemented in~$G$,
and that $G/G_k$ is a product of $p$-adic Lie groups (see~\cite{SIM}).
\end{numba}
\section{Expansive automorphisms}\label{sec-exp}
If $\alpha$ is an expansive automorphism
of a totally dsconnected locally compact group~$G$,
then the subset
\[
\conp(\alpha)\conm(\alpha)
\]
of~$G$ is an open identity neighbourhood (see \cite[Lemma~1.1\,(d)]{GaR}).
In some cases, this enables the finiteness properties of totally disconnected
contraction groups (as described in \ref{finitnss}) to be used
with profit also for the study of expansive automorphsms.
The proof of expansiveness of $\wb{\alpha}$ in
part~(b) of the following result from \cite{GaR}
is an example for this strategy.
\begin{prop}
Let $\alpha$ be an automorphism of a totally disconnected, locally compact group~$G$.
\begin{itemize}
\item[\rm(a)]
If $\alpha$ is expansive, then $\alpha|_H$ is expansive
for each $\alpha$-stable closed subgroup $H\sub G$.
\item[\rm(b)]
Let $N\sub G$ be an $\alpha$-stable closed
normal subgroup and $\wb{\alpha}$ be the induced automorphism
of $G/N$ which takes $gN$ to $\alpha(g)N$.
Then $\alpha$ is expansive if and only if $\alpha|_N$ and $\wb{\alpha}$
are expansive.\,\Punkt
\end{itemize}
\end{prop}
If a $p$-adic Lie group~$G$ admits
an expansive automorphism~$\alpha$, then $L(\alpha)$
is a Lie algebra automorphism of~$L(G)$
such that $|\lambda|\not=1$
for all eigenvalues~$\lambda$
of $L(\alpha)$ in an algebraic closure~$\wb{\K}$
(as recalled in Proposition~\ref{charviaeigen}),
entailing that none of the $\lambda$ is a root of unity.
Hence $L(G)$ is nilpotent
(see Exercise~21\,(b) among the exercises
for Part~I of \cite{Bou}, \S4, or
\cite[Theorem 2]{Jcb}).
If, moreover, $G$ is linear in the sense
that it admits a faithful continuous representation
$G\to\GL_n(\Q_p)$ for some $n\in\N$,
then $G$ has an $\alpha$-stable, nilpotent open
subgroup \cite[Theorem~D]{GaR}.
For closed subgroup of~$\GL_n(\Q_p)$,
such an open subgroup can be made explicit
(see \cite[Proposition~7.8]{GaR}):
\begin{prop}\label{lingroup}
Let $\alpha$ be an expansive automorphism
of a $p$-adic Lie group~$G$. If $G$ is isomorphic to a closed
subgroup of~$\GL_n(\Q_p)$ for some $n\in\N$,
then $\conp(\alpha)\conm(\alpha)$ is a nilpotent, open subgroup of~$G$.\,\Punkt
\end{prop}
The following example is taken from
\cite[Remark~7.7]{GaR}.
\begin{example}\label{non-lin2}
Let $H = \Q_p ^3$ be the $3$-dimensional $p$-adic Heisenberg group
with group multiplication given by
\[
(x _1, y_1,z _1) (x_2, y_2, z_2) := (x_1+x_2 , y_1 + y _2 ,
z_1+z_2+x_1 y_2)
\]
for all $(x _1, y_1,z _1), (x_2, y_2, z_2) \in H$.
Then $N= \{ (0,0,
z ) \in H \colon |z| \leq 1 \}$ is a compact central
subgroup of $H$. Identify $G= H/N$ with $\Q_p\times\Q_p\times
(\Q_p/\Z_p)$ as a set.  Define $\alpha \colon G \to G$ by
\[
\alpha (x,y, z+\Z_p) = (px , p^{-1}y, z+\Z_p)
\]
for all $(x,y, z+\Z_p)\in G$.
Then $\alpha$ is a continuous automorphism of the $p$-adic Lie group
$G$ with
$\lev(\alpha) = \{ (0,0, z+\Z_p) \colon  z\in \Q _p \}$,
\[
\conp(\alpha) =
\{ (x ,0 , 0) \colon  x\in \Q _p \},\quad
\mbox{and}\quad \conm(\alpha) = \{(0,y ,
0) \colon y \in \Q _p \}.
\]
Since $\lev(\alpha)$ is discrete, $\alpha$ is
an expansive automorphism
(see \cite[Propostion~1.3\,(a)]{GaR}).
As
\[
[\conp(\alpha) , \,\conm(\alpha)] = \{
(0,0, z+\Z_p)\colon z \in \Q_p\}
\]
and $\conp(\alpha) \conm(\alpha) =
\{(x,y, xy+\Z_p) \colon x, y \in \Q _p \}$, we find that $\conp(\alpha)
\conm(\alpha)$ is a not a subgroup of~$G$.
Accordingly, $G$ is not isomorphic to a closed subgroup of $\GL_n(\Q_p)$
for any $n\in\N$ (see Proposition~\ref{lingroup}).
\end{example}
\section{Distality and Lie groups of type {\boldmath$R$}}\label{sec-typeR}
Following Palmer~\cite{Pal},
a totally disconnected,
locally compact group~$G$
is called \emph{uniscalar}
if $s_G(x)=1$ for each $x\in G$.
This holds if and only if each group element
$x\in G$ normalizes some compact, open subgroup~$V_x$ of~$G$
(which may depend on~$x$).
It is natural to ask
whether this condition implies
that $V_x$ can be chosen independently
of $x$, i.e., whether $G$ has a compact, open,
\emph{normal} subgroup.
The answer is negative for a
suitable $p$-adic Lie group which is not compactly generated
(see \cite[\S6]{UNI}).
But also for some totally disconnected, locally
compact groups which are compactly generated,
the answer is negative
(see \cite{BaM} together with \cite{KaW}, or also
\cite[Proposition 11.4]{Lec}, where moreover all
contraction groups for inner automorphisms are trivial
and hence closed);
the counterexamples are of the form
\[
(F^{(-\N)}\times F^{\N_0})\rtimes H
\]
with $F$ a finite simple group and a suitable action of
a specific finitely generated group~$H$ on~$\Z$.
Thus, to have a chance for a positive answer,
one has to restrict attention
to particular classes of groups
(like compactly generated $p$-adic Lie groups).
If $G$ has the (even stronger)
property that every identity neighbourhood
contains an open, compact, normal subgroup
of~$G$, then $G$ is called \emph{pro-discrete}.\footnote{Another interesting group
is the semidirect product $G:=F^T\rtimes T$,
where~$F$ is a finite simple group and
$T$ is a Tarski monster (a certain finitely generated, infinite, simple
torsion group) acting on $F^T$ via $(x.f)(y):=f(x^{-1}y)$
for $x,y\in T$. Then, for each $x\in G$, there is a basis of identity neighbourhoods
consisting of compact open subgroups of~$G$ which are normalized by~$x$.
Moreover, $G$ has $F^T$ as a compact open normal
subgroup, but this is the only such and thus~$G$ is not pro-discrete (see \cite{DIR}).}
Finally, a Lie group~$G$ over a local field~$\K$
is \emph{of type~$R$}
if all eigenvalues~$\lambda$ of $L(\alpha)$ in an algebraic closure $\wb{\K}$
have absolute value $|\lambda|_\K=1$,
for each inner automorphism~$\alpha$
(cf.\ \cite{Raj} for $\K=\Q_p$), i.e., if each inner automorphism is distal
(see Proposition~\ref{charviaeigen}).\\[2.3mm]
\noindent
Using the Inverse Function Theorem with Parameters and locally invariant manifolds
as a tool, we can generalize results for $p$-adic
Lie groups from~\cite{Raj}, \cite{UNI}, and \cite{Par}
to Lie groups over local fields of arbitrary characteristic.
The following result was announced in \cite[Proposition~11.2]{Lec}.
\begin{prop}\label{when-typeR}
Let $\alpha$ be an analytic automorphism of a Lie group~$G$
over a totally disconnected local field~$\K$.
Then the following properties
are equivalent:
\begin{itemize}
\item[\rm(a)]
$\conp(\alpha)$ is closed
and $s(\alpha)=s(\alpha^{-1})=1$;
\item[\rm(b)]
All eigenvalues of $L(\alpha)$
in~$\wb{\K}$ have absolute value~$1$;
\item[\rm(c)]
Each $e$-neighbourhood
in~$G$ contains an $\alpha$-stable compact open subgroup.
\end{itemize}
In particular, $G$ is of type~$R$
if and only if $G$ is uniscalar
and $\conp(\alpha)$ is closed
for each inner automorphism~$\alpha$ of~$G$
$($in which case $\conp(\alpha)=\{e\})$.
\end{prop}
\begin{proof}
The implication ``(a)$\impl$(b)'' follows from
Theorem~\ref{main} and Theorem~\ref{thm-lincase}\,(c).

If (b) holds, then $\cg=\cg_1$
coincides with the centre subspace $\cg_1=\lev(L(\alpha))$ with respect to~$L(\alpha)$,
whence $\phi_c\colon W_c\to B^{\cg_1}_R(0)\sub \cg_1=\cg$ (as in \ref{c-stable})
is a diffeomorphism with $\phi_c(e)=0$
and $d\phi_c|_\cg=\id_\cg$. After shrinking~$R$ if necessary, we may assume that the sets
$W^c(t)=\phi_c^{-1}(B^\cg_t(0))$,
which are $\alpha$-stable by~\ref{c-stable},
are compact open subgroups of~$G$
for all $t\in\;]0,R]$ (see Lemma~\ref{laballs}).
Thus~(b) implies~(c).

If~(c) holds, then every identity neighbourhood
of~$G$ contains a compact open subgroup~$V$ which is $\alpha$-stable
and hence tidy for~$\alpha$ with
\[
s(\alpha)=[\alpha(V):\alpha(V)\cap V]=[V:V]=1
\]
and, likewise, $s(\alpha^{-1})=1$.
Moreover, $\conp(\alpha)$ is closed (by Lemma~\ref{tidyclo}\,(a)), and thus~(a) follows.
\end{proof}
As shown in~\cite{Par}
and \cite{UNI}, every
compactly generated,
uniscalar $p$-adic Lie group
is pro-discrete. For Lie groups over totally disconnected local fields,
we have the following analogue
(announced in \cite[Proposition~11.3]{Lec}):
\begin{thm}\label{typeR}
Every compactly generated
Lie group of type~$R$ over
a totally disconnected local field is pro-discrete.
\end{thm}
\begin{proof}
Let $G$ be a Lie group
over a totally disconnected local field~$\K$, with Lie algebra $\cg:=L(G)$,
such that~$G$ is generated by a compact subset~$K$.
Then $\Ad(G)\sub\GL(\cg)$ is
generated by the compact set~$\Ad(K)$, since the
adjoint representation $\Ad\colon G\to\GL(\cg)$ is
a continuous homomorphism (as recalled in~\ref{Ad-ana}).
Moreover, the subgroup generated by $\Ad_x$
is relatively compact in $\GL(\cg)$ for each $x\in G$,
by Lemma~\ref{la-periodic}.
Thus $\Ad(G)$ is relatively compact in $\GL(\cg)$,
by \cite[Th\'{e}or\`{e}me 1]{Par}.
Let~$\bO$ be the valuation ring of~$\K$.
As a consequence of Theorem~1 in
Appendix~1 of \cite[Chapter IV]{Ser},
there is a compact open $\bO$-submodule $M\sub \cg$
with $\Ad_x(M)=M$ for all $x\in G$.
Let
\[
\|.\|\colon \cg\to[0,\infty[,\;\;
x\mto \inf\{|z|\colon \mbox{$z\in \K$ such that $x\in zM$}\}
\]
be the Minkowski functional of~$M$.
Then~$\|.\|$ is a norm on~$\cg$ such that $\Ad(G)\sub\Iso(\cg,\|.\|)$.
Using this norm, we abbreviate $B^\cg_t:=B^\cg_t(0)$ for all $t>0$.
Let $\phi\colon U\to V$ be an analytic diffeomorphism from
a compact open subgroup~$U$ of~$G$
onto an open $0$-neighbourhood $V\sub\cg$ such that
$\phi(e)=0$ and $d\phi|_\cg=\id_\cg$. Let~$R>0$ and the
compact open subgroups $B^\phi_t=\phi^{-1}(B^\cg_t)$ for $t\in\;]0,R]$
be as in~\ref{pre-laballs}.
Since $K$ is compact and the mapping $G\times G\to G$, $(x,y)\mto xyx^{-1}$
is continuous with $xex^{-1}=e$, we find an open identity neighbourhood
$W\sub U$ such that $xWx^{-1}\sub U$ for all $x\in K$.
Now consider the analytic mapping
\[
f\colon G\times \phi(W)\to V,\quad f(x,y):=\phi(x\phi^{-1}(y)x^{-1})
\]
and define $f_x:=f(x,.)\colon \phi(W)\to V$ for $x\in G$.
Then $f_x(0)=0$ and $(f_x)'(0)=\Ad_x$ is an isometry for all $x\in G$.
As a consequence of Lemma~\ref{invpar}, for each $x\in K$ there exists
an open neighbourhood $P_x$ of~$x$ in~$G$ and $r_x>0$ with $B^\cg_{r_x}\sub \phi(W)$
such that
\[
f_z(B^\cg_t)=B^\cg_t\quad\mbox{for all $z\in P_x$ and $t\in\;]0,r_x]$.}
\]
We may assume that $r_x\leq R$ for all $x\in K$.
There is finite subset $\Phi\sub K$ such that $K\sub\bigcup_{x\in\Phi}P_x$.
Set
\[
r:=\min\{r_x\colon x\in\Phi\}.
\]
Then $f_z(B^\cg_t)=B^\cg_t$ for all $z\in K$ and $t\in\;]0,r]$,
entailing that
\[
z B^\phi_t z^{-1}=
\phi^{-1}(\phi(z\phi^{-1}(B^\cg_t)z^{-1}))=\phi^{-1}(f_z(B^\cg_t))=\phi^{-1}(B^\cg_t)=B^\phi_t.
\]
Since~$K$ generates~$G$, we deduce that the compact open subgroup~$B^\phi_t$
is normal in~$G$, for each $t\in\;]0,r]$.
\end{proof}
Related problems were also studied in~\cite{Re3}.
{\bf Helge  Gl\"{o}ckner}, Universit\"at Paderborn, Institut f\"{u}r Mathematik,\\
Warburger Str.\ 100, 33098 Paderborn, Germany;\\[1mm]
e-mail: {\tt  glockner\at{}math.upb.de}\vfill
\end{document}